\documentclass[reqno, a4paper, 11pt]{amsart}
\usepackage{graphicx} 

\usepackage{amsfonts}
\usepackage[dvipsnames]{xcolor}
\usepackage{amssymb}
\usepackage{amsthm}
\usepackage{subfig}
\usepackage{graphicx}
\usepackage{latexsym,amsmath}
\usepackage{afterpage}
\usepackage[T1]{fontenc}
\usepackage{epigraph}
\usepackage[shortlabels]{enumitem}
\usepackage{geometry}
\usepackage{orcidlink}
\usepackage{hyperref}
\hypersetup{colorlinks=true, citecolor=green, linkcolor=red}

\newcommand{\R}{\mathbb{R}}

\newcommand{\N}{\mathbb{N}}

\newcommand{\V}{\mathbb{V}}
\newcommand{\W}{\mathbb{W}}

\newcommand{\HB}{\mathbb{H}}
\newcommand{\LB}{\mathbb{L}}

\theoremstyle{plain}
\newtheorem{theorem}{Theorem}[section]
\newtheorem{lemma}[theorem]{Lemma}

\theoremstyle{definition}
\newtheorem{definition}[theorem]{Definition}

\newtheorem{algorithm}[theorem]{Algorithm}

\theoremstyle{remark}
\newtheorem{remark}[theorem]{Remark}

\numberwithin{equation}{section}

\newcommand{\lnorm}[2]{\left\|#1\right\|_{\LB^{#2}(\Omega)}}
\newcommand{\hnorm}[2]{\left\|#1\right\|_{\HB^{#2}(\Omega)}}
\newcommand{\lnorms}[2]{\left\|#1\right\|_{L^{#2}(\Omega)}}
\newcommand{\hnorms}[2]{\left\|#1\right\|_{H^{#2}(\Omega)}}
\newcommand{\wnorm}[3]{\left\|#1\right\|_{\W^{#2,#3}(\Omega)}}

\newcommand{\liprod}[2]{\left\langle#1,#2\right\rangle_{\LB^2(\Omega)}}
\newcommand{\liprods}[2]{\left\langle#1,#2\right\rangle_{L^2(\Omega)}}

\renewcommand{\vec}[1]{\boldsymbol{#1}}
\newcommand{\femvec}[3]{\vec{#1}_{#2}^{#3}}

\let\div\undefined
\DeclareMathOperator{\div}{div}
\DeclareMathOperator{\curl}{curl}

\newcommand{\pmat}[1]{\begin{pmatrix}#1\end{pmatrix}}

\newcommand{\lsim}{\lesssim}

\newcounter{Ax}

\allowdisplaybreaks

\newcommand{\comm}[1]{\textcolor{blue}{#1}}

\usepackage[normalem]{ulem}
\newcommand\tsout{\bgroup\markoverwith{\textcolor{red}{\rule[0.5ex]{2pt}{1.4pt}}}\ULon}
\newcommand{\stkout}[1]{\ifmmode\text{\tsout{\ensuremath{#1}}}\else\tsout{#1}\fi}

\newcommand{\stokesVelocityProj}{\Pi_{\mathrm{S}}}

\newcommand{\maxwellProj}{\Pi_{\mathrm{M}}}
\newcommand{\ellipticProj}{\Pi_{\mathrm{R}}}
\newcommand{\Lproj}{\Pi_\mathrm{L}}

\newcommand{\discreteDtau}{D_{\tau}}

\title[Error Estimates for a Linear Fully-Discrete FEM for the FMHD Model]{Error Estimates for a Linear Fully-Discrete Finite Element Method for the Ferromagnetic Magnetohydrodynamical Model}
\author{{Noah Vinod\orcidlink{0009-0008-4388-2329} and Thanh Tran\orcidlink{0000-0001-6117-4811}}}
\date{\today}
\address{School of Mathematics and Statistics, The University of New South Wales, Sydney 2052, Australia}
\email{\comm{n.vinod@unsw.edu.au}}
\address{School of Mathematics and Statistics, The University of New South Wales, Sydney 2052, Australia}
\email{\comm{thanh.tran@unsw.edu.au}}

\begin{document}

\begin{abstract}
The ferromagnetic magnetohydrodynamic equations are a non-linear PDE system that models the flow of an electrically conducting fluid that also has intrinsic magnetisation. In this paper, we propose a fully-discrete linear finite element scheme based on Euler's method to approximate the solutions to these equations and perform an error analysis for the scheme. Numerical experiments have been included to corroborate our theoretical results.
\end{abstract}

\maketitle

\tableofcontents

\section{Introduction}

In this paper, we estimate the error of a numerical scheme for the ferromagnetic magnetohydrodynamical model represented by the following system of PDEs
\begin{subequations} \label{eqn:fmhd}
    \begin{alignat}{2}
    &\partial_t \vec{v} + (\vec{v} \cdot \nabla) \vec{v} - \mu\Delta \vec{v} + \nabla p = \curl \vec{B} \times \vec{B} - (\vec{B} \cdot \nabla) \vec{m} + \nabla (\vec{m} \cdot \vec{B}) - \div \big(\nabla \vec{m} \odot \nabla \vec{m}\big), \label{eqn:fmhd-equations a} \\
    &\div \vec{v} = 0, \label{eqn:fmhd-equations b} \\
    &\partial_t \vec{B} + \eta \curl^2 \vec{B} - \curl(\vec{v} \times \vec{B}) = \vec{0}, \label{eqn:fmhd-equations c} \\
    &\partial_t \vec{m} + (\vec{v} \cdot \nabla) \vec{m} = \gamma \vec{m} \times (\Delta \vec{m} + \vec{B}) - \chi \vec{m} \times (\vec{m} \times (\Delta \vec{m} + \vec{B})), \label{eqn:fmhd-equations e} \\
    &|\vec{m}| = 1, \label{eqn:fmhd-equations f}
\end{alignat}
\end{subequations}
over the problem domain $(0,T) \times \Omega$, where $T > 0$, $\Omega$ is a convex polyhedral domain in $\R^3$ with a Lipschitz boundary $\partial\Omega$ and exterior unit normal vector $\vec{n}$. Here $P \odot Q = P^{\top} Q$ when $P$ and $Q$ are two matrices of appropriate sizes such that the matrix multiplication is well-defined. The vector fields $\vec{v}, \vec{B}, \vec{m} : [0,T] \times \Omega \to \R^3$ represent the velocity of the fluid, the magnetic field, and the magnetisation, respectively, while the scalar field $p : [0,T] \times \Omega \to \R$ represents the pressure of the fluid. The constants $\mu, \eta$ and $\chi$ are positive parameters related to the viscosity the fluid, the diffusivity of the magnetic field, and the damping of the magnetisation. The parameter $\gamma$ is a non-zero constant representing the electron's gyromagnetic ratio. This system is subject to the boundary conditions
\begin{subequations} \label{eqn:boundary-conditions}
    \begin{alignat}{3}
    \vec{v} &= \vec{0}, \label{eqn:boundary-conditions a}\\
    \vec{B} \cdot \vec{n} &= 0, \label{eqn:boundary-conditions b}\\
    \curl \vec{B} \times \vec{n} &= \vec{0}, \label{eqn:boundary-conditions c}\\
    \frac{\partial \vec{m}}{\partial \vec{n}} &= \vec{0}, \label{eqn:boundary-conditions d}
    \end{alignat}
\end{subequations}
on $[0,T] \times \partial\Omega$, along with the initial data
\begin{equation*}
    \vec{v}(0,\cdot) = \vec{v}_0, \qquad \vec{B}(0,\cdot) = \vec{B}_0, \qquad \vec{m}(0,\cdot) = \vec{m}_0.
\end{equation*}

The ferromagnetic magnetohydrodynamical model, first proposed by Lingam \cite{lingam2015-article}, is an MHD model that captures the dynamics of a magnetohydrodynamic fluid that is also magnetised. These fluids are electrically conducting, like the usual MHD fluid, but are also internally magnetised in the way permanent magnets are. Lingam accounted for the usual MHD dynamics that come from its electrical conductivity but also for the dynamics that come from its internal magnetisation. This model has potential applications in astrophysical and fusion plasmas, and most importantly in magnetic liquid metals \cite{kim2024-article, xiang2024-article}, which bears the exact properties of the theoretical ferrofluid that Lingam considers. The PDE model that this paper is based on considers the incompressible variant of Lingam's compressible model. The model also takes into account the exchange energy of the magnetisation. This modified model is introduced in \cite{vinod2026-article}, where the well-posedness of local-in-time strong solutions for that system is proved. Details of the derivation of \eqref{eqn:fmhd} can be found there as well.

To the best of our knowledge, no numerical method has been developed for the model \eqref{eqn:fmhd}, and we seek to undertake that enterprise in this paper. We recognise that \eqref{eqn:fmhd} is a combination of the MHD equations in conjunction with the Landau-Lifshitz equation. These two systems have been thoroughly studied separately in the past few decades from both a theoretical and numerical perspective. The literature is vast for the numerical study of the MHD equations and the Landau-Lifshitz equation, and a non-exhaustive list of results for both systems includes \cite{banas2010-article, gao2018-article, gao2023-article,goldys2026-arxiv, ravindran2016-article, wang2022-article, zhang2015-article} and \cite{alouges2008-article, an2021-article, bartels2008-article, bartels2006vol44-article, banas2008-article, le2013-article}, respectively.

We develop error estimates for our numerical scheme which is based on the finite element method and consists of solving a linear system at each time-step. In this numerical scheme, the finite element velocity $\vec{v}_h$ satisfies $\div \vec{v}_h = 0$ in the weak sense and the sphere constraint $|\vec{m}| = 1$ for the Landau-Lifshitz equation is achieved asymptotically.

The paper is organised as follows. In Section \ref{sec:preliminaries} we formulate an equivalent FMHD model \eqref{eqn:fmhd-equiv}, introduce our numerical method, state the main result, and present some useful results relevant to our analysis. The error equations are derived in Section \ref{sec:error-equations}, with further details of related estimates given in the appendix. These results are finally used in Section \ref{sec:theorem-proof} to prove Theorem \ref{thm:main-result}. Section \ref{sec:numerical-experiments} is dedicated to the discussion of a few numerical experiments that were conducted to confirm our theoretical results.


\section{Preliminaries and main results} \label{sec:preliminaries}

We first define the notation used in this paper. The set $\Omega \subset \R^3$ is a bounded convex polyhedral domain with a Lipschitz boundary. The space $\LB^p(\Omega)$ denotes the space of $p$-integrable functions on $\Omega$ taking values in $\R^d$, where the value of $n$ is determined from context. Throughout this paper, we use $\langle \cdot, \cdot\rangle$ to denote the $\LB^2(\Omega)$ inner product. Moreover, $\W^{m,p}(\Omega)$ denotes the usual Sobolev space of vector-valued functions, and $\HB^m(\Omega) := \W^{m,2}(\Omega)$. Furthermore, we define
\begin{align*}
    \HB_n^1(\Omega) &:= \{\vec{u} \in \HB^1(\Omega) : \vec{u} \cdot \vec{n} = 0 \text{ a.e. on } \partial\Omega\}, \\
    \HB_{n,\curl}^1(\Omega) &:= \{\vec{u} \in \HB^1(\Omega) : \vec{u} \cdot \vec{n} = 0 \text{ and } \curl \vec{u} \times \vec{n} = \vec{0} \text{ a.e. on } \partial\Omega\}
\end{align*}
Lastly, the spaces $L^p(0,T; X)$ and $W^{m,p}(0,T; X)$ denote the usual Bochner spaces of functions on $(0,T)$ taking values in a normed vector space $X$. Here we use the notation $\partial_t$ to refer to the time-derivative and $\partial_i$ as shorthand for $\partial_{x_i}$ to refer to the spatial derivatives; we also use $\partial_{\vec{n}}$ as shorthand for the directional derivative (see below). Moreover, for any vector-valued function $\vec{u} : \Omega \to \R^3$ we define
\begin{alignat*}{3}
    &\nabla \vec{u} : \Omega \to \R^{3 \times 3} & & \qquad \text{by} \qquad  \nabla \vec{u} := \pmat{\partial_1 \vec{u} & \partial_2 \vec{u} & \partial_3 \vec{u}} = \pmat{\partial_1 u_1 & \partial_2 u_1 & \partial_3 u_1 \\ \partial_1 u_2 & \partial_2 u_2 & \partial_3 u_2 \\ \partial_1 u_3 & \partial_2 u_3 & \partial_3 u_3}, \\[2ex]
    &\Delta \vec{u} : \Omega \to \R^3 & &\qquad \text{by} \qquad \Delta \vec{u} := \pmat{\Delta u_1 & \Delta u_2 & \Delta u_3}^{\top}, \\[2ex]
    &\displaystyle \frac{\partial \vec{u}}{\partial \vec{n}} : \partial\Omega \to \R^3 & &\qquad \text{by} \qquad \displaystyle\frac{\partial \vec{u}}{\partial \vec{n}} := (\nabla \vec{u}|_{\partial\Omega}) \vec{n},
\end{alignat*}
Throughout this paper we use the constant $C$ to denote an arbitrary positive constant.

\subsection{Weak formulation}
Without loss of generality, we set $\mu = \eta = \gamma = \chi = 1$ in \eqref{eqn:fmhd}. Using the elementary identities
\begin{align}
    \vec{u} \times (\vec{v} \times \vec{w}) &= (\vec{u} \cdot \vec{w}) \vec{v} - (\vec{u} \cdot \vec{v}) \vec{w}, \label{eqn:triple-cross-product} \\
    \Delta (|\vec{u}|^2) &= 2|\nabla \vec{u}|^2 + 2(\vec{u} \cdot \Delta \vec{u}), \\ \label{eqn:laplace-magnitude}
    \div \big( \nabla \vec{u} \odot \nabla \vec{u} \big) &= \nabla \vec{u} \odot \Delta \vec{u} + \frac{1}{2} \nabla |\nabla \vec{u}|^2,
\end{align}
we obtain the following problem equivalent to~\eqref{eqn:fmhd}
\begin{subequations} \label{eqn:fmhd-equiv}
    \begin{alignat}{2}
    &\partial_t \vec{v} + (\vec{v} \cdot \nabla) \vec{v} -  \Delta \vec{v} + \nabla \tilde{p} = \curl \vec{B} \times \vec{B} - (\vec{B} \cdot \nabla) \vec{m} - \nabla \vec{m} \odot \Delta \vec{m}, \label{eq:fmhd-equiv a} \\
    &\div \vec{v} = 0, \label{eq:fmhd-equiv b} \\
    &\partial_t \vec{B} +  \curl^2 \vec{B}  - \curl(\vec{v} \times \vec{B}) = \vec{0}, \label{eq:fmhd-equiv c} \\
    &\partial_t \vec{m} + (\vec{v} \cdot \nabla) \vec{m} -  \Delta \vec{m} =  \vec{m} \times (\Delta \vec{m} + \vec{B}) +  |\nabla \vec{m}|^2 \vec{m} -  \vec{m} \times (\vec{m} \times \vec{B}), \label{eq:fmhd-equiv e} \\
    &|\vec{m}| = 1, \label{eq:fmhd-equiv f}
    \end{alignat}
\end{subequations}
where $\tilde{p}$ is the modified pressure defined by
\begin{equation} \label{eqn:p-tilde}
    \tilde{p} := p - \vec{m} \cdot \vec{B} + \frac{1}{2} |\nabla \vec{m}|^2.
\end{equation}

Formally, the weak formulation in Definition \ref{def:weak-solution} is straightforward to obtain. We multiply equations \eqref{eq:fmhd-equiv a}, \eqref{eq:fmhd-equiv b}, \eqref{eq:fmhd-equiv c} and \eqref{eq:fmhd-equiv e} by the test functions $\vec{\varphi}, q, \vec{\omega}$ and $\vec{\xi}$, respectively, and then integrate by parts. The regularity given in the definition renders all terms well-defined. Additionally, note that we also make use of the integration by parts formula \cite[Theorem 3.29]{monk2003-book}
\begin{equation*}
    \liprod{\curl \vec{u}}{\vec{v}} = \liprod{\vec{u}}{\curl \vec{v}} + \langle \vec{u} \times \vec{n}, \vec{v} \rangle_{\LB^2(\partial\Omega)},
\end{equation*}
and Lemma~\ref{lem:convective} in Section \ref{sec:useful-results}.



\begin{definition}[Weak Solution] \label{def:weak-solution}
Given $T > 0$ and $(\vec{v}_0, \vec{B}_0, \vec{m}_0) \in \HB_0^1(\Omega) \times \HB_n^1(\Omega) \times \HB^2(\Omega)$, the functions $(\vec{v}, \tilde{p}, \vec{B}, \vec{m})$ are said to be a weak solution to \eqref{eqn:fmhd-equiv} if
\begin{enumerate}[i)]
    \item The velocity $\vec{v}$, pressure $\tilde{p}$, magnetic field $\vec{B}$ and magnetisation $\vec{m}$ have the regularity
    \begin{align*}
        \vec{v} &\in H^1(0,T; \HB_0^1(\Omega)), \\
        \tilde{p} &\in L^{\infty}(0,T; \LB^2(\Omega)), \\
        \vec{B} &\in H^1(0,T; \HB_n^1(\Omega)), \\
        \vec{m} &\in H^1(0,T; \HB^2(\Omega)).
    \end{align*}

    \item The magnetisation $\vec{m}$ has magnitude
    \begin{equation*}
        |\vec{m}| = 1 \qquad \text{a.e. on } [0,T] \times \Omega.
    \end{equation*}

    \item The quadruplet $(\vec{v}, \tilde{p}, \vec{B}, \vec{m})$ satisfies 
    \begin{align*}
        &\liprod{\partial_t \vec{v}}{\boldsymbol{\varphi}}  +  \liprod{\nabla \vec{v}}{\nabla \boldsymbol{\varphi}} + \frac{1}{2}\liprod{(\vec{v} \cdot \nabla) \vec{v}}{\boldsymbol{\varphi}} - \frac{1}{2} \liprod{(\vec{v} \cdot \nabla) \boldsymbol{\varphi}}{\vec{v}} - \liprod{\curl \vec{B} \times \vec{B}}{\boldsymbol{\varphi}}  \\
        &\qquad + \liprod{(\vec{B} \cdot \nabla) \vec{m}}{\boldsymbol{\varphi}} + \liprod{\nabla \vec{m} \odot \Delta \vec{m}}{\vec{\varphi}} - \liprod{\tilde{p}}{\div \boldsymbol{\varphi}} = 0, \\[2ex]
        &\liprods{\div \vec{v}}{q} = 0, \\[2ex]
        &\liprod{\partial_t \vec{B}}{\boldsymbol{\omega}} +  \liprod{\curl \vec{B}}{\curl \boldsymbol{\omega}}  - \liprod{\vec{v} \times \vec{B}}{\curl \boldsymbol{\omega}} = 0, \\[2ex]
        &\liprod{\partial_t \vec{m}}{\boldsymbol{\xi}} +  \liprod{\nabla \vec{m}}{\nabla \boldsymbol{\xi}} + \liprod{(\vec{v} \cdot \nabla) \vec{m}}{\boldsymbol{\varphi}} -  \liprod{|\nabla \vec{m}|^2 \vec{m}}{\boldsymbol{\xi}} \\
        &\qquad -  \liprod{\vec{m} \times \Delta \vec{m}}{\boldsymbol{\xi}} -  \liprod{\vec{m} \times \vec{B}}{\boldsymbol{\xi}} -  \liprod{\vec{m} \times (\vec{m} \times \vec{B})}{\boldsymbol{\xi}} = 0,
    \end{align*}
    for all $(\vec{\varphi}, q, \vec{\omega}, \vec{\xi}) \in \HB^1_0(\Omega) \times \LB^2(\Omega) \times \HB_n^1(\Omega) \times \HB^1(\Omega)$ and $t \in [0,T]$.
\end{enumerate}
\end{definition}

\subsection{Finite element spaces and discrete operators} \label{sec:fem-spaces}

Let $\mathcal{T}_h = \{K_j\}_{j=1}^M$ be a globally quasi-uniform tetrahedral partition of $\Omega$ with mesh size $h = \max \{\mathrm{diam}(K_j)\}_{j=1}^M$, where $K_j$ represents a tetrahedral element. For any $m\in\N$ and $h>0$, we define the finite element space $\V_h^m \subset \HB^1(\Omega)$ to be the space
\begin{equation} \label{eqn:finite-element-space}
    \V_h^m(\Omega) = \{\vec{u}_h \in C(\overline{\Omega}) : \vec{u}_h|_{K_j} \in P_m(K_j), \text{ for all } K_j \in \mathcal{T}_h\},
\end{equation}
where $P_m(K_j)$ represents the set of polynomials of degree less than or equal to $m$ restricted to the element $K_j \in \mathcal{T}_h$. Let $\ell, k, r \in \N$ denote the polynomial degrees of the velocity, magnetic field, and magnetisation finite element spaces, respectively, and define
\begin{align*}
    L_0^2(\Omega) = \left\{u \in L^2(\Omega) : \int_{\Omega} u \; d\vec{x} = 0\right\}.
\end{align*}
Hence, we define the following finite element spaces for the velocity, pressure, magnetic field and magnetisation, respectively:
\begin{align*}
    \V_{\vec{v}} := \V_h^{\ell}(\Omega) \cap \HB_0^1(\Omega), \qquad V_{p} := V_h^{\ell-1}(\Omega) \cap L_0^2(\Omega), \qquad \V_{\vec{B}} := \V_h^k(\Omega) \cap \HB_n^1(\Omega), \qquad \V_{\vec{m}} := \V_h^r(\Omega).
\end{align*}
We note that the pair $(\V_{\vec{v}}, V_p)$ is the usual spaces of Taylor-Hood elements if $\ell \geq 2$.

Our numerical method makes use of the following discrete operators. Firstly, the discrete Laplacian operator $\Delta_h : \V_{\vec{m}} \to \V_{\vec{m}}$ is defined by
\begin{equation} \label{eqn:discrete-laplacian}
    \liprod{\Delta_h \vec{u}_h}{\vec{\chi}_h} = - \liprod{\nabla \vec{u}_h}{\nabla \vec{\chi}_h}, \qquad \forall \vec{u}_h, \vec{\chi}_h \in \V_{\vec{m}}.
\end{equation}
The usual $\LB^2$ projection $\Lproj : \LB^2(\Omega) \to \V_{\vec{m}}$ and Ritz projection $\ellipticProj : \HB^1(\Omega) \to \V_{\vec{m}}$ are defined respectively by
\begin{equation*}
    \liprod{\Lproj \vec{u} - \vec{u}}{\vec{\xi}_h} = 0, \qquad \forall \vec{u} \in \LB^2(\Omega), \quad \forall \vec{\xi}_h \in \V_{\vec{m}},
\end{equation*}
and
\begin{equation*}
    \liprod{\nabla(\ellipticProj \vec{u} - \vec{u})}{\nabla \vec{\xi}_h} = 0 \quad \text{subject to} \quad \liprod{\ellipticProj \vec{u} - \vec{u}}{\vec{1}} = 0, \qquad \forall \vec{u} \in \HB^1(\Omega), \quad \forall \vec{\xi}_h \in \V_{\vec{m}}.
\end{equation*}
Next, we define the Stokes projection operator
\begin{equation*}
    \begin{aligned}
        \stokesVelocityProj : \HB_0^1(\Omega) \times L_0^2(\Omega) &\to \V_{\vec{v}} \times V_p \\
        (\vec{u},q) &\mapsto (\hat{\vec{u}}_h,\hat{q}_h)
    \end{aligned}
\end{equation*}
by
\begin{alignat}{2}
    &\liprod{\nabla(\hat{\vec{u}}_h - \vec{u})}{\nabla \vec{\psi}_h} - \liprod{\hat{q}_h - q}{\div \vec{\psi}_h} = 0,  && \qquad
    \forall\vec{\psi}_h \in \V_{\vec{v}}, \label{eqn:stokes-projection-velocity}\\
    &\liprod{\div (\hat{\vec{u}}_h - \vec{u})}{r_h} = 0,  && \qquad
    \forall r_h \in V_p. \label{eqn:stokes-projection-div-velocity}
\end{alignat}
Finally, the Maxwell projection $\maxwellProj : \HB_{n,\curl}^1(\Omega) \to \V_{\vec{B}}$ is defined for any $\vec{u}$ by
\begin{equation} \label{eqn:maxwell-projection}
    \liprod{\curl(\maxwellProj \vec{u} - \vec{u})}{\curl \vec{\omega}_h} + \liprod{\div(\maxwellProj \vec{u} - \vec{u})}{\div \vec{\omega}_h} = 0, \qquad \forall \vec{\omega}_h \in \V_{\vec{B}}.
\end{equation}

\subsection{Numerical method}
We now describe the numerical scheme for the problem \eqref{eqn:fmhd-equiv}. We partition the time interval $[0, T]$ into $N$ equally spaced intervals with time step size $\tau = T/N$ and define
\begin{equation*}
    t_n := n \tau, \quad \vec{u}^n := \vec{u}(t_n,\cdot), \quad \discreteDtau \vec{u}^n := \frac{\vec{u}^n - \vec{u}^{n-1}}{\tau}, \quad n = 1,\ldots,N,
\end{equation*}
for any function $\vec{u}$ continuous with respect to the time variable.

With the finite element spaces and projections defined in the previous subsection, our numerical scheme for problem \eqref{eqn:fmhd-equiv} reads as follows

\begin{algorithm}
	\label{alg:vBm}
	Let $h>0$ and $N\in\N$ be given. Define $\tau:=T/N$.
	\\
	\textbf{Input}: $h>0$, $N\in\N$,
	$\left(\vec{v}_0,\vec{B}_0,\vec{m}_0\right) \in
	\HB_0^1(\Omega)\times\HB_{n,\curl}^1(\Omega)\times\HB^1(\Omega)$.
	\\
	\textbf{Compute}: $\left(\vec{v}_h^0, \vec{B}_h^0, \vec{m}_h^0\right)
	\in \V_{\vec{v}} \times \V_{\vec{B}} \times \V_{\vec{m}}$ by
	\[
    		\left(\vec{v}_h^0, \vec{B}_h^0, \vec{m}_h^0\right) 
		:= 
		\left(\widehat{(\vec{v}_0)}_h, \maxwellProj \vec{B}_0, \ellipticProj \vec{m}_0\right).
	\]
	\\
	\textbf{For} $n=1$ to $N$ \textbf{do}: Find $\left(\femvec{v}{h}{n}, \femvec{B}{h}{n}, 
	\femvec{m}{h}{n}\right) \in \V_{\vec{v}} \times \V_{\vec{B}} \times \V_{\vec{m}}$ such that
	\begin{align}
    		&\liprod{\discreteDtau \femvec{v}{h}{n}}{\vec{\varphi}_h} 
		+
		 \liprod{\nabla \femvec{v}{h}{n}}{\nabla \vec{\varphi}_h} 
		+
		\frac{1}{2} \liprod{\left(\femvec{v}{h}{n-1} \cdot \nabla\right) 
		\femvec{v}{h}{n}}{\vec{\varphi}_h} 
		-
		\frac{1}{2} \liprod{\left(\femvec{v}{h}{n-1} \cdot \nabla\right) 
		\vec{\varphi}_h}{\femvec{v}{h}{n}} 
		\nonumber \\
    		&\qquad 
		- 
		\liprod{\curl \femvec{B}{h}{n} \times \femvec{B}{h}{n-1}}{\vec{\varphi}_h} 
		+
		\liprod{\left(\femvec{B}{h}{n-1} \cdot \nabla\right) \femvec{m}{h}{n}}{\vec{\varphi}_h} 
		+
		\liprod{\nabla \femvec{m}{h}{n-1}\odot \Delta_h \femvec{m}{h}{n}}{\vec{\varphi}_h} 
		\nonumber \\
    		&\qquad 
		-
		\liprod{\Tilde{p}_h^n}{\div \vec{\varphi}_h}
		= 0, 
		\qquad \forall\vec{\varphi}_h\in \V_{\vec{v}},
		\label{eqn:discrete-velocity-formulation} 
		\\[1ex]
    		&\liprod{\div \femvec{v}{h}{n}}{q_h} = 0, 
		\qquad \forall q_h\in V_p,
		\label{eqn:discrete-div-velocity-formulation} 
		\\[1ex]
    		&\liprod{\discreteDtau \femvec{B}{h}{n}}{\vec{\omega}_h} 
		+
		 \liprod{\curl \femvec{B}{h}{n}}{\curl \vec{\omega}_h} 
		+
		 \liprods{\div \femvec{B}{h}{n}}{\div \vec{\omega}_h} 
		\nonumber \\
    		&\qquad 
		-
		\liprod{\femvec{v}{h}{n} \times \femvec{B}{h}{n-1}}{\curl \vec{\omega}_h} 
		= 0, 
		\qquad \forall\vec{\omega}_h\in \V_{\vec{B}},
		\label{eqn:discrete-magnetic-field-formulation} 
		\\[1ex]
    		&\liprod{\discreteDtau \femvec{m}{h}{n}}{\vec{\xi}_h} 
		+
		 \liprod{\nabla \femvec{m}{h}{n}}{\nabla \vec{\xi}_h} 
		+
		\liprod{\left(\femvec{v}{h}{n} \cdot \nabla\right) \femvec{m}{h}{n-1}}{\vec{\xi}_h} 
		-
		 \liprod{\femvec{m}{h}{n-1} \times \Delta_h \femvec{m}{h}{n}}{\vec{\xi}_h} 
		\nonumber \\
    		&\qquad 
		- 
		 \liprod{\left(\nabla \femvec{m}{h}{n} \cdot \nabla \femvec{m}{h}{n-1}\right) 
		\femvec{m}{h}{n-1}}{\vec{\xi}_h} 
		-
		 \liprod{\femvec{m}{h}{n-1} \times \femvec{B}{h}{n}}{\vec{\xi}_h} 
		\nonumber \\
    		&\qquad 
		-
		 \liprod{\femvec{m}{h}{n-1} \times 
		\left(\femvec{m}{h}{n-1} \times \femvec{B}{h}{n}\right)}{\vec{\xi}_h} 
		= 0,
		\qquad \forall\vec{\xi}_h\in \V_{\vec{m}},
		\label{eqn:discrete-magnetisation-formulation} 
		\\[1ex]
    		&\liprod{\Delta_h \femvec{m}{h}{n}}{\vec{\chi}_h} 
		= 
		-
		\liprod{\nabla \femvec{m}{h}{n}}{\nabla \vec{\chi}_h},
		\qquad \forall\vec{\chi}_h\in \V_{\vec{m}},
	\end{align}
	\\
	\textbf{Output}: The sequence
	$\left\{\big(\vec{v}_h^n,\vec{B}_h^n,\vec{m}_h^n\big)\right\}_{n=0}^N$.
\end{algorithm}

Observe that, following \cite{wang2022-article}, we have added the stabilisation term $\liprods{\div \femvec{B}{h}{n}}{\div \vec{\omega}_h}$ to \eqref{eqn:discrete-magnetic-field-formulation}. This is consistent with \eqref{eqn:fmhd-equiv} and Definition \ref{def:weak-solution} provided that $\div \vec{B}_0 = 0$. Indeed, by taking the divergence of both sides of \eqref{eqn:fmhd-equiv} we see that $\div \vec{B} = 0$.

\subsection{Main results}
In the following sections we assume that the initial data has the regularity
\begin{equation} \label{eqn:initial-data-regularity}
    \begin{alignedat}{3}
        \vec{v}_0 &\in \HB^{\ell+1}(\Omega), &&\quad \ell \geq 2, \\
        \vec{B}_0 &\in \HB^{k+1}(\Omega), &&\quad k \geq 2, \\
        \vec{m}_0 &\in \W^{2,\infty}(\Omega) \cap \HB^{r+1}(\Omega), &&\quad r \geq 2,
    \end{alignedat}
\end{equation}
the exact solution to the system \eqref{eqn:fmhd} has regularity
\begin{align} \label{eqn:exact-solution-regularity}
    \begin{split}
        \vec{v} &\in C([0,T]; \HB^{\ell+1}(\Omega)), \\
        p &\in C([0,T]; \HB^{\ell}(\Omega)), \\
        \vec{B} &\in C([0,T]; \HB^{k+1}(\Omega)), \\
        \vec{m} &\in C([0,T]; \W^{2,\infty}(\Omega) \cap \HB^{r+1}(\Omega)), \\
        \partial_t \vec{v}, \partial_t \vec{B}, \partial_t \vec{m} &\in C([0,T]; \HB^1(\Omega)), \\
        \partial_t^2 \vec{v}, \partial_t^2 \vec{B}, \partial_t^2 \vec{m} &\in C([0,T]; \LB^2(\Omega)),
    \end{split}
\end{align}
and that it also satisfies the sphere constraint $|\vec{m}| = 1$ in $[0,T] \times \Omega$. 

The discrete solutions defined by Algorithm \ref{alg:vBm} satisfy the following error estimate.

\begin{theorem} \label{thm:main-result}
Let $\tau$ and $h$ be the temporal and spatial step sizes, respectively. There exist positive constants $\tau_0$ and $h_0$, such that if $\tau \leq \tau_0$, $h \leq h_0$, and 
\begin{equation} \label{eqn:tau-growth}
    \tau = O(h^{1+\beta})
\end{equation}
for an arbitrary $\beta \in (0,1/2)$ then the linear system \eqref{eqn:discrete-velocity-formulation}--\eqref{eqn:discrete-magnetisation-formulation} is well-posed assuming \eqref{eqn:initial-data-regularity} and \eqref{eqn:exact-solution-regularity} hold with $\ell, k, r \geq 2$. Furthermore, the finite element solutions $\femvec{v}{h}{n}, \femvec{B}{h}{n}$ and $\femvec{m}{h}{n}$ satisfy
\begin{equation} \label{eqn:error-estimate}
    \lnorm{\vec{v}^n - \femvec{v}{h}{n}}{2} + \lnorm{\vec{B}^n - \femvec{B}{h}{n}}{2} + \hnorm{\vec{m}^n - \femvec{m}{h}{n}}{1} \leq C \left(\tau + h^{\ell - \beta} + h^{k - \beta} + h^{r - \beta}\right),
\end{equation}
for all $n \in \{0,\ldots,N\}$, where $\tau = T/N$. Consequently, $\femvec{m}{h}{n}$ satisfies
\begin{equation*}
    \lnorm{1 - |\femvec{m}{h}{n}|^2}{2} \leq C \left( \tau + h^{\ell-\beta} + h^{k-\beta} + h^{r-\beta}\right),
\end{equation*}
implying the magnetisation vector $\femvec{m}{h}{n}$ satisfies the sphere constraint asymptotically.
\end{theorem}

\begin{remark}
We note that our assumption \eqref{eqn:tau-growth} is stronger than that of \cite{an2021-article}. The authors of \cite{an2021-article} assume $\tau = O(\varepsilon_0 h)$ and base their arguments on having $\varepsilon_0$ sufficiently small. There is a flaw in their arguments because the big-O constant $C$ goes to infinity when $\varepsilon_0 \to 0$. Indeed, in the derivation of equations (3.27) and (3.28) in \cite{an2021-article}, where it is required that $C h_0^{r-1/2} + C \hat{C}_0 \varepsilon_0^{1/2} \leq 1/2$, it is necessary that the constant $C$ stays constant while $\varepsilon_0 \to 0$. This is exactly our assumption \eqref{eqn:tau-growth}, i.e., $\varepsilon_0 = h^{\beta}$. In other words, the error estimate given in \cite[Theorem 2.3]{an2021-article} is suboptimal. Furthermore, we relax the smooth boundary assumption of \cite{an2021-article} but get a slightly worse estimate in the process.
\end{remark}

\section{Error equations} \label{sec:error-equations}
For the sake of notational simplicity in the analysis, define
\begin{equation}
\label{equ:error split}
\left\{
\begin{aligned}
    \vec{\delta}^n_{\vec{v}} &:= \vec{v}^{n} - \vec{v}^{n}_h = \left(\vec{v}^{n} - \hat{\vec{v}}_h^{n}\right) + \left(\hat{\vec{v}}_h^{n} - \vec{v}^{n}_h\right) =: \vec{E}^n_{\vec{v}} + \vec{e}^n_{\vec{v}} \\
    \delta_p^n &:= \tilde{p}^n - \tilde{p}_h^n = \left(\tilde{p}^n - \hat{\tilde{p}}_h^n\right) + \left(\hat{\tilde{p}}_h^n - \tilde{p}_h^n\right) =: E_p^n + e_p^n \\
    \vec{\delta}^n_{\vec{B}} &:= \vec{B}^{n} - \vec{B}^{n}_h = \left(\vec{B}^{n} - \maxwellProj \vec{B}^{n}\right) + \left(\maxwellProj \vec{B}^{n}_h - \vec{B}^{n}_h\right) =: \vec{E}^n_{\vec{B}} + \vec{e}^n_{\vec{B}} \\
    \vec{\delta}^n_{\vec{m}} &:= \vec{m}^{n} - \vec{m}^{n}_h = \left(\vec{m}^{n} - \ellipticProj \vec{m}^{n}\right) + \left(\ellipticProj \vec{m}^{n} - \vec{m}^{n}_h\right) =: \vec{E}^n_{\vec{m}} + \vec{e}^n_{\vec{m}}.
\end{aligned}
\right.
\end{equation}
Additionally, we abuse notation to also define

\begin{equation*}
     \Delta_h \femvec{\delta}{\vec{m}}{n} := \Delta \vec{m}^n - \Delta_h \femvec{m}{h}{n} \quad \text{and} \quad \Delta_h \femvec{E}{\vec{m}}{n} := \Delta \vec{m}^n - \Delta_h \ellipticProj \vec{m}^n 
\end{equation*}
so that
\begin{equation} \label{eqn:discrete-laplacian-error-definition}
    \Delta_h \femvec{\delta}{\vec{m}}{n} = \Delta_h \femvec{E}{\vec{m}}{n} + \Delta_h \femvec{e}{\vec{m}}{n}.
\end{equation}
Now we derive the weak formulation satisfied by the various $\vec{\delta}$-terms.

The exact velocity $\vec{v}^n$ satisfies
\begin{align} \label{eqn:exact-velocity-exact-formulation}
    &\liprod{ \vec{v}_t^n}{\vec{\varphi}_h} + \liprod{\nabla \vec{v}^n}{\nabla \vec{\varphi}_h} -\liprod{\tilde{p}^n}{\div \vec{\varphi}_h} \nonumber \\
    &\qquad + \frac{1}{2} \liprod{(\vec{v}^n \cdot \nabla)\vec{v}^n}{\vec{\varphi}_h} - \frac{1}{2} \liprod{(\vec{v}^n \cdot \nabla)\vec{\varphi}_h}{\vec{v}^n} \nonumber \\
    &\qquad - \liprod{\curl \vec{B}^n \times \vec{B}^n}{\vec{\varphi}_h} + \liprod{(\vec{B}^n \cdot \nabla) \vec{m}^n}{\vec{\varphi}_h} + \liprod{\nabla \vec{m}^n \odot \Delta \vec{m}^n}{\vec{\varphi}_h} \nonumber \\
    &\qquad \qquad  = 0, \qquad \forall \vec{\varphi}_h \in \V_{\vec{v}}.
\end{align}
Therefore, since
\begin{equation*}
    \liprod{\nabla \hat{\vec{v}}_h^{n}}{\nabla \vec{\varphi}_h} - \liprod{\hat{\tilde{p}}_h^n}{\div \vec{\varphi}_h} = \liprod{\nabla \vec{v}^{n}}{\nabla \vec{\varphi}_h} - \liprod{\tilde{p}^n}{\div \vec{\varphi}_h},
\end{equation*}
the velocity $\vec{v}^n$ also satisfies
\begin{align} \label{eqn:exact-velocity-discrete-formulation}
    &\liprod{\discreteDtau \vec{v}^n}{\vec{\varphi}_h} + \liprod{\nabla \hat{\vec{v}}_h^n}{\nabla \vec{\varphi}_h} - \liprod{\hat{\tilde{p}}_h^n}{\div \vec{\varphi}_h} \nonumber \\
    &\qquad +\frac{1}{2} \liprod{(\vec{v}^{n-1} \cdot \nabla)\vec{v}^{n}}{\vec{\varphi}_h} - \frac{1}{2} \liprod{(\vec{v}^{n-1} \cdot \nabla) \vec{\varphi}_h}{\vec{v}^{n}} \nonumber \\
    &\qquad -\liprod{\curl \vec{B}^{n} \times \vec{B}^{n-1}}{\vec{\varphi}_h} + \liprod{(\vec{B}^{n-1} \cdot \nabla) \vec{m}^{n}}{\vec{\varphi}_h} + \liprod{\nabla \vec{m}^{n-1} \odot \Delta \vec{m}^n}{\vec{\varphi}_h} \nonumber \\
    &\qquad \qquad = \left(\liprod{\discreteDtau \vec{v}^{n}}{\vec{\varphi}_h} - \liprod{\vec{v}_t^{n}}{\vec{\varphi}_h}\right) \nonumber \\
    &\qquad \qquad \qquad + \left(\frac{1}{2} \liprod{(\vec{v}^{n-1} \cdot \nabla)\vec{v}^{n}}{\vec{\varphi}_h} - \frac{1}{2} \liprod{(\vec{v}^{n} \cdot \nabla)\vec{v}^{n}}{\vec{\varphi}_h}\right) \nonumber \\
    &\qquad \qquad \qquad - \left(\frac{1}{2} \liprod{(\vec{v}^{n-1} \cdot \nabla)\vec{\varphi}_h}{\vec{v}^{n}} - \frac{1}{2} \liprod{(\vec{v}^{n} \cdot \nabla) \vec{\varphi}_h}{\vec{v}^{n}}\right) \nonumber \\
    &\qquad \qquad \qquad - \left(\liprod{\curl \vec{B}^{n} \times \vec{B}^{n-1}}{\vec{\varphi}_h} - \liprod{\curl \vec{B}^{n} \times \vec{B}^{n}}{\vec{\varphi}_h}\right) \nonumber \\
    &\qquad \qquad \qquad + \left(\liprod{(\vec{B}^{n-1} \cdot \nabla)\vec{m}^{n}}{\vec{\varphi}_h} - \liprod{(\vec{B}^{n} \cdot \nabla) \vec{m}^{n}}{\vec{\varphi}_h}\right) \nonumber \\
    &\qquad \qquad \qquad + \left(\liprod{\nabla \vec{m}^{n-1} \odot \Delta \vec{m}^n}{\vec{\varphi}_h} - \liprod{\nabla \vec{m}^{n} \odot \Delta \vec{m}^n}{\vec{\varphi}_h}\right) \nonumber \\
    &\qquad \qquad := R_1(\vec{\varphi}_h) \qquad \forall \vec{\varphi}_h \in \V_{\vec{v}}.
\end{align}

Additionally, from the definition of the Stokes projection in Subsection \ref{sec:fem-spaces} and the fact that
\begin{equation*}
    \liprod{\div \vec{v}^{n}}{q_h} = 0, \qquad \forall q_h \in V_p,
\end{equation*}
we have
\begin{equation} \label{eqn:exact-velcity-divergence-discrete-formulation}
    \liprod{\div \hat{\vec{v}}_h^{n}}{q_h} = 0, \qquad \forall q_h \in V_p.
\end{equation}

Similarly, the exact magnetic field $\vec{B}^n$ and the exact magnetisation satisfy, respectively,
\begin{align} \label{eqn:exact-magnetic-field-discrete-formulation}
    &\liprod{\discreteDtau \vec{B}^{n}}{\vec{\omega}_h} + \liprod{\curl(\maxwellProj \vec{B}^{n})}{\curl \vec{\omega}_h} + \liprods{\div\left(\maxwellProj \vec{B}^n\right)}{\div \vec{\omega}_h} \nonumber \\
    &\qquad - \liprod{\vec{v}^{n} \times \vec{B}^{n-1}}{\curl \vec{\omega}_h} \nonumber \\
    &\qquad \qquad = \left(\liprod{\discreteDtau \vec{B}^{n}}{\vec{\omega}_h} - \liprod{\vec{B}_t^{n}}{\vec{\omega}_h}\right) \nonumber \\
    &\qquad \qquad \qquad  - \left(\liprod{\vec{v}^{n} \times \vec{B}^{n-1}}{\curl \vec{\omega}_h} - \liprod{\vec{v}^{n} \times \vec{B}^{n}}{\curl \vec{\omega}_h}\right) \nonumber \\
    &\qquad \qquad := R_2(\vec{\omega}_h) \qquad \forall \vec{\omega}_h \in \V_{\vec{B}},
\end{align}
and
\begin{align} \label{eqn:exact-magnetisation-discrete-formulation}
    &\liprod{\discreteDtau \vec{m}^{n}}{\vec{\xi}_h} +  \liprod{\nabla \ellipticProj \vec{m}^{n}}{\nabla \vec{\xi}_h} + \liprod{(\vec{v}^{n} \cdot \nabla) \vec{m}^{n-1}}{\vec{\xi}_h} -  \liprod{\vec{m}^{n-1} \times \Delta \vec{m}^{n}}{\vec{\xi}_h} \nonumber \\
    &\qquad -  \liprod{(\nabla \vec{m}^{n} \cdot \nabla \vec{m}^{n-1})\vec{m}^{n-1}}{\vec{\xi}_h} - \liprod{\vec{m}^{n-1} \times \vec{B}^{n}}{\vec{\xi}_h} - \liprod{\vec{m}^{n-1} \times (\vec{m}^{n-1} \times \vec{B}^{n})}{\vec{\xi}_h} \nonumber \\
    &\qquad \qquad = \left(\liprod{\discreteDtau \vec{m}^{n}}{\vec{\xi}_h} - \liprod{\vec{m}_t^{n}}{\vec{\xi}_h}\right) \nonumber \\
    &\qquad \qquad \qquad + \left(\liprod{(\vec{v}^{n} \cdot \nabla)\vec{m}^{n-1}}{\vec{\xi}_h} - \liprod{(\vec{v}^{n} \cdot \nabla)\vec{m}^{n}}{\vec{\xi}_h}\right) \nonumber \\
    &\qquad \qquad \qquad - \left(\liprod{\vec{m}^{n-1} \times \Delta \vec{m}^{n}}{\vec{\xi}_h} - \liprod{\vec{m}^{n} \times \Delta \vec{m}^{n}}{\vec{\xi}_h}\right) \nonumber \\
    &\qquad \qquad \qquad - \left(\liprod{\nabla \vec{m}^{n} \cdot \nabla \vec{m}^{n-1})\vec{m}^{n-1}}{\vec{\xi}_h} - \liprod{(\nabla \vec{m}^{n} \cdot \nabla \vec{m}^{n})\vec{m}^{n}}{\vec{\xi}_h}\right) \nonumber \\
    &\qquad \qquad \qquad - \left(\liprod{\vec{m}^{n-1} \times \vec{B}^{n}}{\vec{\xi}_h} - \liprod{\vec{m}^{n} \times \vec{B}^{n}}{\vec{\xi}_h}\right) \nonumber \\
    &\qquad \qquad \qquad - \left(\liprod{\vec{m}^{n-1} \times (\vec{m}^{n-1} \times \vec{B}^{n})}{\vec{\xi}_h} - \liprod{\vec{m}^{n} \times (\vec{m}^{n} \times \vec{B}^{n})}{\vec{\xi}_h}\right) \nonumber \\
    &\qquad \qquad := R_3(\vec{\xi}_h) \qquad \forall \vec{\xi}_h \in \V_{\vec{m}}.
\end{align}

By subtracting \eqref{eqn:discrete-velocity-formulation} from \eqref{eqn:exact-velocity-discrete-formulation}, we obtain the velocity error equation
\begin{align} \label{eqn:error-velocity}
    &\liprod{\discreteDtau \vec{\delta}^{n}_{\vec{v}}}{\vec{\varphi}_h} + \liprod{\nabla \vec{e}^{n}_{\vec{v}}}{\nabla \vec{\varphi}_h} \nonumber \\
    &\qquad + \left(\frac{1}{2} \liprod{(\vec{\delta}^{n-1}_{\vec{v}} \cdot \nabla) \vec{v}^{n}}{\vec{\varphi}_h} + \frac{1}{2} \liprod{(\vec{v}^{n-1}_{h} \cdot \nabla) \vec{\delta}_{\vec{v}}^n}{\vec{\varphi}_h}\right) \nonumber \\
    &\qquad - \left(\frac{1}{2} \liprod{(\vec{\delta}^{n-1}_{\vec{v}} \cdot \nabla) \vec{\varphi}_h}{\vec{v}^{n}} + \frac{1}{2} \liprod{(\vec{v}^{n-1}_{h} \cdot \nabla)\vec{\varphi}_h}{\vec{\delta}^{n}_{\vec{v}}}\right) \nonumber \\
    &\qquad - \left(\liprod{\curl \vec{\delta}^{n}_{\vec{B}} \times \vec{B}^{n-1}_h}{\vec{\varphi}_h} + \liprod{\curl \vec{B}^{n} \times \vec{\delta}^{n-1}_{\vec{B}}}{\vec{\varphi}_h}\right) \nonumber \\
    &\qquad + \left(\liprod{(\vec{\delta}^{n-1}_{\vec{B}} \cdot \nabla) \vec{m}^{n}}{\vec{\varphi}_h} + \liprod{(\vec{B}^{n-1}_{h} \cdot \nabla)\vec{\delta}^{n}_{\vec{m}}}{\vec{\varphi}_h}\right) \nonumber \\
    &\qquad + \left(\liprod{\nabla \femvec{\delta}{\vec{m}}{n-1} \odot \Delta \vec{m}^n}{\vec{\varphi}_h} + \liprod{\nabla \vec{m}^{n-1} \odot \Delta_h \femvec{\delta}{\vec{m}}{n}}{\vec{\varphi}_h} - \liprod{\nabla \femvec{\delta}{\vec{m}}{n-1} \odot \Delta_h \femvec{\delta}{\vec{m}}{n}}{\vec{\varphi}_h}\right) \nonumber \\
    &\qquad - \liprod{e^{n}_{p}}{\div \vec{\varphi}_h} \nonumber \\
    &\qquad \qquad = R_1(\vec{\varphi}_h) \qquad \forall \vec{\varphi}_h \in \V_{\vec{v}}.
\end{align}
Note that
\begin{equation} \label{eqn:error-velocity-divergence}
    \liprod{\div \vec{e}^{n}_{\vec{v}}}{q_h} = 0 \qquad \forall q_h \in V_p.
\end{equation}

Similarly, by subtracting \eqref{eqn:discrete-magnetic-field-formulation} from \eqref{eqn:exact-magnetic-field-discrete-formulation} we obtain the magnetic field error equation
\begin{align} \label{eqn:error-magnetic-field}
    &\liprod{\discreteDtau \vec{\delta}^{n}_{\vec{B}}}{\vec{\omega}_h} + \liprod{\curl \vec{e}^{n}_{\vec{B}}}{\curl \vec{\omega}_h} + \liprods{\div \vec{e}_{\vec{B}}^n}{\div \vec{\omega}_h} \nonumber \\
    &\qquad - \left(\liprod{\vec{\delta}^{n}_{\vec{v}} \times \vec{B}^{n-1}_h}{\curl \vec{\omega}_h} + \liprod{\vec{v}^{n} \times \vec{\delta}^{n-1}_{\vec{B}}}{\curl \vec{\omega}_h}\right) \nonumber \\
    &\qquad \qquad = R_2(\vec{\omega}_h) \qquad \forall \vec{\omega}_h \in \V_{\vec{B}}.
\end{align}

Finally, the magnetisation error equation is obtained by subtracting \eqref{eqn:discrete-magnetisation-formulation} from \eqref{eqn:exact-magnetisation-discrete-formulation}:
\begin{align} \label{eqn:error-magnetisation}
    &\liprod{\discreteDtau \vec{\delta}^{n}_{\vec{m}}}{\vec{\xi}_h} +  \liprod{\nabla \vec{e}^{n}_{\vec{m}}}{\nabla \vec{\xi}_h} \nonumber \\
    &\qquad + \left(\liprod{(\vec{\delta}^{n}_{\vec{v}} \cdot \nabla) \vec{m}^{n-1}_h}{\vec{\xi}_h} + \liprod{(\vec{v}^{n} \cdot \nabla) \vec{\delta}^{n-1}_{\vec{m}}}{\vec{\xi}_h}\right) \nonumber \\
    &\qquad - \left( \liprod{\femvec{\delta}{\vec{m}}{n-1} \times \Delta \vec{m}^n}{\vec{\xi}_h} +  \liprod{\femvec{m}{h}{n-1} \times \Delta_h \femvec{\delta}{\vec{m}}{n}}{\vec{\xi}_h}\right) \nonumber \\
    &\qquad -\bigg( \liprod{(\nabla \vec{m}^n \cdot \nabla \vec{m}^{n-1}) \femvec{\delta}{\vec{m}}{n-1}}{\vec{\xi}_h} +  \liprod{(\nabla \femvec{\delta}{\vec{m}}{n} \cdot \nabla \vec{m}^{n-1}) \femvec{m}{h}{n-1}}{\vec{\xi}_h} \nonumber \\
    &\qquad \qquad +  \liprod{(\nabla \femvec{\delta}{\vec{m}}{n-1} \cdot \nabla \vec{m}^n) \femvec{m}{h}{n-1}}{\vec{\xi}_h} -  \liprod{(\nabla \femvec{\delta}{\vec{m}}{n} \cdot \nabla \femvec{\delta}{\vec{m}}{n-1}) \femvec{m}{h}{n-1}}{\vec{\xi}_h}\bigg) \nonumber \\
    &\qquad - \left(\liprod{\vec{\delta}^{n-1}_{\vec{m}} \times \vec{B}^{n}}{\vec{\xi}_h} + \liprod{\vec{m}^{n-1}_{h} \times \vec{\delta}^{n}_{\vec{B}}}{\vec{\xi}_h}\right) \nonumber \\
    &\qquad - \bigg(\liprod{\vec{\delta}^{n-1}_{\vec{m}} \times (\vec{m}^{n-1} \times \vec{B}^{n})}{\vec{\xi}_h} + \liprod{\vec{m}^{n-1}_{h} \times (\vec{\delta}^{n-1}_{\vec{m}} \times \vec{B}^{n})}{\vec{\xi}_h} \nonumber \\
    &\qquad \qquad + \liprod{\vec{m}^{n-1}_{h} \times (\vec{m}^{n-1}_{h} \times \vec{\delta}^{n}_{\vec{B}})}{\vec{\xi}_h}\bigg) \nonumber \\
    &\qquad \qquad \qquad = R_3(\vec{\xi}_h) \qquad \forall \vec{\xi}_h \in \V_{\vec{m}}.
\end{align}

The error estimates in Theorem \ref{thm:main-result} are essentially estimates for the $\vec{\delta}^n$-terms, which are composed of $\vec{E}^n$ and $\vec{e}^n$. As error estimates for $\vec{E}^n$-terms are just the error estimates of the projection, we focus our efforts on proving error estimates for the $\vec{e}^n$-terms, namely \eqref{eqn:induction-error-estimate} to be proved in the next section.

\section{Proof of Theorem \ref{thm:main-result}} \label{sec:theorem-proof}

We will prove the well-posedness of the linear system \eqref{eqn:discrete-velocity-formulation}--\eqref{eqn:discrete-magnetisation-formulation} concurrently with estimates for the error terms $\femvec{e}{\vec{v}}{n}$, $\femvec{e}{\vec{B}}{n}$, and $\femvec{e}{\vec{m}}{n}$, which, when combined with the projection error terms $\femvec{E}{\vec{v}}{n}$, $\femvec{E}{\vec{B}}{n}$, and $\femvec{E}{\vec{m}}{n}$, will yield \eqref{eqn:error-estimate}, see \eqref{equ:error split}. First, we re-write the system \eqref{eqn:discrete-velocity-formulation}--\eqref{eqn:discrete-magnetisation-formulation} in the following form.

For $1 \leq n \leq N$, given $(\femvec{v}{h}{n-1}, \tilde{p}_h^{n-1}, \femvec{B}{h}{n-1}, \femvec{m}{h}{n-1}) \in \V_{\vec{v}} \times V_p \times \V_{\vec{B}} \times \V_{\vec{m}}$ we seek $(\femvec{v}{h}{n}, \tilde{p}_h^{n}, \femvec{B}{h}{n}, \femvec{m}{h}{n}) \in \V_{\vec{v}} \times V_p \times \V_{\vec{B}} \times \V_{\vec{m}}$ satisfying
\begin{align*}
    &\liprod{\femvec{v}{h}{n}}{\vec{\varphi}_h} + \tau \liprod{\nabla \femvec{v}{h}{n}}{\nabla \vec{\varphi}_h} +
    \frac{\tau}{2} \liprod{\left(\femvec{v}{h}{n-1} \cdot \nabla\right) \femvec{v}{h}{n}}{\vec{\varphi}_h} - \frac{\tau}{2} \liprod{\left(\femvec{v}{h}{n-1} \cdot \nabla\right) \vec{\varphi}_h}{\femvec{v}{h}{n}} \nonumber \\
    &\qquad - \tau \liprod{\curl \femvec{B}{h}{n} \times \femvec{B}{h}{n-1}}{\vec{\varphi}_h} + \tau \liprod{\left(\femvec{B}{h}{n-1} \cdot \nabla\right) \femvec{m}{h}{n}}{\vec{\varphi}_h} + \tau \liprod{\nabla \femvec{m}{h}{n-1}\odot \Delta_h \femvec{m}{h}{n}}{\vec{\varphi}_h} \nonumber \\
    &\qquad - \tau \liprod{\Tilde{p}_h^n}{\div \vec{\varphi}_h} = \liprod{\femvec{v}{h}{n-1}}{\vec{\varphi}_h}, \qquad \forall\vec{\varphi}_h\in \V_{\vec{v}},  \\[1ex]
    &\liprod{\div \femvec{v}{h}{n}}{q_h} = 0, \qquad \forall q_h\in V_p, \\[1ex]
    &\liprod{\femvec{B}{h}{n}}{\vec{\omega}_h} + \tau \liprod{\curl \femvec{B}{h}{n}}{\curl \vec{\omega}_h} + \tau \liprods{\div \femvec{B}{h}{n}}{\div \vec{\omega}_h} \nonumber \\
    &\qquad - \tau \liprod{\femvec{v}{h}{n} \times \femvec{B}{h}{n-1}}{\curl \vec{\omega}_h} = \liprod{\femvec{B}{h}{n-1}}{\vec{\omega}_h}, 
    \qquad \forall\vec{\omega}_h\in \V_{\vec{B}},  \\[1ex]
    &\liprod{\femvec{m}{h}{n}}{\vec{\xi}_h} +\tau \liprod{\nabla \femvec{m}{h}{n}}{\nabla \vec{\xi}_h} + \tau \liprod{\left(\femvec{v}{h}{n} \cdot \nabla\right) \femvec{m}{h}{n-1}}{\vec{\xi}_h} - \tau \liprod{\femvec{m}{h}{n-1} \times \Delta_h \femvec{m}{h}{n}}{\vec{\xi}_h} \nonumber \\
    &\qquad - \tau \liprod{\left(\nabla \femvec{m}{h}{n} \cdot \nabla \femvec{m}{h}{n-1}\right) \femvec{m}{h}{n-1}}{\vec{\xi}_h} - \tau \liprod{\femvec{m}{h}{n-1} \times \femvec{B}{h}{n}}{\vec{\xi}_h} \nonumber \\
    &\qquad - \tau \liprod{\femvec{m}{h}{n-1} \times \left(\femvec{m}{h}{n-1} \times \femvec{B}{h}{n}\right)}{\vec{\xi}_h} = \liprod{\femvec{m}{h}{n-1}}{\vec{\xi}_h}, \qquad \forall\vec{\xi}_h\in \V_{\vec{m}}.
\end{align*}

Take $(\femvec{v}{h}{n-1}, \tilde{p}_h^{n-1}, \femvec{B}{h}{n-1}, \femvec{m}{h}{n-1})$ to be the projection of the initial data when $n=1$ and for $n \geq 2$ take it to be the finite element solutions to \eqref{eqn:discrete-velocity-formulation}--\eqref{eqn:discrete-magnetisation-formulation} at time step $n-1$ as per Algorithm \ref{alg:vBm}. We may now consider the homogeneous version of the system (i.e., with RHS equal to zero) as follows:

For $1 \leq n \leq N$, given $(\femvec{v}{h}{n-1}, \tilde{p}_h^{n-1}, \femvec{B}{h}{n-1}, \femvec{m}{h}{n-1}) \in \V_{\vec{v}} \times V_p \times \V_{\vec{B}} \times \V_{\vec{m}}$ we seek $(\femvec{\dot{v}}{h}{n}, \dot{\tilde{p}}_h^n, \femvec{\dot{B}}{h}{n}, \femvec{\dot{m}}{h}{n}) \in \V_{\vec{v}} \times V_p \times \V_{\vec{B}} \times \V_{\vec{m}}$ such that
\begin{align}
    &\liprod{\femvec{\dot{v}}{h}{n}}{\vec{\varphi}_h} +  \tau \liprod{\nabla \femvec{\dot{v}}{h}{n}}{\nabla \vec{\varphi}_h} +\frac{\tau}{2} \liprod{(\femvec{v}{h}{n-1} \cdot \nabla)\femvec{\dot{v}}{h}{n}}{\vec{\varphi}_h} - \frac{\tau}{2} \liprod{(\femvec{v}{h}{n-1} \cdot \nabla)\vec{\varphi}_h}{\femvec{\dot{v}}{h}{n}} \nonumber \\
    &\qquad - \tau \liprod{\curl \femvec{\dot{B}}{h}{n} \times \femvec{B}{h}{n-1}}{\vec{\varphi}_h} + \tau \liprod{(\femvec{B}{h}{n-1} \cdot \nabla)\femvec{\dot{m}}{h}{n}}{\vec{\varphi}_h} + \tau \liprod{\nabla \femvec{m}{h}{n-1} \odot \Delta_h \femvec{\dot{m}}{h}{n}}{\vec{\varphi}_h}  \nonumber \\
    &\qquad\qquad - \liprods{\dot{\tilde{p}}_h^n}{\div \vec{\varphi}_h} = 0,  \label{eqn:homo-discrete-formulation-velocity} \\[2ex]
    &\liprod{\div \femvec{\dot{v}}{h}{n}}{q_h} = 0 \qquad \forall q_h \in V_p, \label{eqn:homo-discrete-formulation-div-velocity} \\[2ex]
    &\liprod{\femvec{\dot{B}}{h}{n}}{\vec{\omega}_h} +  \tau \liprod{\curl \femvec{\dot{B}}{h}{n}}{\curl \vec{\omega}_h} +  \tau \liprods{\div \femvec{\dot{B}}{h}{n}}{\div \vec{\omega}_h} \nonumber \\
    &\qquad - \tau \liprod{\femvec{\dot{v}}{h}{n} \times \femvec{B}{h}{n-1}}{\curl \vec{\omega}_h} = 0  \qquad \forall \vec{\omega}_h \in \V_{\vec{B}}, \label{eqn:homo-discrete-formulation-magnetic-field} \\[2ex]
    &\liprod{\femvec{\dot{m}}{h}{n}}{\vec{\xi}_h} +  \tau \liprod{\nabla \femvec{\dot{m}}{h}{n}}{\nabla \vec{\xi}_h} +\tau \liprod{(\femvec{\dot{v}}{h}{n} \cdot \nabla)\femvec{m}{h}{n-1}}{\vec{\xi}_h} - \tau \liprod{\femvec{m}{h}{n-1} \times \Delta_h \femvec{\dot{m}}{h}{n}}{\vec{\xi}_h} \nonumber \\
    &\qquad - \tau \liprod{(\nabla \femvec{\dot{m}}{h}{n} \cdot \nabla \femvec{m}{h}{n-1})\femvec{m}{h}{n-1}}{\vec{\xi}_h} - \tau \liprod{\femvec{m}{h}{n-1} \times \femvec{\dot{B}}{h}{n}}{\vec{\xi}_h} \nonumber \\
    &\qquad - \tau \liprod{\femvec{m}{h}{n-1} \times (\femvec{m}{h}{n-1} \times \femvec{\dot{B}}{h}{n})}{\vec{\xi}_h} = 0 \qquad \forall \vec{\xi}_h \in \V_{\vec{m}}. \label{eqn:homo-discrete-formulation-magnetisation}
\end{align}
We shall use the notation $\femvec{\dot{u}}{h}{n}$ to distinguish the solution of the homogeneous linear system from the inhomogeneous version, and we will show that the homogeneous system only has the zero solution.

We shall now use induction to prove that there exists a positive constant $A$ such that for all $n \in \{1,\ldots,N\}$, the solution $(\femvec{v}{h}{n}, \tilde{p}_h^n, \femvec{B}{h}{n}, \femvec{m}{h}{n})$ to \eqref{eqn:discrete-velocity-formulation}--\eqref{eqn:discrete-magnetisation-formulation} exists and satisfies
\begin{align} \label{eqn:induction-error-estimate}
    P^n + \tau \sum_{i=1}^n Q^i \leq A \left(\tau^2 + h^{2(\ell-\beta)} + h^{2(k-\beta)} + h^{2(r-\beta)}\right), \qquad \forall h \in (0,h_0],
\end{align}
where $h_0$ is any positive constant satisfying $h_0 < A^{-1/\beta}$ and where
\begin{align*}
    P^n &:= \lnorm{\femvec{e}{\vec{v}}{n}}{2}^2 + \lnorm{\femvec{e}{\vec{B}}{n}}{2}^2 + \hnorm{\femvec{e}{\vec{m}}{n}}{1}^2, \\
    Q^n &:= \hnorm{\femvec{e}{\vec{v}}{n}}{1}^2 + \hnorm{\femvec{e}{\vec{B}}{n}}{1}^2 + \lnorm{\Delta_h \femvec{e}{\vec{m}}{n}}{2}^2.
\end{align*}

\noindent \textbf{Base Case.} For all $1 \leq  n \leq N$, define
\begin{align*}
    X^n &:= \lnorm{\femvec{\dot{v}}{h}{n}}{2}^2 + \lnorm{\femvec{\dot{B}}{h}{n}}{2}^2 + \hnorm{\femvec{\dot{m}}{h}{n}}{1}^2, \\
    Y^n &:= \hnorm{\femvec{\dot{v}}{h}{n}}{1}^2 + \hnorm{\femvec{\dot{B}}{h}{n}}{1}^2 + \lnorm{\Delta_h \femvec{\dot{m}}{h}{n}}{2}^2, \\
    R &:= \tau^2 + \big(h^{2\ell} +h^{2k} + h^{2r}\big) + \tau^{-2} \big(h^{2(\ell+1)} +h^{2(k+1)} + h^{2(r+1)}\big).
\end{align*}
Observe that the initial values $(\femvec{v}{h}{0}, \femvec{B}{h}{0}, \femvec{m}{h}{0})$ are well-defined from Algorithm \ref{alg:vBm}, and note that this implies that
\begin{equation} \label{eqn:thm-proof_1}
    \femvec{\delta}{\vec{v}}{0} = \femvec{E}{\vec{v}}{0}, \quad \femvec{\delta}{\vec{B}}{0} = \femvec{E}{\vec{B}}{0}, \quad \text{and} \quad \femvec{\delta}{\vec{m}}{0} = \femvec{E}{\vec{m}}{0},
\end{equation}
since
\begin{equation} \label{eqn:thm-proof_2}
    \femvec{e}{\vec{v}}{0} = \femvec{e}{\vec{B}}{0} = \femvec{e}{\vec{m}}{0} = \vec{0}.
\end{equation}
By Lemma \ref{lem:homo-estimates} we conclude that
\begin{equation*}
    X^1 + \tau Y^1 \leq c(\varepsilon) \tau X^1 + C(\varepsilon + h^{1/2}) \tau Y^1.
\end{equation*}
For sufficiently small $\varepsilon$, and subsequently sufficiently small $h_0$, we may absorb $Y^1$ into the left-hand side to obtain
\begin{equation*}
    X^1 + \tau Y^1 \leq C \tau X^1.
\end{equation*}
Hence, for sufficiently small $\tau_0$ we have
\begin{equation*}
    X^1 + \tau Y^1 \leq 0,
\end{equation*}
which gives $X^1 = 0$. In other words, for $n=1$ the homogeneous linear system \eqref{eqn:homo-discrete-formulation-velocity}--\eqref{eqn:homo-discrete-formulation-magnetisation} has only the zero solution
\begin{equation*}
    \femvec{\dot{v}}{h}{1} = \femvec{\dot{B}}{h}{1} = \femvec{\dot{m}}{h}{1} = \vec{0},
\end{equation*}
and $\dot{\tilde{p}}_h^1 = 0$, the latter being obtained by noting that Taylor-Hood elements are inf-sup stable. Therefore, the homogeneous linear system is well-posed at $n=1$ which implies that \eqref{eqn:discrete-velocity-formulation}--\eqref{eqn:discrete-magnetisation-formulation} is well-posed at $n=1$ as well. And so, the finite element solutions $(\femvec{v}{h}{1}, \tilde{p}_h^1, \femvec{B}{h}{1}, \femvec{m}{h}{1})$ exist.

Having obtained the existence of a solution at $n=1$ and noting \eqref{eqn:thm-proof_1}, we apply Lemmas \ref{lem:inductive-step-velocity}--\ref{lem:inductive-step-magnetisation-H1} to obtain
\begin{equation*}
    P^1 - P^0 + C_{\Omega} \tau Q^1 \leq c(\varepsilon) \tau R + c(\varepsilon) \tau P^0 + c(\varepsilon) \tau P^1 + C(\varepsilon + h^{1/2}) \tau Q^1.
\end{equation*}
By \eqref{eqn:thm-proof_2} we have $P^0 = 0$, and so we have
\begin{equation*}
    P^1 + C_{\Omega} \tau Q^1 + \leq c(\varepsilon) \tau R + c(\varepsilon) \tau P^1 + C(\varepsilon + h^{1/2}) \tau Q^1.
\end{equation*}
In a similar manner to $X^1$ and $Y^1$, we obtain for sufficiently small $\varepsilon$, $h_0$, and $\tau_0$,
\begin{equation*}
    P^1 + \tau Q^1 \leq C_0 \tau R,
\end{equation*}
where $C_0$ is some positive constant, implying \eqref{eqn:induction-error-estimate} for $n=1$, completing the proof of the base case. \\

\noindent \textbf{Inductive Assumption.} Let $n \geq 2$. For all $1 \leq i \leq n-1$ assume that $(\femvec{v}{h}{i}, \tilde{p}_h^i, \femvec{B}{h}{i}, \femvec{m}{h}{i})$ exists and that
\begin{align} \label{eqn:inductive-assumption}
    P^i + \tau \sum_{j=1}^i Q^j \leq A (\tau^2 + h^{2(\ell-\beta)} + h^{2(k-\beta)} + h^{2(r-\beta)}), \qquad \forall h \in (0,h_0]
\end{align}
where $A$ is some positive constant to be chosen and $h_0$ is any positive
constant satisfying $h_0 < A^{-1/\beta}$.  \\

\noindent \textbf{Inductive Step.} By Lemma \ref{lem:homo-estimates} and the same arguments as in the base case, we conclude that the non-homogeneous linear system \eqref{eqn:discrete-velocity-formulation}--\eqref{eqn:discrete-magnetisation-formulation} has a unique solution $(\femvec{v}{h}{n}, \tilde{p}_h^n, \femvec{B}{h}{n}, \femvec{m}{h}{n})$ at time step $n$.

Furthermore, it follows from Lemmas \ref{lem:inductive-step-velocity}--\ref{lem:inductive-step-magnetisation-H1} that there exists $c(\varepsilon) > 0$ and $C > 0$ (which may depend only on the exact solution and the domain $\Omega$) satisfying
\begin{equation*}
    P^i - P^{i-1} + C_{\Omega} Q^i \leq c(\varepsilon) \tau R + c(\varepsilon) \tau P^{i-1} + c(\varepsilon) \tau P^i + C(\varepsilon + h^{1/2}) \tau Q^i,
\end{equation*}
for all $1 \leq i \leq n$. Observe that the case when $i=1$ was established in the arguments for the base case. By choosing $\varepsilon > 0$ and $h_0$ sufficiently small, we derive
\begin{equation*}
    P^i - P^{i-1} + C_{\Omega} \tau Q^i \leq C \tau R + C \tau P^{i-1} + C \tau P^i,
\end{equation*}
for all $1 \leq i \leq n$, where the positive constant $C$ on the right-hand side is independent of $i$ and $n$. Summing these inequalities from $i=1$ to $i=n$ we obtain (recalling that $P^0 = 0$)
\begin{align*}
    P^n + \tau \sum_{i=1}^n Q^i &\leq C \tau n R + C \tau \sum_{i=1}^n P^{i-1} + C \tau \sum_{i=1}^n P^i \\
    &= C n \tau R + C \tau P^n + C \tau \sum_{i=1}^{n-1} P^i \\
    &\leq C T R + C \tau P^n + C \tau \sum_{i=1}^{n-1} P^i, 
\end{align*}
where $1 \leq n \leq N$. Thus, for $\tau \in (0,\tau_0]$ with sufficiently small $\tau_0 > 0$, we have
\begin{equation} \label{eqn:thm-proof_3}
    P^n + \tau \sum_{i=1}^n Q^i \leq C T R + C \tau \sum_{i=1}^{n-1} P^i.
\end{equation}
By the discrete Gronwall inequality,
\begin{equation*}
    P^n \leq C T R \exp \left(\sum_{i=1}^{n-1} C \tau\right) \leq C T \exp(C T) R = \tilde{C}_0 R,
\end{equation*}
where $\tilde{C}_0 := C T \exp(C T)$, for all $1 \leq n \leq N$. Observe that $\tilde{C}_0$ is independent of $n$. Inserting this upper bound of $P^n$ into the right-hand side of \eqref{eqn:thm-proof_3} we obtain
\begin{equation*}
    P^n + \tau \sum_{i=1}^n Q^i \leq C T R + C \tau \sum_{i=1}^{n-1} \tilde{C}_0 R \leq C T R + C \tilde{C}_0 T R = \hat{C}_0 R,
\end{equation*}
where $\hat{C}_0 := CT + C \tilde{C}_0 T$, for all $1 \leq n \leq N$. Observe that $\hat{C}_0$ is also independent of $n$. Recalling the positive constant $C_0$ in the base case, we may choose $A$ such that
\begin{equation*}
    A = \max \{C_0, \hat{C}_0\}
\end{equation*}
for sufficiently small $h_0$ such that $h_0 < A^{-1/\beta}$. The proof by induction of \eqref{eqn:induction-error-estimate} is complete once we apply \eqref{eqn:tau-growth}. \\

\noindent \textbf{Error Estimate.} Finally, we complete the proof of Theorem \ref{thm:main-result} by proving the error estimate \eqref{eqn:error-estimate}. Observe that \eqref{eqn:error-estimate} is just the error estimate for the $\vec{\delta}^n$-terms which is composed of $\vec{E}^n$-terms and $\vec{e}^n$-terms, see \eqref{equ:error split}. The error estimate for the $\vec{E}^n$-terms follows from Lemmas \ref{lem:L2-proj-estimates}--\ref{lem:maxwell-proj-estimates} while the error estimate for $\vec{e}^n$-terms follows from \eqref{eqn:induction-error-estimate}. Hence, by combining these results and taking the square root, we obtain \eqref{eqn:error-estimate}. Next, we show that $\femvec{m}{h}{n}$ satisfies the constraint $|\femvec{m}{h}{n}| = 1$ asymptotically. Using Lemma \ref{lem:induction-assumptions-solution-estimates} we have
\begin{align*}
    \lnorm{1 - |\femvec{m}{h}{n}|^2}{2} &= \lnorm{|\vec{m}^n|^2 - |\femvec{m}{h}{n}|^2}{2} \\
    &\leq \left(\lnorm{\vec{m}^n}{\infty} + \lnorm{\femvec{m}{h }{n}}{\infty}\right) \lnorm{\vec{m}^n - \femvec{m}{h}{n}}{2} \\
    &\leq C \lnorm{\vec{m}^n - \femvec{m}{h}{n}}{2} \\
    &\leq C( \tau + h^{\ell - \beta} + h^{k - \beta} + h^{r - \beta}) \to 0.
\end{align*}
This completes the proof of Theorem \ref{thm:main-result}.

\section{Numerical experiments} \label{sec:numerical-experiments}

Numerical experiments were conducted for a two-fold purpose, namely to observe the convergence of the numerical method visually and the order of convergence numerically. We carry out these experiments in the four different scenarios, with two experiments for each purpose.

In the first scenario, we visually confirm the convergence of the numerical method by solving the linear system \eqref{eqn:discrete-velocity-formulation}--\eqref{eqn:discrete-magnetisation-formulation} over a sphere of radius $1/2$ with $T = 1$ and $\tau = h$, i.e., $\beta = 0$. Having chosen the exact solutions
\begin{subequations}
	\label{eqn:exact-soln-1}
\begin{alignat}{1}
\vec{v}(t, \vec{x}) &= e^t \sin \left(4 \pi |\vec{x}|^2\right) (x_1, -x_0, 0)^{\top}, \label{eqn:exact-soln-1a} \\
p(t,\vec{x}) &= e^t x_0 x_1 x_2, \label{eqn:exact-soln-1b} \\
\vec{B}(t,\vec{x}) &= e^t \sin^2\left(4 \pi |\vec{x}|^2\right) \cos\left(4 \pi |\vec{x}|^2\right) (x_1, -x_0, 0)^{\top}, \label{eqn:exact-soln-1c}\\
\vec{m}(t,\vec{x}) &= (\cos t, 0, \sin t), \label{eqn:exact-soln-1d}
\end{alignat}
\end{subequations}
observe that the divergence-free conditions, the boundary conditions, and the constant magnitude condition $|\vec{m}| = 1$ are satisfied. For this scenario, we use quadratic finite elements (i.e., $\ell = k = r = 2$), quadratic isoparametric elements for the boundary, and we set the initial data to be the values of the exact solutions at $t = 0$. 

Figure \ref{fig:velocity_sphere} demonstrates that the velocity vector field at $t=1$ remains stable as $\tau, h \to 0$.
\begin{figure}
    \centering
    \subfloat[$h=1/8$]{\includegraphics[scale=0.1]{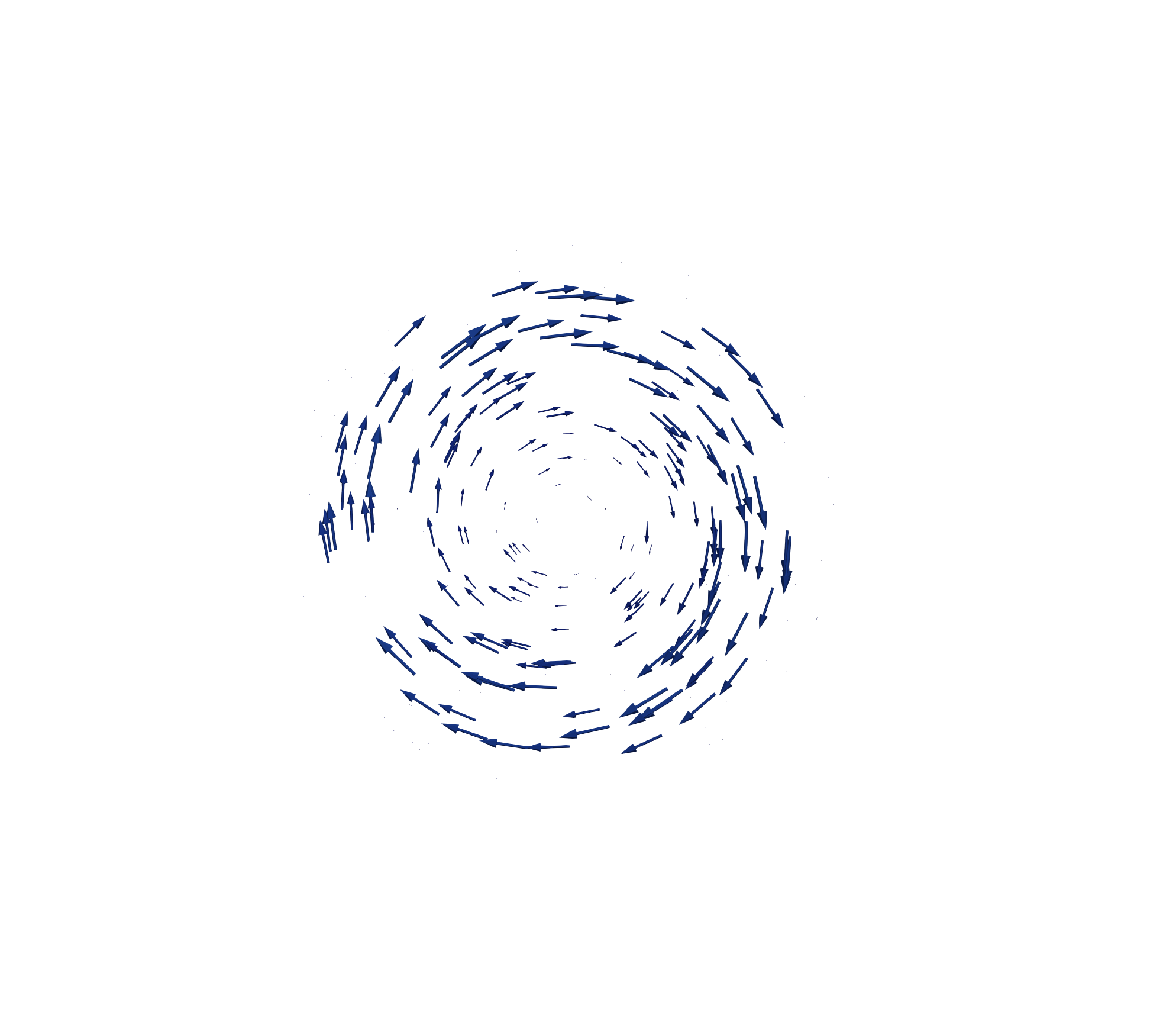}} 
    \subfloat[$h=1/12$]{\includegraphics[scale=0.1]{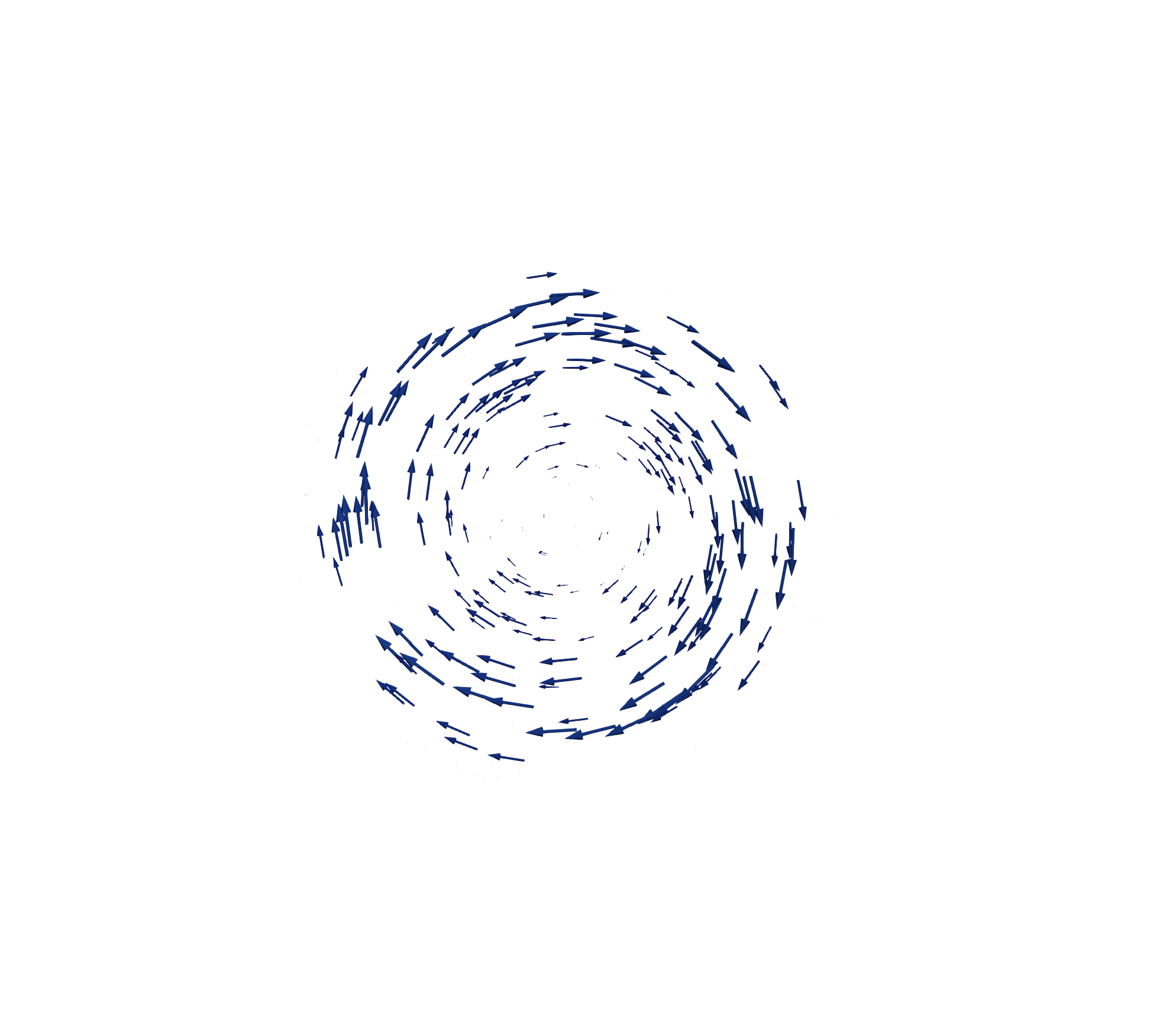}}\\
    \subfloat[$h = 1/16$]{\includegraphics[scale=0.1]{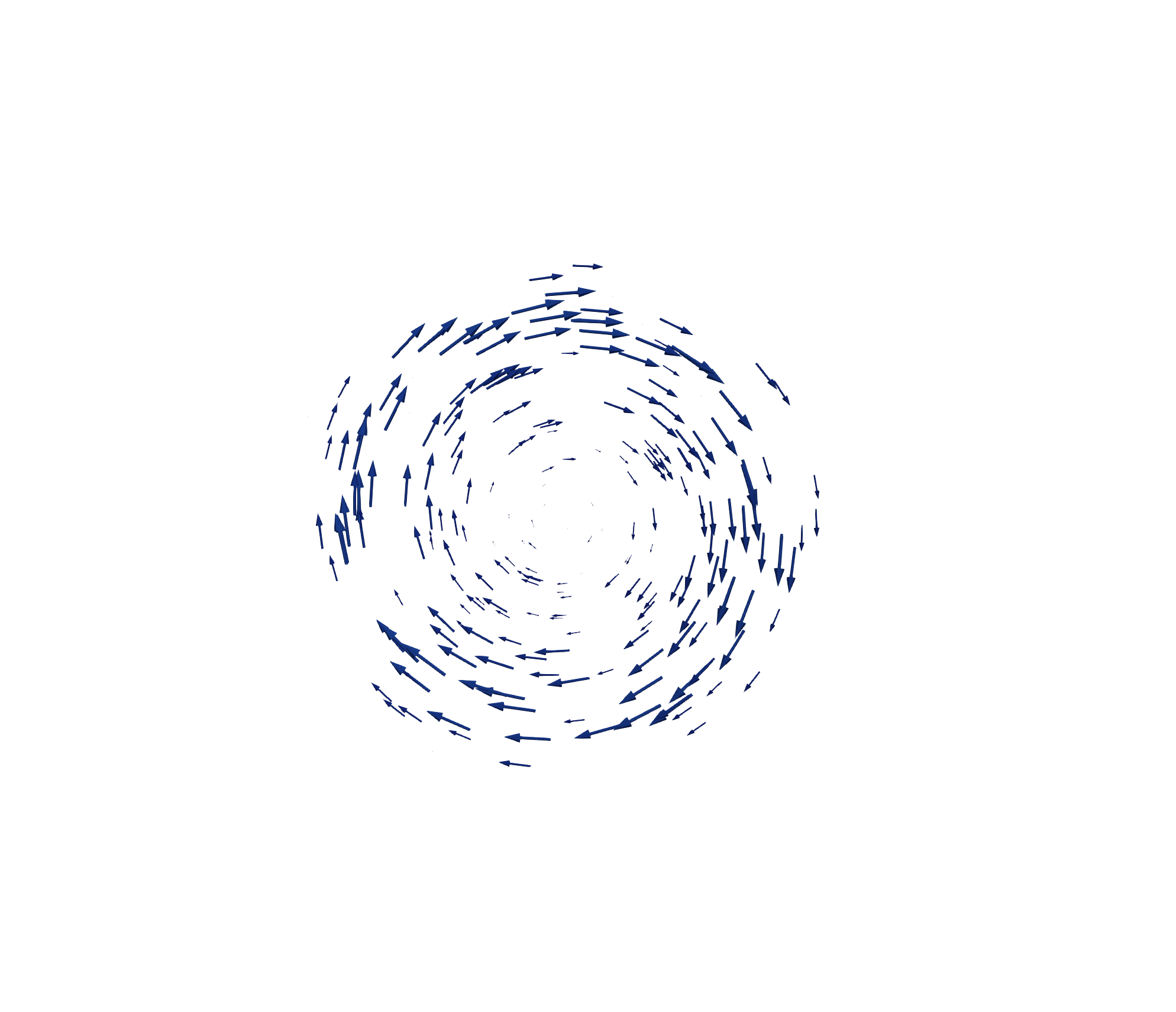}}
    \subfloat[$h = 1/20$]{\includegraphics[scale=0.1]{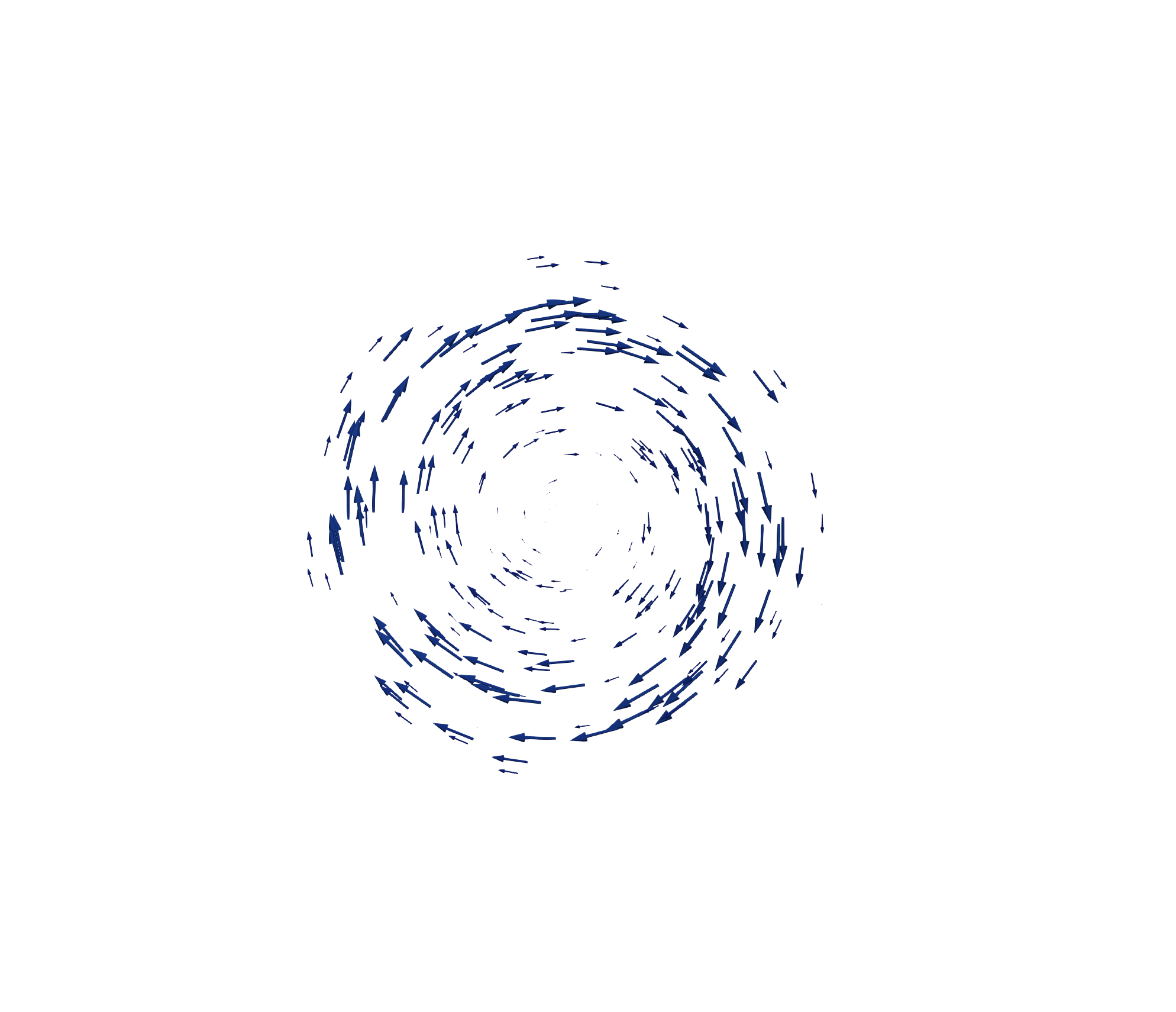}} 
    \caption{Velocity vector field at $t = 1$.}
    \label{fig:velocity_sphere}
\end{figure}

In the next scenario, we visually confirm the convergence of the numerical method over the unit cube by observing the magnetic field instead. We use a new set of exact solutions that account for the boundary conditions on the new domain, and these exact solutions take the form
\begin{subequations}
\begin{alignat}{1}
\vec{v}(t,\vec{x}) &= 2 \pi e^t \left(\sin ^2(2 \pi  x) \sin (2 \pi  y) \cos (2 \pi  y) \sin ^2(2 \pi  z),-\sin (2 \pi  x) \cos (2 \pi  x) \sin ^2(2 \pi  y) \sin ^2(2 \pi  z),0\right)^{\top}, \label{eqn:exact-soln-2a} \\
    p(t,\vec{x}) &= e^t \sin (2 \pi  x) \sin (2 \pi  y) \sin (2 \pi  z), \label{eqn:exact-soln-2b} \\
    \vec{B}(t,\vec{x}) &= 2\pi e^t \left( \sin ^2(\pi  x) \sin (\pi  y) \cos (\pi  y) \sin ^2(\pi  z),- \sin (\pi  x) \cos (\pi  x) \sin ^2(\pi  y) \sin ^2(\pi  z),0\right)^{\top} \label{eqn:exact-soln-2c} \\
    \vec{m}(t,\vec{x}) &= (\cos t, 0, \sin t)^{\top}. \label{eqn:exact-soln-2d}
\end{alignat}
\end{subequations}

The results are given in Figure \ref{fig:magnetic_field_cube}, where we have set
$t = 1/2$ and have let $\tau, h \to 0$. Figure \ref{fig:magnetic_field_cube} is of the same nature as Figure \ref{fig:velocity_sphere} and shows the stability of the magnetic field as $\tau, h \to 0$.

\begin{figure}
    \centering
    \subfloat[$h=1/8$]{\includegraphics[scale=0.1]{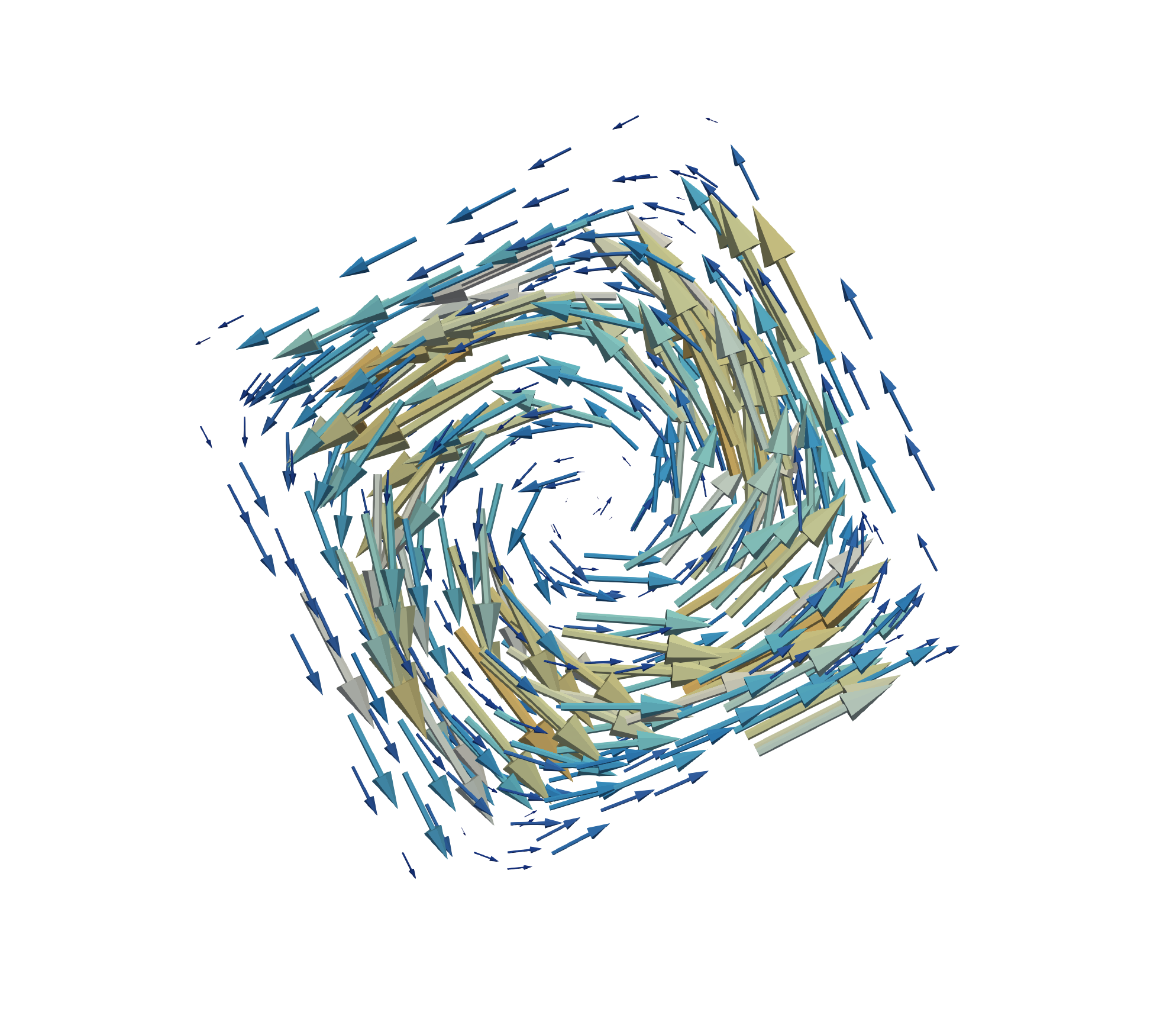}} 
    \subfloat[$h=1/12$]{\includegraphics[scale=0.1]{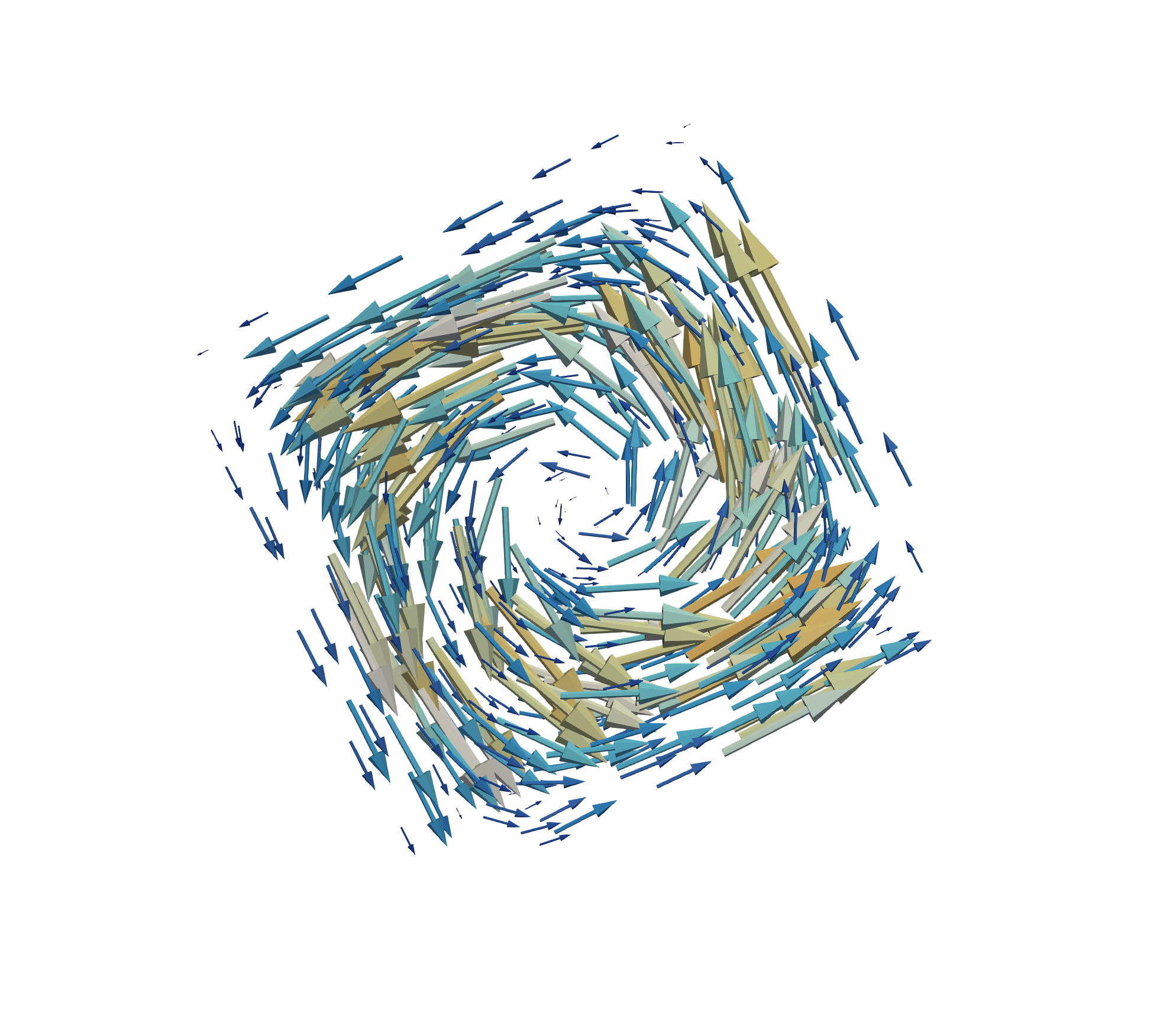}}\\
    \subfloat[$h = 1/16$]{\includegraphics[scale=0.1]{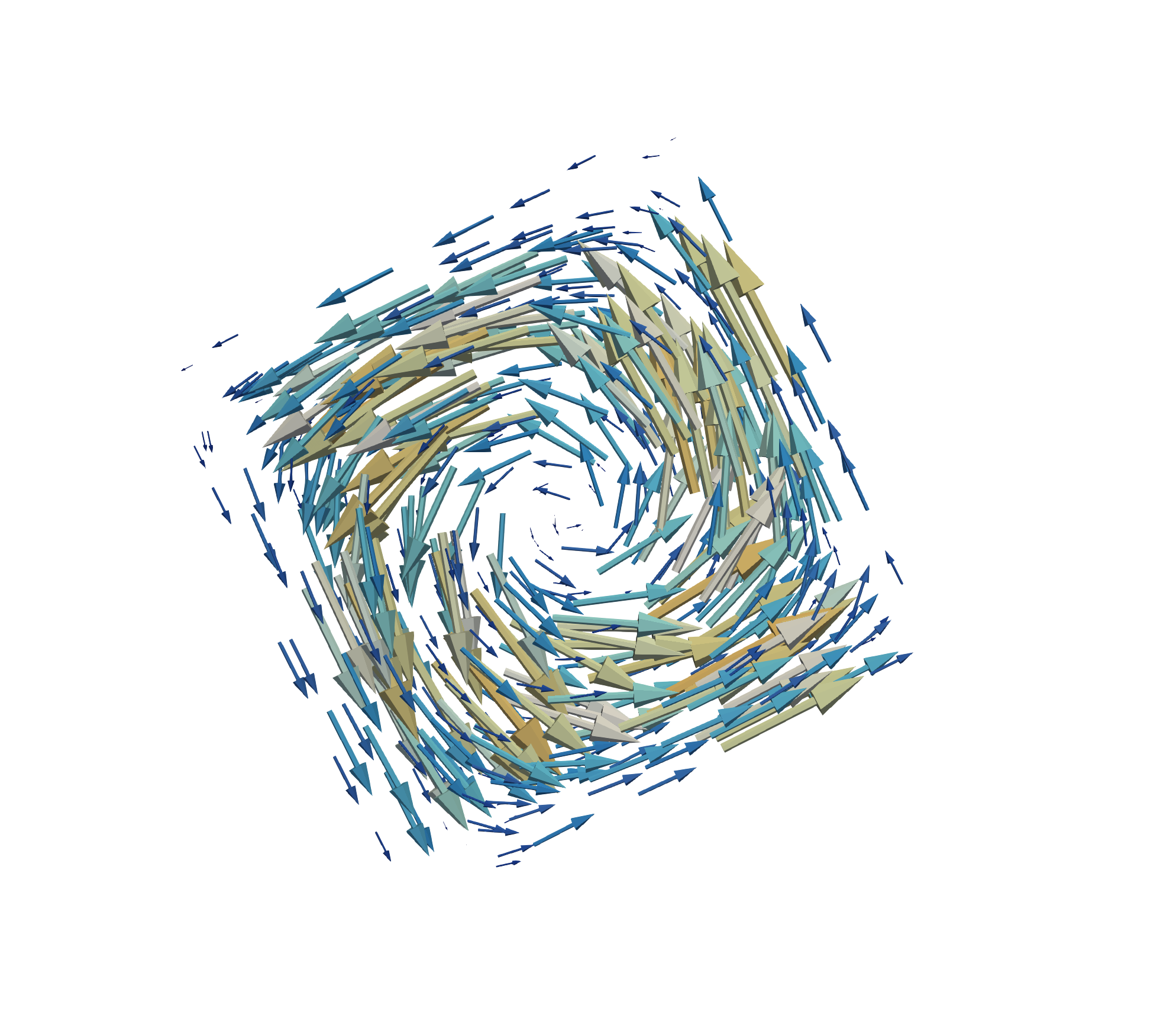}}
    \subfloat[$h = 1/20$]{\includegraphics[scale=0.1]{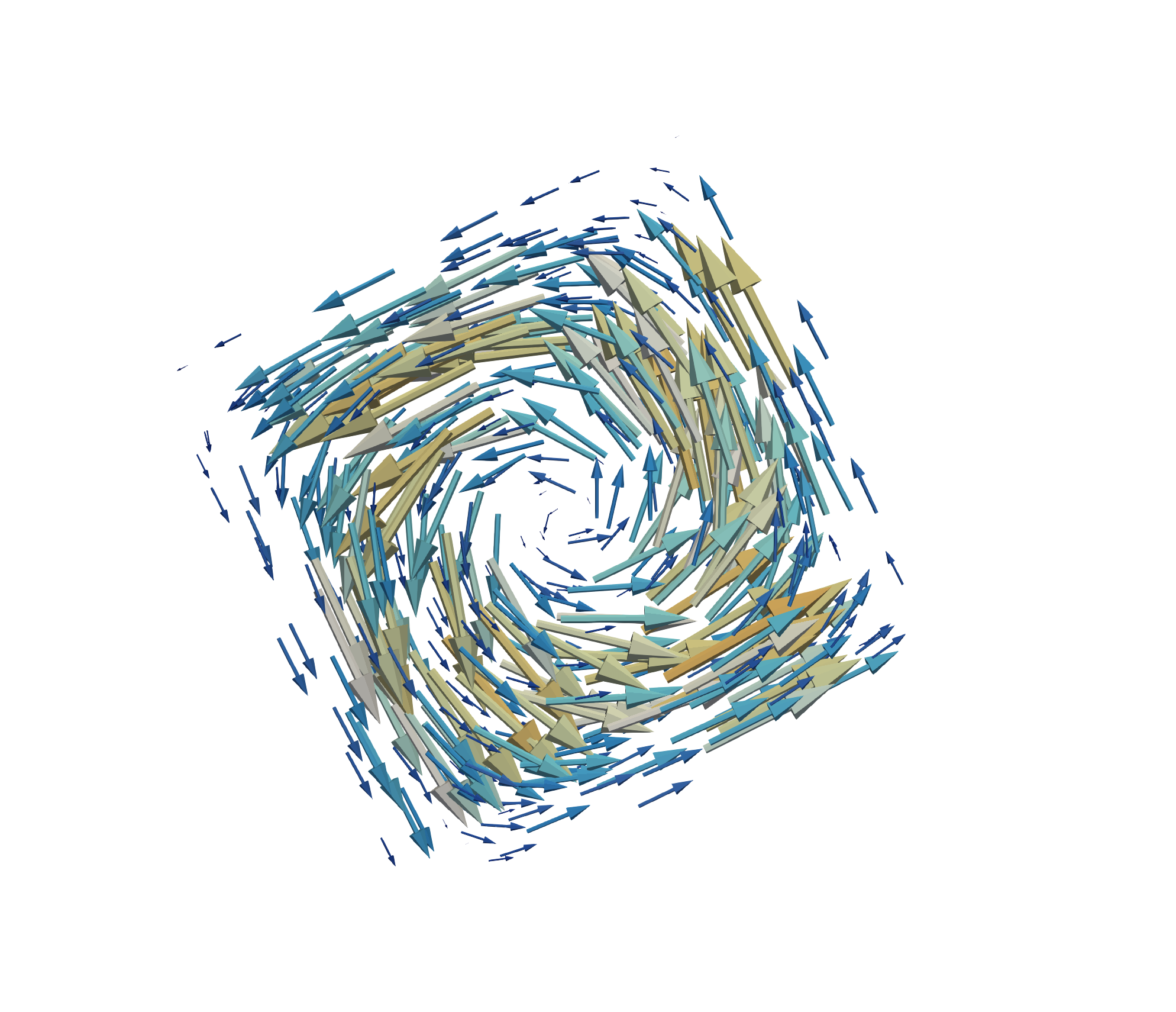}} 
    \caption{Magnetic field vector field at $t = 1/2$.}
    \label{fig:magnetic_field_cube}
\end{figure}

In the next two scenarios we calculate the order of convergence for this
numerical method. The linear system
\eqref{eqn:discrete-velocity-formulation}--\eqref{eqn:discrete-magnetisation-formulation}
is solved exclusively on the sphere of radius $1/2$ centered at the origin, once
again using quadratic isoparametric elements for the boundary of the sphere. In
these scenarios, we work exclusively with the exact solution \eqref{eqn:exact-soln-1}.

Having set the time step condition $\tau = h$ with $\ell = k = r = 2$, in the optimal case, the error estimate should converge like $\tau + h^{\ell+1} + h^{k+1} + h^r = h + 2h^3 + h^2 = O(h)$. Figure \ref{fig:rate-Oh} demonstrates that we obtain an order of convergence of at least $O(h)$. To be precise, we observe a super-optimal order of convergence which we will discuss soon enough.

Next, by setting $\tau = h^3$ and using cubic finite elements for the
magnetisation, i.e., we set $r=3$, we expect the optimal order of convergence
$O(h^3)$. Figure \ref{fig:rate-Oh3} seems to support our theoretical results.

The slightly better numerical convergence rates in both Figure~\ref{fig:rate-Oh} and Figure~\ref{fig:rate-Oh3} may be due to the choice of our exact solutions. They are scaled by a factor of $e^t$ as $t$ grows, i.e., there is no spatial variation as time moves forward. Moreover, the exact solution for $\vec{m}$ lacks for dependence on the spatial variable. It is also noted that assumption~\eqref{eqn:tau-growth} may be just a technical requirement of our
proofs.

\begin{figure}
    \centering
    \includegraphics[width=0.5\linewidth]{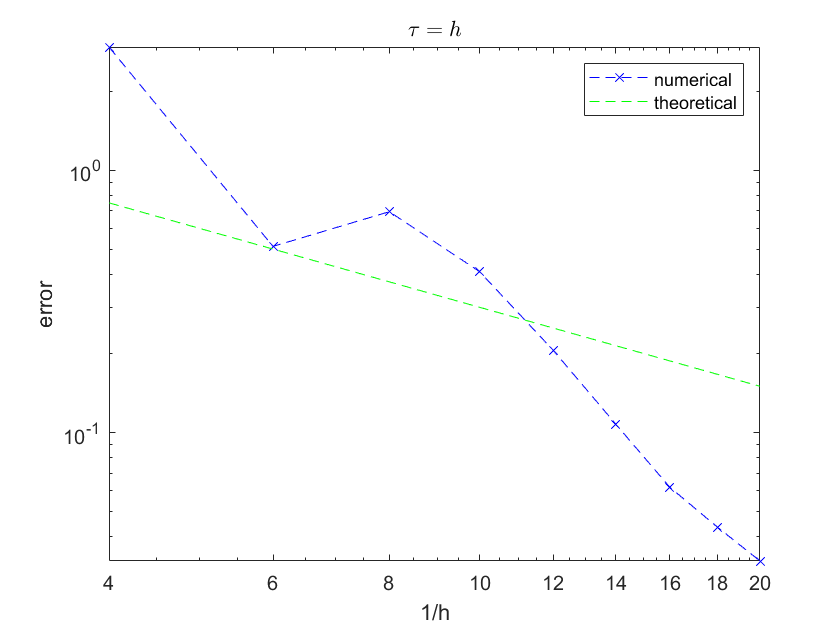}
    \caption{Rate of convergence for $\tau = h$}
    \label{fig:rate-Oh}
\end{figure}

\begin{figure}
    \centering
    \includegraphics[width=0.5\linewidth]{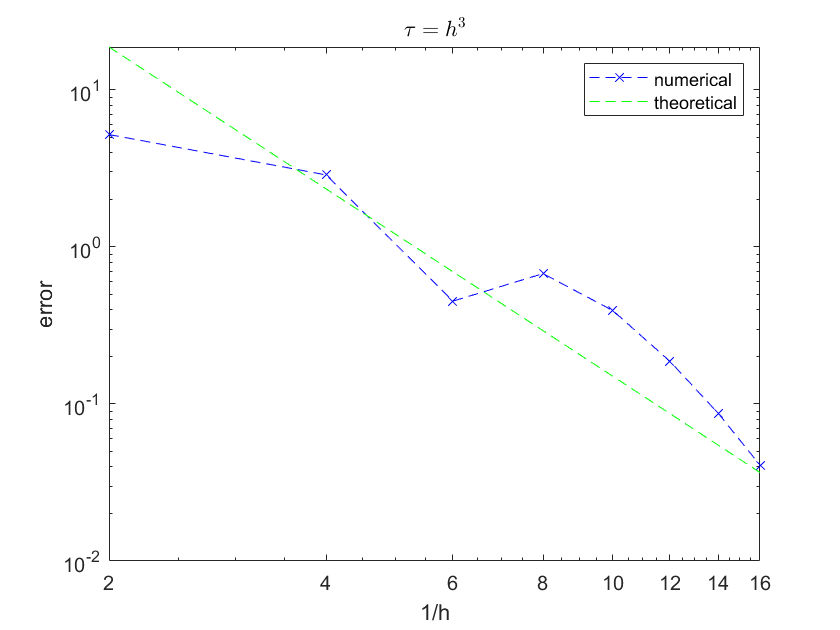}
    \caption{Rate of convergence for $\tau = h^3$}
    \label{fig:rate-Oh3}
\end{figure}

\section{Appendix} \label{sec:app}

In this section, we prove some results which are required in the proof
of the main result. In the following pages, $c(\varepsilon)$ denotes a positive constant that depends on $\varepsilon$, and may also depend on the domain and the initial data. This constant $c(\varepsilon)$ may take different values at different occurrences.

\subsection{Useful results} \label{sec:useful-results}

\begin{lemma} \label{lem:convective}
For any $\vec{\psi}, \vec{\varphi} \in \HB^1(\Omega)$ and $\vec{u} \in \HB_{\div}^1(\Omega)$, we have
\begin{equation*}
    \liprod{(\vec{u} \cdot \nabla) \vec{\psi}}{\vec{\varphi}} = - \liprod{(\vec{u} \cdot \nabla) \vec{\varphi}}{\vec{\psi}} + \langle (\vec{\psi} \vec{u}^{\top}) \vec{n}, \vec{\varphi}\rangle_{\LB^2(\partial\Omega)}.
\end{equation*}
\end{lemma}
\begin{proof}
First, note that
\begin{align*}
    (\vec{u} \cdot \nabla) \vec{\psi} = \div \left(\vec{\psi} \vec{u}^{\top}\right).
\end{align*}
Hence, integrating by parts we obtain
\begin{align*}
    \liprod{(\vec{u} \cdot \nabla) \vec{\psi}}{\vec{\varphi}} &= -\liprod{\vec{\psi} \vec{u}^{\top}}{\nabla \vec{\varphi}} + \left \langle \left(\vec{\psi} \vec{u}^{\top}\right) \vec{n}, \vec{\varphi} \right \rangle_{\LB^2(\partial\Omega)}.
\end{align*}
A simple calculation reveals that
\begin{equation*}
    \left(\vec{\psi} \vec{u}^{\top}\right) \cdot \nabla \vec{\varphi} = [(\vec{u} \cdot \nabla)\vec{\varphi}] \cdot \vec{\psi},
\end{equation*}
proving the lemma.
\end{proof}

\begin{lemma} \label{lem:curl-estimate}
	There exists a positive constant $C_{\Omega}$ depending only on $\Omega$
	such that for any $\vec{u} \in \HB_n^1(\Omega)$, we have
\begin{equation*}
    \hnorm{\vec{u}}{1} 
    \le 
    C_{\Omega} \big(\lnorm{\curl \vec{u}}{2}^2 + \lnorms{\div
    \vec{u}}{2}^2\big).
\end{equation*}
\end{lemma}
\begin{proof}
This result follows from \cite[Theorem 2.17]{amrouche1998-article} and \cite[Lemma 3.6, page 53]{girault1986-book}.
\end{proof}

\begin{lemma}
For any $\vec{u}_h \in \V_h^m(\Omega)$ we have
\begin{align} \label{eqn:inverse-estimates}
    \lnorm{\vec{u}_h}{\infty} + \wnorm{\vec{u}_h}{1}{3} &\lsim h^{-1/2} \hnorm{\vec{u}_h}{1},
\end{align}
where the constant may depend on $m$ and $\Omega$.
\end{lemma}
\begin{proof}
See, for example, \cite[Theorem 4.5.11]{brenner2007-book}.
\end{proof}


The projections defined in Subsection \ref{sec:fem-spaces} have the following approximation and stability properties.


\begin{lemma}[$\LB^2$ Projection] \label{lem:L2-proj-estimates}
There exists $h_0>0$ such that for any $\vec{u} \in \HB^{s}(\Omega)$, $1 \le s \le r+1$, and
for any $h \in (0, h_0]$ the $\LB^2$-orthogonal projection $\Lproj \vec{u} \in \V_{\vec{m}}$ of
$\vec{u}$ satisfies the approximation property
\begin{align*}
    \lnorm{\vec{u} - \Lproj \vec{u}}{2} + h \hnorm{\vec{u} - \Lproj \vec{u}}{1}
    \lsim h^{s} \hnorm{\vec{u}}{s}.
\end{align*}
\end{lemma}

\begin{proof}
The proof can be found in \cite{thomee2006-book}.
\end{proof}


\begin{lemma}[Ritz Projection] \label{lem:elliptic-proj-estimates}
There exists $h_0>0$ such that if $\vec{u} \in \W^{s,p}(\Omega)$, where $1 \le s \le r+1$ and $2 \leq p < \infty$, then for any $h \in (0,h_0]$ the Ritz projection $\ellipticProj \vec{u} \in \V_{\vec{m}}$ of $\vec{u}$ satisfies the approximation and stability properties
\begin{align}
    \lnorm{\vec{u} - \ellipticProj \vec{u}}{p} + h \wnorm{\vec{u} - \ellipticProj \vec{u}}{1}{p} &\lsim h^{s} \wnorm{\vec{u}}{s}{p}. \label{eqn:ritz-proj-estimates-approximation} 
\end{align}
In particular, if $\vec{u}\in\HB^2(\Omega)$, then
\begin{equation}
    \lnorm{\ellipticProj \vec{u}}{\infty} + \wnorm{\ellipticProj \vec{u}}{1}{3} \lsim \hnorm{\vec{u}}{2}. \label{eqn:ritz-proj-estimates-stability}
\end{equation}
Moreover, the Ritz projection satisfies the identity $\Delta_h \ellipticProj = \Lproj \Delta$, where $\Delta_h$ is the discrete Laplacian defined by \eqref{eqn:discrete-laplacian}.
\end{lemma}
\begin{proof}
See, e.g., \cite{brenner2007-book}.
\end{proof}

\begin{lemma}[Stokes projection] \label{lem:stokes-proj-estimates}
Let $2 \le s \le \ell+1$. There exists a sufficiently small $h_0 > 0$ such that for all $h \in (0,h_0]$, $\vec{u} \in \HB^{s}(\Omega) \cap \HB_0^1(\Omega)$, and $q \in H^{s-1}(\Omega) \cap L_0^2(\Omega)$, if $(\hat{\vec{u}}_h, \hat{q}_h) \in \V_{\vec{v}}\times V_p$ is the Stokes projection of $(\vec{u},q)$ then
\begin{align}
    \lnorm{\vec{u} - \hat{\vec{u}}_h}{2} + h \hnorm{\vec{u} - \hat{\vec{u}}_h}{1} + \lnorm{q - \hat{q}_h}{2} &\lsim h^{s} \left(\hnorm{\vec{u}}{s} + \|q\|_{H^{s-1}(\Omega)}\right), \label{eqn:stokes-proj-estimates-approximation} \\
    \lnorm{\hat{\vec{u}}_h}{\infty} + \wnorm{\hat{\vec{u}}_h}{1}{3} &\lsim \hnorm{\vec{u}}{2} + \|q\|_{H^1(\Omega)}. \label{eqn:stokes-proj-estimates-stability}
\end{align}
\end{lemma}
\begin{proof}
Estimate \eqref{eqn:stokes-proj-estimates-approximation} follows from
\cite[Theorems 1.8--1.9, pp. 125--127]{girault1986-book} since Taylor-Hood
elements have the required approximation properties via the interpolation
operator and satisfy the inf-sup condition.
We briefly prove \eqref{eqn:stokes-proj-estimates-stability}.
Let $I_h\vec{u}\in\V_{\vec{v}}$ be the Lagrange interpolation of $\vec{u}$.
By using the triangle inequality,
\eqref{eqn:inverse-estimates},
\eqref{eqn:stokes-proj-estimates-approximation} with $s=2$, \cite[Corollary
4.4.24]{brenner2007-book}, and the embeddings
$\HB^2(\Omega) \hookrightarrow \W^{1,6}(\Omega) \hookrightarrow \W^{1,3}(\Omega)$ and $\HB^2(\Omega) \hookrightarrow \LB^\infty(\Omega)$, we have
\begin{align*}
	\lnorm{\hat{\vec{u}}_h}{\infty} 
	+
	\wnorm{\hat{\vec{u}}_h}{1}{3}
    	&\leq 
	\lnorm{\hat{\vec{u}}_h - I_h \vec{u}}{\infty} 
	+
	\wnorm{\hat{\vec{u}}_h - I_h \vec{u}}{1}{3}
	\\
	&\quad
	+
	\lnorm{I_h \vec{u}-\vec{u}}{\infty} 
	+
	\wnorm{I_h \vec{u}-\vec{u}}{1}{3}
	+
	\lnorm{\vec{u}}{\infty} 
	+
	\wnorm{\vec{u}}{1}{3}
    	\\
    	&\lsim 
    	h^{-1/2} 
	\left(
	\hnorm{\hat{\vec{u}}_h - \vec{u}}{1}  
	+
	\hnorm{\vec{u} - I_h\vec{u}}{1}  
	\right)
	\\
	&\quad
	+
	\lnorm{I_h \vec{u}-\vec{u}}{\infty} 
	+
	\wnorm{I_h \vec{u}-\vec{u}}{1}{3}
	+
	\lnorm{\vec{u}}{\infty} 
	+
	\wnorm{\vec{u}}{1}{3}
		\\
    	&\lsim 
    	h^{1/2} 
	\left(
	\hnorm{\vec{u}}{2}  
	+
	\hnorm{q}{1}  
	\right)
	+
	\hnorm{\vec{u}}{2}
	\\
	&\lsim
	\hnorm{\vec{u}}{2}
	+
	\hnorm{q}{1},
\end{align*}
completing the proof of the lemma.
\end{proof}

The Maxwell projection admits the following approximation and stability properties, which are well-known (see, e.g., \cite{ravindran2016-article}).

\begin{lemma}[Maxwell projection] \label{lem:maxwell-proj-estimates}
Let $1 \le s \le k+1$. There exists $h_0 > 0$ such that if $\vec{u} \in \HB^{s}(\Omega) \cap \HB_{n,\curl}^1(\Omega)$ then for any $h \in (0,h_0]$ the Maxwell projection $\maxwellProj \vec{u} \in \V_{\vec{B}}$ of $\vec{u}$ satisfies the approximation and stability properties
\begin{align}
    \lnorm{\vec{u} - \maxwellProj \vec{u}}{2} + h\hnorm{\vec{u} - \maxwellProj
    \vec{u}}{1} &\lsim h^{s} \hnorm{\vec{u}}{s}, \label{eqn:maxwell-proj-estimates-approximation} \\
    \lnorm{\maxwellProj \vec{u}}{\infty} + \wnorm{\maxwellProj \vec{u}}{1}{3} &\lsim \hnorm{\vec{u}}{2}. \label{eqn:maxwell-proj-estimates-stability}
\end{align}
\end{lemma}
\begin{proof}
	The results can be obtained in the same manner as in the proof of  Lemma
	\ref{lem:stokes-proj-estimates}.
\end{proof}

\subsection{Technical Estimates for the Proof of Theorem \ref{thm:main-result}}

We first state a lemma that estimates the error that arises from the time
discretisation of the exact form of the FMHD PDEs.

\begin{lemma} \label{lem:truncation-estimates}
For every $\varepsilon > 0$ there exists $c(\varepsilon) > 0$ such that the truncation terms $R_1(\vec{\varphi}_h)$, $R_2(\vec{\omega}_h)$ and $R_3(\vec{\xi}_h)$, defined respectively by \eqref{eqn:exact-velocity-discrete-formulation}, \eqref{eqn:exact-magnetic-field-discrete-formulation} and \eqref{eqn:exact-magnetisation-discrete-formulation}, satisfy 
\begin{align*}
    |R_1(\vec{\varphi}_h)| &\leq c(\varepsilon) \tau^2 + \varepsilon \lnorm{\vec{\varphi}_h}{2}^2 \qquad \forall \vec{\varphi}_h \in \V_{\vec{v}},\\
    |R_2(\vec{\omega}_h)| &\leq c(\varepsilon) \tau^2 + \varepsilon \lnorm{\vec{\omega}_h}{2}^2 \qquad \forall \vec{\omega}_h \in \V_{\vec{B}},\\
    |R_3(\vec{\xi}_h)| &\leq c(\varepsilon) \tau^2 + \varepsilon \lnorm{\vec{\xi}_h}{2}^2 \qquad \forall \vec{\xi}_h \in \V_{\vec{m}},
\end{align*}
where the constant $c(\varepsilon)$ may depend only on the domain and the initial data.
\end{lemma}
\begin{proof}
The truncation term $R_1(\vec{\varphi}_h)$ satisfies
\begin{align*}
    R_1(\vec{\varphi}_h) &= \liprod{\discreteDtau \vec{v}^n - \vec{v}_t^n}{\vec{\varphi}_h} - \frac{\tau}{2} \liprod{(\discreteDtau \vec{v}^n \cdot \nabla)\vec{v}^n}{\vec{\varphi}_h} + \frac{\tau}{2} \liprod{(\discreteDtau \vec{v}^n \cdot \nabla)\vec{\varphi}_h}{\vec{v}^n} \\
    &\qquad + \tau \liprod{\curl \vec{B}^n \times \discreteDtau \vec{B}^n}{\vec{\varphi}_h} + \tau \liprod{(\discreteDtau \vec{B}^n \cdot \nabla) \vec{m}^n}{\vec{\varphi}_h} \\
    &\qquad - \tau \liprod{\nabla \discreteDtau \vec{m}^n \odot \Delta \vec{m}^n}{\vec{\varphi}_h}.
\end{align*}
Hence, using Lemma \ref{lem:convective} for the third term,
\begin{align*}
    |R_1(\vec{\varphi}_h)| &\lsim \lnorm{\discreteDtau \vec{v}^n - \vec{v}_t^n}{2} \lnorm{\vec{\varphi}_h}{2} + \tau \lnorm{\nabla \vec{v}^n}{\infty} \lnorm{\discreteDtau \vec{v}^n}{2} \lnorm{\vec{\varphi}_h}{2} \\
    &\qquad + \tau \lnorm{\vec{v}^n}{\infty} \lnorm{\nabla \discreteDtau \vec{v}^n}{2} \lnorm{\vec{\varphi}_h}{2} + \tau \lnorm{\curl \vec{B}^n}{\infty} \lnorm{\discreteDtau \vec{B}^n}{2} \lnorm{\vec{\varphi}_h}{2} \\
    &\qquad + \tau \lnorm{\nabla \vec{m}^n}{\infty} \lnorm{\discreteDtau \vec{B}^n}{2} \lnorm{\vec{\varphi}_h}{2} + \tau \lnorm{\Delta \vec{m}^n}{\infty} \lnorm{\nabla \discreteDtau \vec{m}^n}{2} \lnorm{\vec{\varphi}_h}{2} \\
    &\leq c(\varepsilon) \tau^2 + \varepsilon \lnorm{\vec{\varphi}_h}{2}^2.
\end{align*}
The remaining estimates can be proved similarly.
\end{proof}

The next six lemmas are used to prove estimate
\eqref{eqn:induction-error-estimate}. We remind that the constants $A$ and $h_0$
in the statements and proofs of these lemmas are stated
in~\eqref{eqn:induction-error-estimate}.

\begin{lemma} \label{lem:induction-assumptions-error-estimates} 
Assume that \eqref{eqn:tau-growth} holds. Assume further that the inductive assumption \eqref{eqn:inductive-assumption} holds. Then for $0 \leq i \leq n-1$ we have
\begin{align} \label{eqn:induction-estimate-linf-w13}
    \lnorm{\femvec{\delta}{\vec{m}}{i}}{\infty} + \wnorm{\femvec{\delta}{\vec{m}}{i}}{1}{3} &\leq C h^{1/2}
    \quad \forall h\in(0,h_0].
\end{align}
Here, the constant $C$ may depend only on the exact solution and the domain $\Omega$.
\end{lemma}
\begin{proof}
We begin with the case where $i \geq 1$. Choosing $h_0 < A^{-1/\beta}$ and recalling definition \eqref{equ:error split}, we use successively the triangle
inequality, the Sobolev embeddings 
$\W^{1,6}(\Omega) \hookrightarrow \LB^{\infty}(\Omega)$
and
$\W^{1,6}(\Omega) \hookrightarrow \W^{1,3}(\Omega)$,
\eqref{eqn:inverse-estimates}, \eqref{eqn:ritz-proj-estimates-approximation},
\eqref{eqn:inductive-assumption}, and \eqref{eqn:tau-growth} to deduce that
\begin{align*}
    	\lnorm{\femvec{\delta}{\vec{m}}{i}}{\infty} 
	+
	\wnorm{\femvec{\delta}{\vec{m}}{i}}{1}{3} 
	&\leq 
	\lnorm{\femvec{E}{\vec{m}}{i}}{\infty} 
	+
	\wnorm{\femvec{E}{\vec{m}}{i}}{1}{3} 
	+
	\lnorm{\femvec{e}{\vec{m}}{i}}{\infty} 
	+
	\wnorm{\femvec{e}{\vec{m}}{i}}{1}{3} 
	\\
    	&\leq 
	C \wnorm{\femvec{E}{\vec{m}}{i}}{1}{6} 
	+
	C  h^{-1/2}\hnorm{\femvec{e}{\vec{m}}{i}}{1} 
	\\
    	&\leq 
	C h \wnorm{\vec{m}^i}{2}{6} 
	+
	C h^{-1/2} A^{1/2}(\tau + h^{2-\beta}) 
	\\
    	&\leq 
	C h \|\vec{m}\|_{C([0,T];\W^{2,6}(\Omega))}
	+ 
	C A^{1/2} h^{1/2 + \beta} 
	\\
    	&\leq C h^{1/2}
	\quad \forall h > 0, \ h \le h_0 < A^{-1/\beta},
\end{align*}
since $\beta \in (0,1/2)$. Of course, in the case where $i=0$ the result
follows trivially since $\femvec{e}{\vec{m}}{i} = \vec{0}$.
\end{proof}

\begin{lemma} \label{lem:induction-assumptions-solution-estimates}
Assume that \eqref{eqn:tau-growth} holds. Assume further that the inductive assumption \eqref{eqn:inductive-assumption} holds. Then for $0 \leq i \leq n-1$ we have
\begin{equation*}
    \left(\lnorm{\femvec{v}{h}{i}}{\infty} + \lnorm{\femvec{B}{h}{i}}{\infty} + \lnorm{\femvec{m}{h}{i}}{\infty}\right) + \left(\wnorm{\femvec{v}{h}{i}}{1}{3} + \wnorm{\femvec{B}{h}{i}}{1}{3} + \wnorm{\femvec{m}{h}{i}}{1}{3}\right) \leq C.
\end{equation*}
Here, the constant $C$ may depend only on the exact solution and the domain $\Omega$.
\end{lemma}
\begin{proof}
Choosing $h_0 < A^{-1/\beta}$, if $i \geq 1$, it follows successively from the triangle inequality, \eqref{eqn:inverse-estimates}, \eqref{eqn:stokes-proj-estimates-stability}, and \eqref{eqn:inductive-assumption} that
\begin{align*}
    	\lnorm{\femvec{v}{h}{i}}{\infty} 
	+
	\wnorm{\femvec{v}{h}{i}}{1}{3} 
	&\leq 
	\big(
	\lnorm{\femvec{e}{\vec{v}}{i}}{\infty} 
	+
	\wnorm{\femvec{e}{\vec{v}}{i}}{1}{3} 
	\big)
	+
	\big(
	\lnorm{\femvec{\hat{v}}{h}{i}}{\infty} 
	+
	\wnorm{\femvec{\hat{v}}{h}{i}}{1}{3}
	\big)
	\\
    	&\leq 
	C h^{-1/2} 
	\hnorm{\femvec{e}{\vec{v}}{i}}{1} 
	+
	C \big(\hnorm{\vec{v}^i}{2} 
	+
	\hnorms{\tilde{p}^i}{1}\big) \\
    	&\leq 
	C h^{-1/2} A^{1/2}\big(\tau^{1/2} + h^{3/2(1-\beta)}\big) 
	+
	C \big(\hnorm{\vec{v}^i}{2} 
	+
	\hnorms{\tilde{p}^i}{1}\big) \\
    	&\leq C A^{1/2} h^{\beta/2} + C,
\end{align*}
where the constant $C$ depends on $\|\vec{v}\|_{C([0,T]; \HB^2(\Omega))}$ and
$\|\tilde{p}\|_{C([0,T]; H^1(\Omega))}$, which are well-defined due to the
regularity assumption~\eqref{eqn:exact-solution-regularity}, noting
definition~\eqref{eqn:p-tilde}. Of course, if $i = 0$ we have $\femvec{e}{\vec{m}}{i} = \vec{0}$ and the result follows trivially. The other terms can be proved similarly. 
\end{proof}

Using the above estimates, valid at time steps $1, \ldots, n-1$, we estimate the velocity error $\vec{e}_{\vec{v}}^i$.

\begin{lemma} \label{lem:inductive-step-velocity}
Let $\varepsilon > 0$, $n \geq 2$, $1 \leq i \leq n$, and assume that \eqref{eqn:tau-growth} and \eqref{eqn:inductive-assumption} hold. The velocity error term $\femvec{e}{\vec{v}}{i}$ satisfies the estimate
\begin{align*}
    	\lnorm{\femvec{e}{\vec{v}}{i}}{2}^2 
	-
	\lnorm{\femvec{e}{\vec{v}}{i-1}}{2}^2 
	+
	2 \tau \hnorm{\femvec{e}{\vec{v}}{i}}{1}^2 
	&\leq 
	c(\varepsilon) \tau \left(\tau^2 
	+ h^{2\ell} +h^{2k} + h^{2r} + \tau^{-2} h^{2(\ell+1)}
	 \right)
	\\
    	&\qquad 
	+
	c(\varepsilon) \tau 
	\lnorm{\femvec{e}{\vec{v}}{i}}{2}^2 
	+
	c(\varepsilon) \tau 
	\left( \lnorm{\femvec{e}{\vec{v}}{i-1}}{2}^2 + \lnorm{\femvec{e}{\vec{B}}{i-1}}{2}^2\right) 
	\\
    	&\qquad 
	+
	C (\varepsilon + h^{1/2}) \tau 
	\left(\hnorm{\femvec{e}{\vec{v}}{i}}{1}^2 + \hnorm{\femvec{e}{\vec{B}}{i}}{1}^2\right) 
	\\
    	&\qquad
	+
	C \tau \hnorm{\femvec{e}{\vec{m}}{i}}{1}^2 
	+
	C \tau \hnorm{\femvec{e}{\vec{m}}{i-1}}{1}^2 
	\\
    	&\qquad 
	+
	C (\varepsilon + h^{1/2}) \tau 
	\lnorm{\Delta_h \femvec{e}{\vec{m}}{i}}{2}^2,
\end{align*}
where $c(\varepsilon)$ and $C$ may depend only on the exact solution and the domain $\Omega$.
\end{lemma}
\begin{proof}
Set $\vec{\varphi}_h = \femvec{e}{\vec{v}}{i}$ in \eqref{eqn:error-velocity}, $q_h = e_p^i$ in \eqref{eqn:error-velocity-divergence}. Multiplying the resulting equation by $\tau$, expanding $\tau D_{\tau} \vec{\delta}_{\vec{v}}^i = \vec{E}_{\vec{v}}^i + \vec{e}_{\vec{v}}^i - \vec{\delta}_{\vec{v}}^{i-1}$, see \eqref{equ:error split}, and adding $\tau \lnorm{\femvec{e}{\vec{v}}{i}}{2}^2$ to both sides, we obtain noting that $\liprod{\div \vec{e}_{\vec{v}}^i}{e_p^i} = 0$ by \eqref{eqn:error-velocity-divergence},
\begin{align*}
    \lnorm{\femvec{e}{\vec{v}}{i}}{2}^2 + \tau  \hnorm{\femvec{e}{\vec{v}}{i}}{1}^2 &= - \liprod{\femvec{E}{\vec{v}}{i}}{\femvec{e}{\vec{v}}{i}} + \liprod{\femvec{\delta}{\vec{v}}{i-1}}{\femvec{e}{\vec{v}}{i}} - \frac{\tau}{2} \liprod{(\vec{\delta}^{i-1}_{\vec{v}} \cdot \nabla) \vec{v}^{i}}{\femvec{e}{\vec{v}}{i}} \\
    &\qquad  - \frac{\tau}{2} \liprod{(\vec{v}^{i-1}_{h} \cdot \nabla) \vec{\delta}_{\vec{v}}^i}{\femvec{e}{\vec{v}}{i}}+ \frac{\tau}{2} \liprod{(\vec{\delta}^{i-1}_{\vec{v}} \cdot \nabla) \femvec{e}{\vec{v}}{i}}{\vec{v}^{i}} \\
    &\qquad  + \frac{\tau}{2} \liprod{(\vec{v}^{i-1}_{h} \cdot \nabla)\femvec{e}{\vec{v}}{i}}{\vec{\delta}^{i}_{\vec{v}}} + \tau \liprod{\curl \vec{\delta}^{i}_{\vec{B}} \times \vec{B}^{i-1}_h}{\femvec{e}{\vec{v}}{i}} \\
    &\qquad  + \tau \liprod{\curl \vec{B}^{i} \times \vec{\delta}^{i-1}_{\vec{B}}}{\femvec{e}{\vec{v}}{i}} - \tau \liprod{(\vec{\delta}^{i-1}_{\vec{B}} \cdot \nabla) \vec{m}^{i}}{\femvec{e}{\vec{v}}{i}} \\
    &\qquad  - \tau \liprod{(\vec{B}^{i-1}_{h} \cdot \nabla)\vec{\delta}^{i}_{\vec{m}}}{\femvec{e}{\vec{v}}{i}} - \tau \liprod{\nabla \femvec{\delta}{\vec{m}}{i-1} \odot \Delta \vec{m}^i}{\femvec{e}{\vec{v}}{i}} \\
    &\qquad  -  \tau \liprod{\nabla \vec{m}^{i-1} \odot \Delta_h \femvec{\delta}{\vec{m}}{i}}{\femvec{e}{\vec{v}}{i}} + \tau \liprod{\nabla \femvec{\delta}{\vec{m}}{i-1} \odot \Delta_h \femvec{\delta}{\vec{m}}{i}}{\femvec{e}{\vec{v}}{i}} \\
    &\qquad + \tau R_1(\femvec{e}{\vec{v}}{i}) +  \tau \lnorm{\femvec{e}{\vec{v}}{i}}{2}^2 \\
    &= I_1 + \cdots + I_{14} +  \tau \lnorm{\femvec{e}{\vec{v}}{i}}{2}^2.
\end{align*}
In estimating these integrals, we make regular use of Hölder's inequality, the Sobolev embedding theorem, Young's inequality, the regularity of our exact solutions $\vec{v}, \vec{B}$ and $\vec{m}$ in \eqref{eqn:exact-solution-regularity} and Lemma \ref{lem:induction-assumptions-solution-estimates}. We begin by estimating $I_1$ and $I_2$ by using \eqref{eqn:stokes-proj-estimates-approximation} to obtain
\begin{align*}
    |I_1| &:= \left|\liprod{\femvec{E}{\vec{v}}{i}}{\femvec{e}{\vec{v}}{i}}\right| \\
    &\leq \lnorm{\femvec{E}{\vec{v}}{i}}{2} \lnorm{\femvec{e}{\vec{v}}{i}}{2} \\
    &\leq C h^{\ell+1} \lnorm{\femvec{e}{\vec{v}}{i}}{2} \\
    &\leq C \tau^{-1} h^{2(\ell+1)} + C \tau \lnorm{\femvec{e}{\vec{v}}{i}}{2}^2,
\end{align*}
and
\begin{align*}
    |I_2| &:= \left|\liprod{\femvec{\delta}{\vec{v}}{i-1}}{\femvec{e}{\vec{v}}{i}}\right| \\
    &\leq \lnorm{\femvec{E}{\vec{v}}{i-1}}{2} \lnorm{\femvec{e}{\vec{v}}{i}}{2} + \lnorm{\femvec{e}{\vec{v}}{i-1}}{2} \lnorm{\femvec{e}{\vec{v}}{i}}{2} \\
    &\leq C h^{\ell+1} \lnorm{\femvec{e}{\vec{v}}{i}}{2} + \frac{1}{2} \lnorm{\femvec{e}{\vec{v}}{i-1}}{2}^2 + \frac{1}{2} \lnorm{\femvec{e}{\vec{v}}{i}}{2}^2 \\
    &\leq C \tau^{-1} h^{2(\ell+1)} + \frac{1}{2} \lnorm{\femvec{e}{\vec{v}}{i-1}}{2}^2 + \frac{1}{2} \lnorm{\femvec{e}{\vec{v}}{i}}{2}^2 + C \tau \lnorm{\femvec{e}{\vec{v}}{i}}{2}^2
\end{align*}
We estimate $I_3$ by using \eqref{eqn:stokes-proj-estimates-approximation} to obtain
\begin{align*}
    |I_3| &:= \left|\frac{\tau}{2} \liprod{(\vec{\delta}^{i-1}_{\vec{v}} \cdot \nabla) \vec{v}^{i}}{\femvec{e}{\vec{v}}{i}}\right|\\
    &\leq \tau \lnorm{\nabla \vec{v}^i}{\infty} \lnorm{\femvec{\delta}{\vec{v}}{i-1}}{2} \lnorm{\femvec{e}{\vec{v}}{i}}{2} \\\
    &\leq C \tau \lnorm{\femvec{\delta}{\vec{v}}{i-1}}{2} \lnorm{\femvec{e}{\vec{v}}{i}}{2} \\
    &\leq C \tau \left(h^{\ell + 1} + \lnorm{\femvec{e}{\vec{v}}{i-1}}{2}\right) \lnorm{\femvec{e}{\vec{v}}{i}}{2} \\
    &\leq C \tau h^{2(\ell+1)} + C \tau \lnorm{\femvec{e}{\vec{v}}{i}}{2}^2 + C \tau \lnorm{\femvec{e}{\vec{v}}{i-1}}{2}^2.
\end{align*}
Next, with the help of Lemma \ref{lem:induction-assumptions-solution-estimates}, we see that
\begin{align*}
    |I_4| &:= \left|\frac{\tau}{2} \liprod{(\vec{v}^{i-1}_{h} \cdot \nabla) \vec{\delta}_{\vec{v}}^i}{\femvec{e}{\vec{v}}{i}}\right|\\
    &\leq C \tau \lnorm{\femvec{v}{h}{i-1}}{\infty} \lnorm{\nabla \femvec{\delta}{\vec{v}}{i}}{2} \lnorm{\femvec{e}{\vec{v}}{i}}{2} \\ 
    &\leq C \tau \lnorm{\nabla \femvec{\delta}{\vec{v}}{i}}{2} \lnorm{\femvec{e}{\vec{v}}{i}}{2} \\
    &\leq C \tau \left(h^{\ell} + \lnorm{\nabla \femvec{e}{\vec{v}}{i}}{2}\right) \lnorm{\femvec{e}{\vec{v}}{i}}{2} \\
    &\leq C \tau h^{2 \ell} + c(\varepsilon) \tau \lnorm{\femvec{e}{\vec{v}}{i}}{2}^2 + \varepsilon \tau \hnorm{\femvec{e}{\vec{v}}{i}}{1}^2.
\end{align*}
The estimate for $I_5$ is carried out in the same manner as $I_3$ so that
\begin{align*}
    |I_5| &\leq c(\varepsilon) \tau h^{2(\ell+1)} + c(\varepsilon) \tau \lnorm{\femvec{e}{\vec{v}}{i-1}}{2}^2 + \varepsilon \tau \hnorm{\femvec{e}{\vec{v}}{i}}{1}^2,
\end{align*}
while $I_6$ requires Lemma \ref{lem:induction-assumptions-solution-estimates}, in addition, to obtain
\begin{align*}
    |I_6| &\leq c(\varepsilon) \tau h^{2(\ell+1)} + c(\varepsilon) \tau \lnorm{\femvec{e}{\vec{v}}{i}}{2}^2 + \varepsilon \tau \hnorm{\femvec{e}{\vec{v}}{i}}{1}^2.
\end{align*}
We bound $I_7$ as we did $I_4$, but here we use \eqref{eqn:maxwell-proj-estimates-approximation} instead, thus giving
\begin{align*}
    |I_7| &\leq C \tau h^{2k} + c(\varepsilon) \tau \lnorm{\femvec{e}{\vec{v}}{i}}{2}^2 + \varepsilon \tau \hnorm{\femvec{e}{\vec{B}}{i}}{1}^2.
\end{align*}
The integrals $I_8$ and $I_9$ are bounded in a similar manner to $I_3$ so that
\begin{align*}
    |I_8| &\leq C \tau h^{2(k+1)} + C \tau \lnorm{\femvec{e}{\vec{v}}{i}}{2}^2 + C \tau \lnorm{\femvec{e}{\vec{B}}{i-1}}{2}^2, \\
    |I_9| &\leq C \tau h^{2(k+1)} + C \tau \lnorm{\femvec{e}{\vec{v}}{i}}{2}^2 + C \tau \lnorm{\femvec{e}{\vec{B}}{i-1}}{2}^2.
\end{align*}
Using \eqref{eqn:ritz-proj-estimates-approximation}, we estimate $I_{10}$, as we did $I_4$, to obtain
\begin{align*}
    |I_{10}| &\leq C \tau h^{2r} + C \tau \lnorm{\femvec{e}{\vec{v}}{i}}{2}^2 + C \tau \hnorm{\femvec{e}{\vec{m}}{i}}{1}^2,
\end{align*}
and similarly
\begin{align*}
    |I_{11}| &\leq C \tau h^{2r} + C \tau \lnorm{\femvec{e}{\vec{v}}{i}}{2}^2 + C \tau \hnorm{\femvec{e}{\vec{m}}{i-1}}{1}^2.
\end{align*}
For $I_{12}$, using \eqref{eqn:discrete-laplacian-error-definition}, we have
\begin{align*}
    |I_{12}| &:= \left|\tau \liprod{\nabla \vec{m}^{i-1} \odot \Delta_h \femvec{\delta}{\vec{m}}{i}}{\femvec{e}{\vec{v}}{i}}\right|\\
    &\leq \tau \lnorm{\nabla \vec{m}^{i-1}}{\infty} \lnorm{\Delta_h \femvec{\delta}{\vec{m}}{i}}{2} \lnorm{\femvec{e}{\vec{v}}{i}}{2} \\
    &\leq C \tau \left(\lnorm{\Delta \vec{m}^i - \Lproj \Delta \vec{m}^i}{2} + \lnorm{\Delta_h \femvec{e}{\vec{m}}{i}}{2} \right)\lnorm{\femvec{e}{\vec{v}}{i}}{2}  \\
    &\leq C \tau h^{r+1} \lnorm{\femvec{e}{\vec{v}}{i}}{2} + C \tau \lnorm{\Delta_h \femvec{e}{\vec{m}}{i}}{2} \lnorm{\femvec{e}{\vec{v}}{i}}{2}  \\
    &\leq C \tau h^{2(r+1)} + c(\varepsilon) \tau \lnorm{\femvec{e}{\vec{v}}{i}}{2}^2 + \varepsilon \tau \lnorm{\Delta_h \femvec{e}{\vec{m}}{i}}{2}^2.
\end{align*}
Estimating $I_{13}$ is similar but requires extra care. We use the Sobolev embedding $\HB^1(\Omega) \hookrightarrow \LB^6(\Omega)$ and \eqref{eqn:induction-estimate-linf-w13} to obtain
\begin{align*}
    |I_{13}| &:= \left|\tau \liprod{\nabla \femvec{\delta}{\vec{m}}{i-1} \odot \Delta_h \femvec{\delta}{\vec{m}}{i}}{\femvec{e}{\vec{v}}{i}}\right|\\
    &\leq \tau \lnorm{\femvec{e}{\vec{v}}{i}}{6} \lnorm{\nabla \femvec{\delta}{\vec{m}}{i-1}}{3} \lnorm{\Delta_h \femvec{\delta}{\vec{m}}{i}}{2} \\
    &\leq C h^{1/2} \tau \lnorm{\Delta_h \femvec{\delta}{\vec{m}}{i}}{2} \hnorm{\femvec{e}{\vec{v}}{i}}{1} \\
    &\leq C h^{1/2} \tau h^{2(r+1)} \hnorm{\femvec{e}{\vec{v}}{i}}{1} + C h^{1/2} \tau \lnorm{\Delta_h \femvec{e}{\vec{m}}{i}}{2} \hnorm{\femvec{e}{\vec{v}}{i}}{1} \\
    &\leq C \tau h^{2(r+1)} + C h^{1/2} \tau \hnorm{\femvec{e}{\vec{v}}{i}}{1}^2 + C h^{1/2} \tau \lnorm{\Delta_h \femvec{e}{\vec{m}}{i}}{2}^2.
\end{align*}
Lastly, by Lemma \ref{lem:truncation-estimates}
\begin{align*}
    |I_{14}| &\leq C \tau^3 + C \tau \lnorm{\femvec{e}{\vec{v}}{i}}{2}^2.
\end{align*}
We complete the proof by adding up these estimates, subtracting
\begin{equation*}
    \frac{1}{2} \lnorm{\femvec{e}{\vec{v}}{i-1}}{2}^2 + \frac{1}{2} \lnorm{\femvec{e}{\vec{v}}{i}}{2}^2
\end{equation*}
from both sides, and grouping the terms appropriately.
\end{proof}

Next up, we estimate the magnetic field $\femvec{e}{\vec{B}}{i}$.

\begin{lemma} \label{lem:inductive-step-magnetic-field}
Let $\varepsilon > 0$, $n \geq 2$, $1 \leq i \leq n$, and assume that \eqref{eqn:tau-growth} and \eqref{eqn:inductive-assumption} hold. The magnetic field error term $\femvec{e}{\vec{B}}{i}$ satisfies the estimate
\begin{align*}
     \lnorm{\femvec{e}{\vec{B}}{i}}{2}^2 - \lnorm{\femvec{e}{\vec{B}}{i-1}}{2}^2
     + C_\Omega \tau \hnorm{\femvec{e}{\vec{B}}{i}}{1}^2 
     &\leq c(\varepsilon) \tau \left(\tau^2 + h^{2(\ell+1)} + h^{2(k+1)} + \tau^{-2} h^{2(k+1)}\right)   \\
     &\qquad + c(\varepsilon) \tau \left(\lnorm{\femvec{e}{\vec{v}}{i}}{2}^2 + \lnorm{\femvec{e}{\vec{B}}{i}}{2}^2\right) \\
     &\qquad + c(\varepsilon) \tau \lnorm{\femvec{e}{\vec{B}}{i-1}}{2}^2 + C
     \varepsilon \tau \hnorm{\femvec{e}{\vec{B}}{i}}{1}^2,
\end{align*}
where $C_\Omega$ is given in Lemma~\ref{lem:curl-estimate}, and $c(\varepsilon)$
and $C$ may depend only on the exact solution and the domain $\Omega$.
\end{lemma}
\begin{proof}
Setting $\vec{\omega}_h = \femvec{e}{\vec{B}}{i}$ in \eqref{eqn:error-magnetic-field}, multiplying the result by $\tau$, expanding $\tau \discreteDtau \femvec{\delta}{\vec{B}}{i}$ as in Lemma \ref{lem:inductive-step-velocity}, we obtain
\begin{align*}
    \lnorm{\femvec{e}{\vec{B}}{i}}{2}^2 + \tau \lnorm{\curl \femvec{e}{\vec{B}}{i}}{2}^2 + \tau \lnorms{\div \femvec{e}{\vec{B}}{i}}{2}^2  &= - \liprod{\femvec{E}{\vec{B}}{i}}{\femvec{e}{\vec{B}}{i}} + \liprod{\femvec{\delta}{\vec{B}}{i-1}}{\femvec{e}{\vec{B}}{i}} \\
    &\qquad + \tau \liprod{\femvec{\delta}{\vec{v}}{i} \times \femvec{B}{h}{i-1}}{\curl \femvec{e}{\vec{B}}{i}} \\
    &\qquad  + \tau \liprod{\vec{v}^i \times \femvec{\delta}{\vec{B}}{i-1}}{\curl \femvec{e}{\vec{B}}{i}} + \tau R_2(\femvec{e}{\vec{B}}{i}) \\
    &= I_1 + \cdots + I_5.
\end{align*}
In estimating these integrals, we make regular use of Hölder's inequality, the Sobolev embedding theorem, Young's inequality, the regularity of our exact solutions $\vec{v}, \vec{B}$ and $\vec{m}$ in \eqref{eqn:exact-solution-regularity} and Lemma \ref{lem:induction-assumptions-solution-estimates}. We bound $I_1$ and $I_2$ in the same manner as $I_1$ and $I_2$ in Lemma \ref{lem:inductive-step-velocity}, using \eqref{eqn:maxwell-proj-estimates-approximation}, to obtain
\begin{align*}
    |I_1| &\leq C \tau^{-1} h^{2(k+1)} + C \tau \lnorm{\femvec{e}{\vec{B}}{i}}{2}^2, \\
    |I_2| &\leq C \tau^{-1} h^{2(k+1)} + \frac{1}{2} \lnorm{\femvec{e}{\vec{B}}{i-1}}{2}^2 + \frac{1}{2} \lnorm{\femvec{e}{\vec{B}}{i}}{2}^2 + C \tau \lnorm{\femvec{e}{\vec{B}}{i}}{2}^2.
\end{align*}
Using Lemma \ref{lem:induction-assumptions-solution-estimates} and \eqref{eqn:stokes-proj-estimates-approximation}, we have for $I_3$
\begin{align*}
    |I_3| &:= \left|\tau \liprod{\femvec{\delta}{\vec{v}}{i} \times \femvec{B}{h}{i-1}}{\curl \femvec{e}{\vec{B}}{i}}\right|\\
    &\leq \tau \lnorm{\femvec{B}{h}{i-1}}{\infty} \lnorm{\femvec{\delta}{\vec{v}}{i}}{2} \lnorm{\curl \femvec{e}{\vec{B}}{i}}{2} \\
    &\leq C \tau \lnorm{\femvec{\delta}{\vec{v}}{i}}{2} \hnorm{\femvec{e}{\vec{B}}{i}}{1} \\
    &\leq C \tau \left(h^{\ell+1} + \lnorm{\femvec{e}{\vec{v}}{i}}{2}\right) \hnorm{\femvec{e}{\vec{B}}{i}}{1} \\
    &\leq c(\varepsilon) \tau h^{2(\ell+1)} + c(\varepsilon) \tau \lnorm{\femvec{e}{\vec{v}}{i}}{2}^2 + \varepsilon \tau \hnorm{\femvec{e}{\vec{B}}{i}}{1}^2.
\end{align*}
For $I_4$, we use \eqref{eqn:maxwell-proj-estimates-approximation} to obtain
\begin{align*}
    |I_4| &\leq c(\varepsilon) \tau h^{2(k+1)} + c(\varepsilon) \tau \lnorm{\femvec{e}{\vec{B}}{i-1}}{2}^2 + \varepsilon \tau \hnorm{\femvec{e}{\vec{B}}{i}}{1}^2.
\end{align*}
Lastly, by Lemma \ref{lem:truncation-estimates} we obtain
\begin{align*}
    |I_5| &\leq C \tau^3 + C \tau \lnorm{\femvec{e}{\vec{B}}{i}}{2}^2.
\end{align*}
The result is proved by summing up the estimates, using Lemma \ref{lem:curl-estimate}, subtracting
\begin{equation*}
    \frac{1}{2} \lnorm{\femvec{e}{\vec{B}}{i-1}}{2}^2 + \frac{1}{2} \lnorm{\femvec{e}{\vec{B}}{i}}{2}^2
\end{equation*}
from both sides and grouping the terms appropriately.
\end{proof}

Finally, we estimate the magnetisation error $\femvec{e}{\vec{m}}{i}$.

\begin{lemma} \label{lem:inductive-step-magnetisation}
Let $\varepsilon > 0$, $n \geq 2$, $1 \leq i \leq n$, and assume that \eqref{eqn:tau-growth} and \eqref{eqn:inductive-assumption} hold. The magnetisation error term $\femvec{e}{\vec{m}}{i}$ satisfies the estimate
\begin{align*}
    \lnorm{\femvec{e}{\vec{m}}{i}}{2}^2 - \lnorm{\femvec{e}{\vec{m}}{i-1}}{2}^2
    + 2 \tau \hnorm{\femvec{e}{\vec{m}}{i}}{1}^2 &\leq c(\varepsilon) \tau
    \left(\tau^2 + h^{2(\ell+1)} + h^{2(k+1)} + h^{2r} + \tau^{-2} h^{2(r+1)}\right)\\
    &\qquad + C \tau \left(\lnorm{\femvec{e}{\vec{v}}{i}}{2}^2 + \lnorm{\femvec{e}{\vec{B}}{i}}{2}^2\right) \\
    &\qquad + c(\varepsilon) \tau \hnorm{\femvec{e}{\vec{m}}{i}}{1}^2 + C \tau \hnorm{\femvec{e}{\vec{m}}{i-1}}{1}^2 \\
    &\qquad + \varepsilon \tau \lnorm{\Delta_h \femvec{e}{\vec{m}}{i}}{2}^2,
\end{align*}
where $c(\varepsilon)$ and $C$ may depend only on the exact solutions and the domain $\Omega$.
\end{lemma}
\begin{proof}
Setting $\vec{\xi}_h = \femvec{e}{\vec{m}}{i}$ in \eqref{eqn:error-magnetisation}, multiplying the result by $\tau$, expanding $\tau \discreteDtau \femvec{\delta}{\vec{m}}{i}$ as in Lemma \ref{lem:inductive-step-velocity}, and adding $ \tau \lnorm{\femvec{e}{\vec{m}}{i}}{2}^2$ to both sides, we obtain
\begin{align*}
    &\lnorm{\femvec{e}{\vec{m}}{i}}{2}^2 + \tau  \hnorm{\femvec{e}{\vec{m}}{i}}{1}^2 \\
    &\qquad = - \liprod{\femvec{E}{\vec{m}}{i}}{\femvec{e}{\vec{m}}{i}}+ \liprod{\femvec{\delta}{\vec{m}}{i-1}}{\femvec{e}{\vec{m}}{i}} \\
    &\qquad \qquad -\tau \liprod{(\femvec{\delta}{\vec{v}}{i} \cdot \nabla)\femvec{m}{h}{i-1}}{\femvec{e}{\vec{m}}{i}} - \tau \liprod{(\vec{v}^i \cdot \nabla) \femvec{\delta}{\vec{m}}{i-1}}{\femvec{e}{\vec{m}}{i}} \\
    &\qquad \qquad +  \tau \liprod{\femvec{\delta}{\vec{m}}{i-1} \times \Delta \vec{m}^i}{\femvec{e}{\vec{m}}{i}} +  \tau \liprod{\femvec{m}{h}{i-1} \times \Delta_h \femvec{\delta}{\vec{m}}{i}}{\femvec{e}{\vec{m}}{i}} \\
    &\qquad \qquad + \tau  \liprod{(\nabla \vec{m}^i \cdot \nabla \vec{m}^{i-1}) \femvec{\delta}{\vec{m}}{i-1}}{\femvec{e}{\vec{m}}{i}} + \tau  \liprod{(\nabla \femvec{\delta}{\vec{m}}{i} \cdot \nabla \vec{m}^{i-1})\femvec{m}{h}{i-1}}{\femvec{e}{\vec{m}}{i}}\\
    &\qquad \qquad  + \tau  \liprod{(\nabla \femvec{\delta}{\vec{m}}{i-1} \cdot \nabla \vec{m}^i) \femvec{m}{h}{i-1}}{\femvec{e}{\vec{m}}{i}} - \tau  \liprod{(\nabla \femvec{\delta}{\vec{m}}{i} \cdot \nabla \femvec{\delta}{\vec{m}}{i-1}) \femvec{m}{h}{i-1}}{\femvec{e}{\vec{m}}{i}}\\
    &\qquad \qquad  + \tau \liprod{\femvec{\delta}{\vec{m}}{i-1} \times \vec{B}^i}{\femvec{e}{\vec{m}}{i}} + \tau \liprod{\femvec{m}{h}{i-1} \times \femvec{\delta}{\vec{B}}{i}}{\femvec{e}{\vec{m}}{i}}\\
    &\qquad \qquad  + \tau \liprod{\femvec{\delta}{\vec{m}}{i-1} \times (\vec{m}^{i-1} \times \vec{B}^i)}{\femvec{e}{\vec{m}}{i}} + \tau \liprod{\femvec{m}{h}{i-1} \times (\femvec{\delta}{\vec{m}}{i-1} \times \vec{B}^i)}{\femvec{e}{\vec{m}}{i}}\\
    &\qquad \qquad  + \tau \liprod{\femvec{m}{h}{i-1} \times (\femvec{m}{h}{i-1} \times \femvec{\delta}{\vec{B}}{i})}{\femvec{e}{\vec{m}}{i}}  + \tau R_3(\femvec{e}{\vec{m}}{i}) +  \tau \lnorm{\femvec{e}{\vec{m}}{i}}{2}^2\\
    &\qquad = I_1 + \cdots + I_{16} +  \tau \lnorm{\femvec{e}{\vec{m}}{i}}{2}^2.
\end{align*}
In estimating these integrals, we make regular use of Hölder's inequality, the Sobolev embedding theorem, Young's inequality, the regularity of our exact solutions $\vec{v}, \vec{B}$ and $\vec{m}$ in \eqref{eqn:exact-solution-regularity} and Lemma \ref{lem:induction-assumptions-solution-estimates}. Similarly to $I_1$ and $I_2$ in Lemma \ref{lem:inductive-step-velocity} but using only Holder's inequality, and not the $\HB^{-1}(\Omega)$ estimate, we obtain
\begin{align*}
    |I_1| &\leq C \tau^{-1} h^{2(r+1)} + C \tau \lnorm{\femvec{e}{\vec{m}}{i}}{2}^2, \\
    |I_2| &\leq C \tau^{-1} h^{2(r+1)} + \frac{1}{2} \lnorm{\femvec{e}{\vec{m}}{i-1}}{2}^2 + \frac{1}{2} \lnorm{\femvec{e}{\vec{m}}{i}}{2}^2 + C \tau \lnorm{\femvec{e}{\vec{m}}{i}}{2}^2.
\end{align*}
Using the Sobolev embedding $\HB^1(\Omega) \hookrightarrow \LB^6(\Omega)$, \eqref{eqn:stokes-proj-estimates-approximation} and Lemma \ref{lem:induction-assumptions-solution-estimates} we obtain
\begin{align*}
    |I_3| &:= \left|\tau \liprod{(\femvec{\delta}{\vec{v}}{i} \cdot \nabla)\femvec{m}{h}{i-1}}{\femvec{e}{\vec{m}}{i}}\right|\\
    &\leq \tau \lnorm{\femvec{e}{\vec{m}}{i}}{6} \lnorm{\nabla \femvec{m}{h}{i-1}}{3} \lnorm{\femvec{\delta}{\vec{v}}{i}}{2} \\
    &\leq C \tau \hnorm{\femvec{e}{\vec{m}}{i}}{1} \lnorm{\femvec{\delta}{\vec{v}}{i}}{2} \\
    &\leq C \tau \left( h^{\ell+1} + \lnorm{\femvec{e}{\vec{v}}{i}}{2}\right) \hnorm{\femvec{e}{\vec{m}}{i}}{1} \\
    &\leq C \tau h^{2(\ell+1)} + C \tau \lnorm{\femvec{e}{\vec{v}}{i}}{2}^2 + C \tau \hnorm{\femvec{e}{\vec{m}}{i}}{1}^2.
\end{align*}
For $I_4$, we bound it similarly to $I_4$ in Lemma \ref{lem:inductive-step-velocity}
\begin{align*}
    |I_4| &\leq C \tau h^{2r} + C \tau \hnorm{\femvec{e}{\vec{m}}{i-1}}{1}^2 + C \tau \lnorm{\femvec{e}{\vec{m}}{i}}{2}^2,
\end{align*}
with $I_5$ also being estimated in the same manner to produce
\begin{align*}
    |I_5| &\leq C \tau h^{2(r+1)} + C \tau \lnorm{\femvec{e}{\vec{m}}{i-1}}{2}^2 + C \tau \lnorm{\femvec{e}{\vec{m}}{i}}{2}^2.
\end{align*}
Noting \eqref{eqn:discrete-laplacian-error-definition} and Lemma \ref{lem:L2-proj-estimates}, we see that
\begin{align*}
    |I_6| &:= \left|\tau \liprod{\femvec{m}{h}{i-1} \times \Delta_h \femvec{\delta}{\vec{m}}{i}}{\femvec{e}{\vec{m}}{i}}\right| \\
    &\leq \tau \lnorm{\femvec{m}{h}{i-1}}{\infty} \lnorm{\Delta_h \femvec{\delta}{\vec{m}}{i}}{2} \lnorm{\femvec{e}{\vec{m}}{i}}{2} \\
    &\leq C \tau \lnorm{\Delta \vec{m}^i - \Lproj \Delta \vec{m}^i}{2} \lnorm{\femvec{e}{\vec{m}}{i}}{2} + C \tau \lnorm{\Delta_h \femvec{e}{\vec{m}}{i}}{2} \lnorm{\femvec{e}{\vec{m}}{i}}{2} \\
    &\leq C \tau h^{r+1} \lnorm{\femvec{e}{\vec{m}}{i}}{2} + C \tau \lnorm{\femvec{e}{\vec{m}}{i}}{2} \lnorm{\Delta_h \femvec{e}{\vec{m}}{i}}{2} \\
    &\leq C \tau h^{2(r+1)} + c(\varepsilon) \tau \lnorm{\femvec{e}{\vec{m}}{i}}{2}^2 + \varepsilon \tau \lnorm{\Delta_h \femvec{e}{\vec{m}}{i}}{2}^2.
\end{align*}
The inner products $I_7$--$I_9$ can be estimated in the same manner as $I_4$ so that we obtain
\begin{align*}
    |I_7| &\leq C \tau h^{2(r+1)} + C \tau \lnorm{\femvec{e}{\vec{m}}{i-1}}{2}^2 + C \tau \lnorm{\femvec{e}{\vec{m}}{i}}{2}^2, \\
    |I_8| &\leq C \tau h^{2r} + C \tau \hnorm{\femvec{e}{\vec{m}}{i}}{1}^2, \\
    |I_9| &\leq C \tau h^{2r} + C \tau \lnorm{\femvec{e}{\vec{m}}{i}}{2}^2 + C \tau \hnorm{\femvec{e}{\vec{m}}{i-1}}{1}^2,
\end{align*}
while for $I_{10}$ we use the same analysis as for $I_3$ and $I_4$ to produce
\begin{align*}
    |I_{10}| &\leq C \tau h^{2r} + C \tau \hnorm{\femvec{e}{\vec{m}}{i}}{1}^2.
\end{align*}
With the help of \eqref{eqn:maxwell-proj-estimates-approximation}, the estimates for $I_{11}$--$I_{15}$ follow in the same manner as $I_4$ so that we have
\begin{align*}
    |I_{11}| &\leq C \tau h^{2(r+1)} + C \tau \lnorm{\femvec{e}{\vec{m}}{i-1}}{2}^2 + C \tau \lnorm{\femvec{e}{\vec{m}}{i}}{2}^2, \\
    |I_{12}| &\leq C \tau h^{2(k+1)} + C \tau \lnorm{\femvec{e}{\vec{B}}{i}}{2}^2 + C \tau \lnorm{\femvec{e}{\vec{m}}{i}}{2}^2, \\
    |I_{13}| &\leq C \tau h^{2(r+1)} + C \tau \lnorm{\femvec{e}{\vec{m}}{i-1}}{2}^2 + C \tau \lnorm{\femvec{e}{\vec{m}}{i}}{2}^2, \\
    |I_{14}| &\leq C \tau h^{2(r+1)} + C \tau \lnorm{\femvec{e}{\vec{m}}{i-1}}{2}^2 + C \tau \lnorm{\femvec{e}{\vec{m}}{i}}{2}^2, \\
    |I_{15}| &\leq C \tau h^{2(k+1)} + C \tau \lnorm{\femvec{e}{\vec{B}}{i}}{2}^2 + C \tau \lnorm{\femvec{e}{\vec{m}}{i}}{2}^2.
\end{align*}
Lastly, from Lemma \ref{lem:truncation-estimates} we see that
\begin{align*}
    |I_{16}| &\leq C \tau^3 + C \tau \lnorm{\femvec{e}{\vec{m}}{i}}{2}^2.
\end{align*}
The result is proved by summing up the estimates, subtracting
\begin{equation*}
    \frac{1}{2} \lnorm{\femvec{e}{\vec{m}}{i-1}}{2}^2 + \frac{1}{2} \lnorm{\femvec{e}{\vec{m}}{i}}{2}^2
\end{equation*}
from both sides and grouping the terms appropriately.
\end{proof}

We can obtain more estimates on the magnetisation error $\femvec{e}{\vec{m}}{i}$.

\begin{lemma} \label{lem:inductive-step-magnetisation-H1}
Let $\varepsilon > 0$, $n \geq 2$, $1 \leq i \leq n$, and assume that \eqref{eqn:tau-growth} and \eqref{eqn:inductive-assumption} hold. The magnetisation error term $\femvec{e}{\vec{m}}{i}$ satisfies the estimate
\begin{align*}
    \lnorm{\nabla \femvec{e}{\vec{m}}{i}}{2}^2 - \lnorm{\nabla
    \femvec{e}{\vec{m}}{i-1}}{2}^2 + 2 \tau \lnorm{\Delta_h
\femvec{e}{\vec{m}}{i}}{2}^2 &\leq c(\varepsilon) \tau \left(\tau^2 + h^{2\ell} + h^{2r} + \tau^{-2} h^{2(r+1)}\right)\\
    &\qquad + c(\varepsilon) \tau \left(\lnorm{\femvec{e}{\vec{v}}{i}}{2}^2 +  \lnorm{\femvec{e}{\vec{B}}{i}}{2}^2\right) \\
    &\qquad + C h^{1/2} \tau \hnorm{\femvec{e}{\vec{v}}{i}}{1}^2 + c(\varepsilon) \tau \hnorm{\femvec{e}{\vec{m}}{i}}{1}^2 \\
    &\qquad + c(\varepsilon) \tau \hnorm{\femvec{e}{\vec{m}}{i-1}}{1}^2 \\
    &\qquad + C (\varepsilon + h^{1/2}) \tau \lnorm{\Delta_h \femvec{e}{\vec{m}}{i}}{2}^2,
\end{align*}
where the constant $c(\varepsilon)$ and $C$ may depend only on the exact solution and the domain $\Omega$.
\end{lemma}
\begin{proof}
Setting $\vec{\xi}_h = -\Delta_h \femvec{e}{\vec{m}}{i}$ in \eqref{eqn:error-magnetisation}, multiplying the resulting equation by $\tau$, expanding $\tau \discreteDtau \femvec{\delta}{\vec{m}}{i}$ as in Lemma \ref{lem:inductive-step-velocity}, and using \eqref{eqn:discrete-laplacian}, we obtain
\begin{align*}
    &\lnorm{\nabla \femvec{e}{\vec{m}}{i}}{2}^2 + \tau  \lnorm{\Delta_h \femvec{e}{\vec{m}}{i}}{2}^2 \\
    &\qquad = \liprod{\femvec{E}{\vec{m}}{i}}{\Delta_h \femvec{e}{\vec{m}}{i}} - \liprod{\femvec{\delta}{\vec{m}}{i-1}}{\Delta_h \femvec{e}{\vec{m}}{i}} \\
    &\qquad \qquad  + \tau \liprod{(\femvec{\delta}{\vec{v}}{i} \cdot \nabla)\femvec{m}{h}{i-1}}{\Delta_h \femvec{e}{\vec{m}}{i}} + \tau \liprod{(\vec{v}^i \cdot \nabla) \femvec{\delta}{\vec{m}}{i-1}}{\Delta_h \femvec{e}{\vec{m}}{i}} \\
    &\qquad \qquad -  \tau \liprod{\femvec{\delta}{\vec{m}}{i-1} \times \Delta \vec{m}^i}{\Delta_h \femvec{e}{\vec{m}}{i}} -  \tau \liprod{\femvec{m}{h}{i-1} \times \Delta_h \femvec{\delta}{\vec{m}}{i}}{\Delta_h \femvec{e}{\vec{m}}{i}} \\
    &\qquad \qquad - \tau  \liprod{(\nabla \vec{m}^i \cdot \nabla \vec{m}^{i-1}) \femvec{\delta}{\vec{m}}{i-1}}{\Delta_h \femvec{e}{\vec{m}}{i}} - \tau  \liprod{(\nabla \femvec{\delta}{\vec{m}}{i} \cdot \nabla \vec{m}^{i-1})\femvec{m}{h}{i-1}}{\Delta_h \femvec{e}{\vec{m}}{i}}\\
    &\qquad \qquad  - \tau  \liprod{(\nabla \femvec{\delta}{\vec{m}}{i-1} \cdot \nabla \vec{m}^i) \femvec{m}{h}{i-1}}{\Delta_h \femvec{e}{\vec{m}}{i}} + \tau  \liprod{(\nabla \femvec{\delta}{\vec{m}}{i} \cdot \nabla \femvec{\delta}{\vec{m}}{i-1}) \femvec{m}{h}{i-1}}{\Delta_h \femvec{e}{\vec{m}}{i}}\\
    &\qquad \qquad  - \tau \liprod{\femvec{\delta}{\vec{m}}{i-1} \times \vec{B}^i}{\Delta_h \femvec{e}{\vec{m}}{i}} + \tau \liprod{\femvec{m}{h}{i-1} \times \femvec{\delta}{\vec{B}}{i}}{\Delta_h \femvec{e}{\vec{m}}{i}}\\
    &\qquad \qquad  - \tau \liprod{\femvec{\delta}{\vec{m}}{i-1} \times (\vec{m}^{i-1} \times \vec{B}^i)}{\Delta_h \femvec{e}{\vec{m}}{i}} - \tau \liprod{\femvec{m}{h}{i-1} \times (\femvec{\delta}{\vec{m}}{i-1} \times \vec{B}^i)}{\Delta_h \femvec{e}{\vec{m}}{i}}\\
    &\qquad \qquad  - \tau \liprod{\femvec{m}{h}{i-1} \times (\femvec{m}{h}{i-1} \times \femvec{\delta}{\vec{B}}{i})}{\Delta_h \femvec{e}{\vec{m}}{i}} + \tau R_3(-\Delta_h \femvec{e}{\vec{m}}{i}) \\
    &\qquad = I_1 + \cdots + I_{16}.
\end{align*}
In estimating these integrals, we make regular use of Hölder's inequality, the Sobolev embedding theorem, Young's inequality, the regularity of our exact solutions $\vec{v}, \vec{B}$ and $\vec{m}$ in \eqref{eqn:exact-solution-regularity} and Lemma \ref{lem:induction-assumptions-solution-estimates}. As in Lemma \ref{lem:inductive-step-magnetisation} we have
\begin{align*}
    |I_1| &\leq c(\varepsilon) \tau^{-1} h^{2(r+1)} + \varepsilon \tau \lnorm{\Delta_h \femvec{e}{\vec{m}}{i}}{2}^2,
\end{align*}
but for $I_2$ we use \eqref{eqn:discrete-laplacian} to obtain
\begin{align*}
    I_2 &:= -\liprod{\femvec{\delta}{\vec{m}}{i-1}}{\Delta_h \femvec{e}{\vec{m}}{i}} \\
    &= -\liprod{\femvec{E}{\vec{m}}{i-1}}{\Delta_h \femvec{e}{\vec{m}}{i}} + \liprod{\nabla \femvec{e}{\vec{m}}{i-1}}{\nabla \femvec{e}{\vec{m}}{i}},
\end{align*}
and so,
\begin{align*}
    |I_2| &\leq c(\varepsilon) \tau^{-1} h^{2(r+1)} + \varepsilon \tau \lnorm{\Delta_h \femvec{e}{\vec{m}}{i}}{2}^2 + \frac{1}{2} \lnorm{\nabla \femvec{e}{\vec{m}}{i-1}}{2}^2 + \frac{1}{2} \lnorm{\nabla \femvec{e}{\vec{m}}{i}}{2}^2.
\end{align*}
Using the Sobolev embedding $\HB^1(\Omega) \hookrightarrow \LB^6(\Omega)$, \eqref{eqn:stokes-proj-estimates-approximation} and Lemmas \ref{lem:induction-assumptions-error-estimates}--\ref{lem:induction-assumptions-solution-estimates} we obtain
\begin{align*}
    |I_3| &:= \left|\tau \liprod{(\femvec{\delta}{\vec{v}}{i} \cdot \nabla)\femvec{m}{h}{i-1}}{\Delta_h \femvec{e}{\vec{m}}{i}}\right| \\
    &\leq \left|\tau \liprod{(\femvec{\delta}{\vec{v}}{i} \cdot \nabla)\vec{m}^i}{\Delta_h \femvec{e}{\vec{m}}{i}}\right| + \left|\tau \liprod{(\femvec{\delta}{\vec{v}}{i} \cdot \nabla)\femvec{\delta}{\vec{m}}{i-1}}{\Delta_h \femvec{e}{\vec{m}}{i}}\right| \\
    &\leq \tau \lnorm{\nabla \vec{m}^{i-1}}{\infty} \lnorm{\femvec{\delta}{\vec{v}}{i}}{2} \lnorm{\Delta_h \femvec{e}{\vec{m}}{i}}{2} + \tau \lnorm{\femvec{\delta}{\vec{v}}{i}}{6} \lnorm{\nabla \femvec{\delta}{\vec{m}}{i-1}}{3} \lnorm{\Delta_h \femvec{e}{\vec{m}}{i}}{2} \\
    &\leq C \tau \lnorm{\femvec{\delta}{\vec{v}}{i}}{2} \lnorm{\Delta_h \femvec{e}{\vec{m}}{i}}{2} + C h^{1/2} \tau \hnorm{\femvec{\delta}{\vec{v}}{i}}{1} \lnorm{\Delta_h \femvec{e}{\vec{m}}{i}}{2} \\
    &\leq C \tau \left(h^{\ell+1} + \lnorm{\femvec{e}{\vec{v}}{i}}{2}\right) \lnorm{\Delta_h \femvec{e}{\vec{m}}{i}}{2} + C h^{1/2} \tau \left(h^{\ell} + \hnorm{\femvec{e}{\vec{v}}{i}}{1}\right) \lnorm{\Delta_h \femvec{e}{\vec{m}}{i}}{2} \\
    &\leq c(\varepsilon) \tau h^{2(\ell+1)} + C \tau h^{2 \ell} + c(\varepsilon) \tau \lnorm{\femvec{e}{\vec{v}}{i}}{2}^2 + C h^{1/2} \tau \hnorm{\femvec{e}{\vec{v}}{i}}{1}^2 + C (\varepsilon + h^{1/2}) \tau \lnorm{\Delta_h \femvec{e}{\vec{m}}{i}}{2}^2,
\end{align*}
For $I_4$ and $I_5$, we use \eqref{eqn:ritz-proj-estimates-approximation} and bound it similarly to $I_4$ in Lemma \ref{lem:inductive-step-velocity} to obtain
\begin{align*}
    |I_4| &\leq c(\varepsilon) \tau h^{2r} + c(\varepsilon) \tau \hnorm{\femvec{e}{\vec{m}}{i-1}}{1}^2 + \varepsilon \tau \lnorm{\Delta_h \femvec{e}{\vec{m}}{i}}{2}^2, \\
    |I_5| &\leq C \tau h^{2(r+1)} + C \tau \lnorm{\femvec{e}{\vec{m}}{i-1}}{2}^2 + \varepsilon \tau \lnorm{\Delta_h \femvec{e}{\vec{m}}{i}}{2}^2.
\end{align*}
Recalling \eqref{eqn:discrete-laplacian-error-definition}, Lemma \eqref{lem:L2-proj-estimates} and Lemma \ref{lem:induction-assumptions-solution-estimates}, we see that
\begin{align*}
    |I_6| &:= \left|\tau \liprod{\femvec{m}{h}{i-1} \times \Delta_h \femvec{\delta}{\vec{m}}{i}}{\Delta_h \femvec{e}{\vec{m}}{i}}\right|\\
    &\leq \tau \lnorm{\femvec{m}{h}{i-1}}{\infty} \lnorm{\Delta \vec{m}^i - \Lproj \Delta \vec{m}^i}{2} \lnorm{\Delta_h \femvec{e}{\vec{m}}{i}}{2} \\
    &\leq C \tau h^{r+1} \lnorm{\Delta_h \femvec{e}{\vec{m}}{i}}{2} \\
    &\leq c(\varepsilon) \tau h^{2(r+1)} + \varepsilon \tau \lnorm{\Delta_h \femvec{e}{\vec{m}}{i}}{2}^2.
\end{align*}
The estimates for $I_7$--$I_9$ follow in the same manner as $I_4$ so that we obtain
\begin{align*}
    |I_7| &\leq c(\varepsilon) \tau h^{2(r+1)} + c(\varepsilon) \tau \lnorm{\femvec{e}{\vec{m}}{i-1}}{2}^2 + \varepsilon \tau \lnorm{\Delta_h \femvec{e}{\vec{m}}{i}}{2}^2, \\
    |I_8| &\leq c(\varepsilon) \tau h^{2r} + c(\varepsilon) \tau \hnorm{\femvec{e}{\vec{m}}{i}}{1}^2 + \varepsilon \tau \lnorm{\Delta_h \femvec{e}{\vec{m}}{i}}{2}^2, \\
    |I_9| &\leq c(\varepsilon) \tau h^{2r} + c(\varepsilon) \tau \hnorm{\femvec{e}{\vec{m}}{i-1}}{1}^2 + \varepsilon \tau \lnorm{\Delta_h \femvec{e}{\vec{m}}{i}}{2}^2.
\end{align*}
Estimating $I_{10}$ is similar to estimating $I_3$ but requires the estimate (see e.g., \cite{gui2022-article})
\begin{equation} \label{eqn:discrete-laplacian-estimate}
    \lnorm{\nabla \vec{u}_h}{6} \leq C \lnorm{\Delta_h \vec{u}_h}{2}, \qquad \forall \vec{u}_h \in \V_h^m(\Omega)
\end{equation}
so that we have
\begin{align*}
    |I_{10}| &:= \left|\tau  \liprod{(\nabla \femvec{\delta}{\vec{m}}{i} \cdot \nabla \femvec{\delta}{\vec{m}}{i-1}) \femvec{m}{h}{i-1}}{\Delta_h \femvec{e}{\vec{m}}{i}}\right|\\
    &\leq \tau \lnorm{\femvec{m}{h }{i-1}}{\infty} \lnorm{\nabla \femvec{\delta}{\vec{m}}{i}}{6} \lnorm{\nabla \femvec{\delta}{\vec{m}}{i-1}}{3} \lnorm{\Delta_h \femvec{e}{\vec{m}}{i}}{2} \\
    &\leq C h^{1/2} \tau \lnorm{\nabla \femvec{\delta}{\vec{m}}{i}}{6} \lnorm{\Delta_h \femvec{e}{\vec{m}}{i}}{2} \\
    &\leq C h^{1/2} \tau \left(h^r + \lnorm{\nabla \femvec{e}{\vec{m}}{i}}{6}\right) \lnorm{\Delta_h \femvec{e}{\vec{m}}{i}}{2} \\
    &\leq C \tau h^r \lnorm{\Delta_h \femvec{e}{\vec{m}}{i}}{2} + C h^{1/2} \tau \lnorm{\Delta_h \femvec{e}{\vec{m}}{i}}{2}^2, \\
    &\leq C \tau h^{2r} + C(\varepsilon +  h^{1/2}) \tau \lnorm{\Delta_h \femvec{e}{\vec{m}}{i}}{2}^2.
\end{align*}
The inner products $I_{11}$--$I_{15}$ can be estimated using the same procedure as $I_4$, using \eqref{eqn:maxwell-proj-estimates-approximation} in addition, to obtain
\begin{align*}
    |I_{11}| &\leq c(\varepsilon) \tau h^{2(r+1)} +c(\varepsilon) \tau \lnorm{\femvec{e}{\vec{m}}{i-1}}{2}^2 + \varepsilon \tau \lnorm{\Delta_h \femvec{e}{\vec{m}}{i}}{2}^2, \\
    |I_{12}| &\leq c(\varepsilon) \tau h^{2(k+1)} + c(\varepsilon) \tau \lnorm{\femvec{e}{\vec{B}}{i}}{2}^2 + \varepsilon \tau \lnorm{\Delta_h \femvec{e}{\vec{m}}{i}}{2}^2, \\
    |I_{13}| &\leq c(\varepsilon) \tau h^{2(r+1)} + c(\varepsilon) \tau \lnorm{\femvec{e}{\vec{m}}{i-1}}{2}^2 + \varepsilon \tau \lnorm{\Delta_h \femvec{e}{\vec{m}}{i}}{2}^2, \\
    |I_{14}| &\leq c(\varepsilon) \tau h^{2(r+1)} + c(\varepsilon) \tau \lnorm{\femvec{e}{\vec{m}}{i-1}}{2}^2 + \varepsilon \tau \lnorm{\Delta_h \femvec{e}{\vec{m}}{i}}{2}^2, \\
    |I_{15}| &\leq c(\varepsilon) \tau h^{2(k+1)} + c(\varepsilon) \tau \lnorm{\femvec{e}{\vec{B}}{i}}{2}^2 + \varepsilon \tau \lnorm{\Delta_h \femvec{e}{\vec{m}}{i}}{2}^2.
\end{align*}
Lastly, from Lemma \ref{lem:truncation-estimates} we see that
\begin{align*}
    |I_{16}| &\leq c(\varepsilon) \tau^3 + \varepsilon \tau \lnorm{\Delta_h \femvec{e}{\vec{m}}{i}}{2}^2.
\end{align*}
The result is proved by summing up the estimates and subtracting
\begin{equation*}
    \frac{1}{2} \lnorm{\nabla \femvec{e}{\vec{m}}{i-1}}{2}^2 + \frac{1}{2} \lnorm{\nabla \femvec{e}{\vec{m}}{i}}{2}^2
\end{equation*}
from both sides.
\end{proof}

The following estimates are technical and are required for the proof of Theorem \ref{thm:main-result} in Section \ref{sec:theorem-proof}. We establish estimates for the finite element solutions $\femvec{\dot{v}}{h}{n}$, $\femvec{\dot{B}}{h}{n}$, and $\femvec{\dot{m}}{h}{n}$, obtained from the homogeneous linear system \eqref{eqn:homo-discrete-formulation-velocity}--\eqref{eqn:homo-discrete-formulation-magnetisation}. As we did in Section \ref{sec:theorem-proof}, we use the notation $\femvec{\dot{u}}{h}{n}$ to denote the finite element solution of the homogeneous linear system \eqref{eqn:homo-discrete-formulation-velocity}--\eqref{eqn:homo-discrete-formulation-magnetisation}, so that it is distinguished from the finite element solutions of the non-homogeneous linear system \eqref{eqn:discrete-velocity-formulation}--\eqref{eqn:discrete-magnetisation-formulation}.

\begin{lemma} \label{lem:homo-estimates}
Let $\varepsilon > 0$ and $n \geq 1$. Then the solutions
$(\femvec{\dot{v}}{h}{n}, \femvec{\dot{B}}{h}{n}, \femvec{\dot{m}}{h}{n})$ of
the homogeneous system \eqref{eqn:homo-discrete-formulation-velocity}--\eqref{eqn:homo-discrete-formulation-magnetisation} satisfy the estimates
\begin{align}
    &\lnorm{\femvec{\dot{v}}{h}{n}}{2}^2 + \tau \lnorm{\femvec{\dot{v}}{h}{n}}{2}^2 \leq c(\varepsilon) \tau \lnorm{\femvec{\dot{v}}{h}{n}}{2}^2+ C (\varepsilon + h^{1/2}) \tau \left(\hnorm{\femvec{\dot{v}}{h}{n}}{1}^2 + \hnorm{\femvec{\dot{B}}{h}{n}}{1}^2\right) \nonumber \\
    &\qquad \qquad \qquad \qquad \qquad \qquad + C \tau \hnorm{\femvec{\dot{m}}{h}{n}}{1}^2 + C (\varepsilon + h^{1/2}) \tau \lnorm{\Delta_h \femvec{\dot{m}}{h}{n}}{2}^2, \label{eqn:homo-estimate-v} \\[2ex]
    &\lnorm{\femvec{\dot{B}}{h}{n}}{2}^2 + \tau \hnorm{\femvec{\dot{B}}{h}{n}}{1}^2 \leq c(\varepsilon) \tau \left(\lnorm{\femvec{\dot{v}}{h}{n}}{2}^2 + \lnorm{\femvec{\dot{B}}{h }{n}}{2}^2\right) + \varepsilon \tau \hnorm{\femvec{\dot{B}}{h}{n}}{1}^2, \label{eqn:homo-estimate-B} \\[2ex]
    &\lnorm{\femvec{\dot{m}}{h}{n}}{2}^2 + \tau \hnorm{\femvec{\dot{m}}{h}{n}}{1}^2 \leq C \tau \lnorm{\femvec{\dot{B}}{h}{n}}{2}^2 + c(\varepsilon) \tau \hnorm{\femvec{\dot{m}}{h}{n}}{1}^2 + \varepsilon \tau \hnorm{\femvec{\dot{v}}{h}{n}}{1}^2 \nonumber \\
    &\qquad \qquad \qquad \qquad \qquad \qquad \qquad   + \varepsilon \tau \lnorm{\Delta_h \femvec{\dot{m}}{h}{n}}{2}^2, \label{eqn:homo-estimate-m-L2} \\[2ex]
    &\lnorm{\nabla \femvec{\dot{m}}{h}{n}}{2}^2 + \tau \lnorm{\Delta_h \femvec{\dot{m}}{h}{n}}{2}^2 \leq c(\varepsilon) \tau \left(\lnorm{\femvec{\dot{v}}{h}{n}}{2}^2 + \lnorm{\femvec{B}{h }{n}}{2}^2\right) + c(\varepsilon) \tau \hnorm{\femvec{\dot{m}}{h}{n}}{1}^2 \nonumber \\
    &\qquad \qquad \qquad \qquad \qquad \qquad \qquad \qquad + C h^{1/2} \tau \hnorm{\femvec{\dot{v}}{h}{n}}{1}^2 + C (\varepsilon + h^{1/2}) \tau \lnorm{\Delta_h \femvec{\dot{m}}{h}{n}}{2}^2, \label{eqn:homo-estimate-m-H1}
\end{align}
where the $c(\varepsilon)$ and $C$ may depend only on the exact solution and the domain $\Omega$.
\end{lemma}

\begin{proof} \mbox{}

\noindent \underline{Proof of \eqref{eqn:homo-estimate-v}}. Setting $\vec{\varphi}_h = \femvec{\dot{v}}{h}{n}$ in \eqref{eqn:homo-discrete-formulation-velocity} and $q_h = \tilde{p}_h^n$ \eqref{eqn:homo-discrete-formulation-div-velocity}, and adding $ \tau \lnorm{\femvec{\dot{v}}{h}{n}}{2}^2$ to both sides, we obtain
\begin{align*}
    \lnorm{\femvec{\dot{v}}{h}{n}}{2}^2 +  \tau \hnorm{\femvec{\dot{v}}{h}{n}}{1}^2 &= -\frac{\tau}{2} \liprod{(\femvec{v}{h}{n-1} \cdot \nabla) \femvec{\dot{v}}{h}{n}}{\femvec{\dot{v}}{h}{n}} + \frac{\tau}{2} \liprod{(\femvec{v}{h}{n-1} \cdot \nabla)\femvec{\dot{v}}{h}{n}}{\femvec{\dot{v}}{h}{n}} \\
    &\qquad + \tau \liprod{\curl \femvec{\dot{B}}{h}{n} \times \femvec{B}{h}{n-1}}{\femvec{\dot{v}}{h}{n}} - \tau \liprod{(\femvec{B}{h}{n-1} \cdot \nabla)\femvec{\dot{m}}{h}{n}}{\femvec{\dot{v}}{h}{n}} \\
    &\qquad - \tau \liprod{\nabla \femvec{m}{h}{n-1} \odot \Delta_h \femvec{\dot{m}}{h}{n}}{\vec{\varphi}_h} \\
    &= \tau \liprod{\curl \femvec{\dot{B}}{h}{n} \times \femvec{B}{h}{n-1}}{\femvec{\dot{v}}{h}{n}} - \tau \liprod{(\femvec{B}{h}{n-1} \cdot \nabla)\femvec{\dot{m}}{h}{n}}{\femvec{\dot{v}}{h}{n}} \\
    &\qquad - \tau \liprod{\nabla \femvec{m}{h}{n-1} \odot \Delta_h \femvec{\dot{m}}{h}{n}}{\vec{\varphi}_h} \\
    &= I_1 + \cdots + I_3 +  \tau \lnorm{\femvec{\dot{v}}{h}{n}}{2}^2.
\end{align*}
We regularly use Hölder's inequality and Young's inequality throughout our estimates. Estimating $I_1$ using Lemma \ref{lem:induction-assumptions-solution-estimates} we have
\begin{align*}
    |I_1| &:= \left|\tau \liprod{\curl \femvec{\dot{B}}{h}{n} \times \femvec{B}{h}{n-1}}{\femvec{\dot{v}}{h}{n}}\right| \\
    &\leq \tau \lnorm{\femvec{B}{h}{n-1}}{\infty} \lnorm{\femvec{\dot{v}}{h}{n}}{2} \lnorm{\curl \femvec{\dot{B}}{h}{n}}{2} \\
    &\leq C \tau \lnorm{\femvec{\dot{v}}{h}{n}}{2} \hnorm{\femvec{\dot{B}}{h}{n}}{1} \\
    &\leq c(\varepsilon) \tau \lnorm{\femvec{\dot{v}}{h}{n}}{2}^2 + \varepsilon \tau \hnorm{\femvec{\dot{B}}{h}{n}}{1}^2,
\end{align*}
and similarly,
\begin{align*}
    |I_2| &\leq C \tau \lnorm{\femvec{\dot{v}}{h}{n}}{2}^2 + C \tau \hnorm{\femvec{\dot{m}}{h}{n}}{1}^2.
\end{align*}
The inner product $I_3$ requires greater care and we estimate it as such using the Sobolev embedding $\HB^1(\Omega) \hookrightarrow \LB^6(\Omega)$ and \eqref{eqn:induction-estimate-linf-w13}
\begin{align*}
    |I_3| &:= \left|\tau \liprod{\nabla \femvec{m}{h}{n-1} \odot \Delta_h \femvec{\dot{m}}{h}{n}}{\vec{\varphi}_h}\right|\\
    &\leq \tau \left|\liprod{\nabla \vec{m}^{n-1} \odot \Delta_h \femvec{\dot{m}}{h}{n}}{\femvec{\dot{v}}{h}{n}}\right| + \tau \left|\liprod{\nabla \femvec{\delta}{\vec{m}}{n-1} \odot \Delta_h \femvec{\dot{m}}{h}{n}}{\femvec{\dot{v}}{h}{n}}\right| \\
    &\leq \tau \lnorm{\nabla \vec{m}^{n-1}}{\infty} \lnorm{\femvec{\dot{v}}{h}{n}}{2} \lnorm{\Delta_h \femvec{\dot{m}}{h}{n}}{2} + \tau \lnorm{\femvec{\dot{v}}{h}{n}}{6} \lnorm{\nabla \femvec{\delta}{\vec{m}}{n-1}}{3} \lnorm{\Delta_h \femvec{\dot{m}}{h}{n}}{2} \\
    &\leq C \tau \lnorm{\femvec{\dot{v}}{h}{n}}{2} \lnorm{\Delta_h \femvec{\dot{m}}{h}{n}}{2} + C h^{1/2} \tau \hnorm{\femvec{\dot{v}}{h}{n}}{1} \lnorm{\Delta_h \femvec{\dot{m}}{h}{n}}{2} \\
    &\leq c(\varepsilon) \tau \lnorm{\femvec{\dot{v}}{h}{n}}{2}^2 + C h^{1/2} \tau \hnorm{\femvec{\dot{v}}{h}{n}}{1}^2 + C (\varepsilon + h^{1/2}) \tau \lnorm{\Delta_h \femvec{\dot{m}}{h}{n}}{2}^2.
\end{align*}
The result is proved by summing up these estimates. \\

\noindent \underline{Proof of \eqref{eqn:homo-estimate-B}}. Setting $\vec{\omega}_h = \femvec{\dot{B}}{h}{n}$ in \eqref{eqn:homo-discrete-formulation-magnetic-field}, adding $ \tau \lnorm{\femvec{\dot{B}}{h}{n}}{2}^2$ to both sides and using Lemma \ref{lem:curl-estimate} we have
\begin{align*}
    \lnorm{\femvec{\dot{B}}{h}{n}}{2}^2 + \tau \hnorm{\femvec{\dot{B}}{h}{n}}{1}^2 &\leq C \tau \left|\liprod{\femvec{\dot{v}}{h}{n} \times \femvec{B}{h}{n-1}}{\curl \femvec{\dot{B}}{h}{n}}\right| + C \tau \lnorm{\femvec{\dot{B}}{h}{n}}{2}^2\\
    &= I_1 + C \tau \lnorm{\femvec{\dot{B}}{h}{n}}{2}^2.
\end{align*}
Using Hölder's inequality, Young's inequality and Lemma \ref{lem:induction-assumptions-solution-estimates} we have
\begin{align*}
    |I_1| &\leq \tau \lnorm{\femvec{B}{h}{n-1}}{\infty} \lnorm{\femvec{\dot{v}}{h}{n}}{2} \lnorm{\curl \femvec{\dot{B}}{h}{n}}{2} \\
    &\leq C \tau \lnorm{\femvec{\dot{v}}{h}{n}}{2} \hnorm{\femvec{\dot{B}}{h}{n}}{1} \\
    &\leq c(\varepsilon) \tau \lnorm{\femvec{\dot{v}}{h}{n}}{2}^2 + \varepsilon
    \tau \hnorm{\femvec{\dot{B}}{h}{n}}{1}^2,
\end{align*}
which implies \eqref{eqn:homo-estimate-B}.  \\

\noindent \underline{Proof of \eqref{eqn:homo-estimate-m-L2}}. Setting $\vec{\xi}_h = \femvec{\dot{m}}{h}{n}$ in \eqref{eqn:homo-discrete-formulation-magnetisation} and adding $ \tau \lnorm{\femvec{\dot{m}}{h}{n}}{2}^2$ to both sides, we obtain
\begin{align*}
    \lnorm{\femvec{\dot{m}}{h}{n}}{2}^2 +  \tau \hnorm{\femvec{\dot{m}}{h}{n}}{1}^2 &= -\tau \liprod{(\femvec{\dot{v}}{h}{n} \cdot \nabla)\femvec{m}{h}{n-1}}{\femvec{\dot{m}}{h}{n}} - \tau \liprod{\femvec{m}{h}{n-1} \times \Delta_h \femvec{\dot{m}}{h}{n}}{\femvec{\dot{m}}{h}{n}} \\
    &\qquad + \tau \liprod{(\nabla \femvec{\dot{m}}{h}{n} \cdot \nabla \femvec{m}{h}{n-1})\femvec{m}{h}{n-1}}{\femvec{\dot{m}}{h}{n}} + \tau \liprod{\femvec{m}{h}{n-1} \times \femvec{\dot{B}}{h}{n}}{\femvec{\dot{m}}{h}{n}} \\
    &\qquad + \tau \liprod{\femvec{m}{h}{n-1} \times (\femvec{m}{h}{n-1} \times \femvec{\dot{B}}{h}{n})}{\femvec{\dot{m}}{h}{n}} \\
    &= I_1 + \cdots + I_5 +  \tau \lnorm{\femvec{\dot{m}}{h}{n}}{2}^2.
\end{align*}
We regularly use Hölder's inequality and Young's inequality throughout our estimates. To estimate $I_1$, we use the Sobolev embedding $\HB^1(\Omega) \hookrightarrow \LB^6(\Omega)$ and Lemma \ref{lem:induction-assumptions-solution-estimates} to obtain
\begin{align*}
    |I_1| &:= \left|\tau \liprod{(\femvec{\dot{v}}{h}{n} \cdot \nabla)\femvec{m}{h}{n-1}}{\femvec{\dot{m}}{h}{n}}\right| \\
    &\leq \tau \lnorm{\femvec{\dot{v}}{h}{n}}{6} \lnorm{\nabla \femvec{m}{h}{n-1}}{3} \lnorm{\femvec{\dot{m}}{h}{n}}{2} \\
    &\leq C \tau \lnorm{\femvec{\dot{m}}{h}{n}}{2} \hnorm{\femvec{\dot{v}}{h}{n}}{1} \\
    &\leq c(\varepsilon) \tau \lnorm{\femvec{\dot{m}}{h}{n}}{2}^2 + \varepsilon \tau \hnorm{\femvec{\dot{v}}{h}{n}}{1}^2.
\end{align*}
Similarly, we obtain
\begin{align*}
    |I_2| &\leq c(\varepsilon) \tau \lnorm{\femvec{\dot{m}}{h}{n}}{2}^2 + \varepsilon \tau \lnorm{\Delta_h \femvec{\dot{m}}{h}{n}}{2}^2,
\end{align*}
and
\begin{align*}
    |I_3| &:= \left|\tau \liprod{\femvec{m}{h}{n-1} \times (\femvec{m}{h}{n-1} \times \femvec{\dot{B}}{h}{n})}{\femvec{\dot{m}}{h}{n}}\right| \\
    &\leq C \tau \lnorm{\femvec{m}{h}{n-1}}{\infty} \lnorm{\femvec{\dot{m}}{h}{n}}{6} \lnorm{\nabla \femvec{m}{h}{n-1}}{3} \lnorm{\nabla \femvec{\dot{m}}{h}{n}}{2} \\
    &\leq C \tau \hnorm{\femvec{\dot{m}}{h}{n}}{1}^2. 
\end{align*}
Lastly, it is easy to see that
\begin{align*}
    |I_4| &\leq C \tau \lnorm{\femvec{\dot{B}}{h}{n}}{2}^2 + C \tau \lnorm{\femvec{\dot{m}}{h}{n}}{2}^2, \\
    |I_5| &\leq C \tau \lnorm{\femvec{\dot{B}}{h}{n}}{2}^2 + C \tau \lnorm{\femvec{\dot{m}}{h}{n}}{2}^2.
\end{align*}
The result is proved by summing up these estimates. \\

\noindent \underline{Proof of \eqref{eqn:homo-estimate-m-H1}}. Setting $\vec{\xi}_h = \Delta_h \femvec{\dot{m}}{h}{n}$ in \eqref{eqn:homo-discrete-formulation-magnetisation} we obtain
\begin{align*}
    &\lnorm{\nabla \femvec{\dot{m}}{h}{n}}{2}^2 +  \tau \lnorm{\Delta_h \femvec{\dot{m}}{h}{n}}{2}^2 \\
    &\qquad = -\tau \liprod{(\femvec{\dot{v}}{h}{n} \cdot \nabla)\femvec{m}{h}{n-1}}{\Delta_h \femvec{\dot{m}}{h}{n}} - \tau \liprod{\femvec{m}{h}{n-1} \times \Delta_h \femvec{\dot{m}}{h}{n}}{\Delta_h \femvec{\dot{m}}{h}{n}} \\
    &\qquad \qquad + \tau \liprod{(\nabla \femvec{\dot{m}}{h}{n} \cdot \nabla \femvec{m}{h}{n-1})\femvec{m}{h}{n-1}}{\Delta_h \femvec{\dot{m}}{h}{n}} + \tau \liprod{\femvec{m}{h}{n-1} \times \femvec{\dot{B}}{h}{n}}{\Delta_h \femvec{\dot{m}}{h}{n}} \\
    &\qquad \qquad + \tau \liprod{\femvec{m}{h}{n-1} \times (\femvec{m}{h}{n-1} \times \femvec{\dot{B}}{h}{n})}{\Delta_h \femvec{\dot{m}}{h}{n}} \\
    &\qquad = I_1 + \cdots + I_5.
\end{align*}
We regularly use Hölder's inequality and Young's inequality throughout our estimates. To estimate $I_1$, we use the Sobolev embedding $\HB^1(\Omega) \hookrightarrow \LB^6(\Omega)$ and Lemma \ref{lem:induction-assumptions-error-estimates} to obtain
\begin{align*}
    |I_1| &:= \left|\tau \liprod{(\femvec{\dot{v}}{h}{n} \cdot \nabla)\femvec{m}{h}{n-1}}{\Delta_h \femvec{\dot{m}}{h}{n}}\right| \\
    &\leq \left|\tau \liprod{(\femvec{\dot{v}}{h}{n} \cdot \nabla)\femvec{\delta}{\vec{m}}{n-1}}{\Delta_h \femvec{\dot{m}}{h}{n}}\right| + \left|\tau \liprod{(\femvec{\dot{v}}{h}{n} \cdot \nabla)\vec{m}^n}{\Delta_h \femvec{\dot{m}}{h}{n}}\right| \\
    &\leq \tau \lnorm{\femvec{\dot{v}}{h}{n}}{6} \lnorm{\nabla \femvec{\delta}{\vec{m}}{n-1}}{3} \lnorm{\Delta_h \femvec{\dot{m}}{h}{n}}{2} + \tau \lnorm{\nabla \vec{m}^{n-1}}{\infty} \lnorm{\femvec{\dot{v}}{h}{n}}{2} \lnorm{\Delta_h \femvec{\dot{m}}{h}{n}}{2} \\
    &\leq C h^{1/2} \tau \hnorm{\femvec{\dot{v}}{h}{n}}{1} \lnorm{\Delta_h \femvec{\dot{m}}{h}{n}}{2} + C \tau \lnorm{\femvec{\dot{v}}{h}{n}}{2} \lnorm{\Delta_h \femvec{\dot{m}}{h}{n}}{2} \\
    &\leq c(\varepsilon) \tau \lnorm{\femvec{\dot{v}}{h}{n}}{2}^2 + C h^{1/2} \tau \hnorm{\femvec{\dot{v}}{h}{n}}{1}^2 + C (\varepsilon + h^{1/2})
     \tau \lnorm{\Delta_h \femvec{\dot{m}}{h}{n}}{2}^2.
\end{align*}
We now estimate $I_3$ since $I_2 = 0$. Using \eqref{eqn:discrete-laplacian-estimate} and the same Sobolev embedding as for $I_1$ we have
\begin{align*}
    |I_3| &:= \left|\tau \liprod{(\nabla \femvec{\dot{m}}{h}{n} \cdot \nabla \femvec{m}{h}{n-1})\femvec{m}{h}{n-1}}{\Delta_h \femvec{\dot{m}}{h}{n}}\right| \\
    &\leq \left|\tau \liprod{(\nabla \femvec{\dot{m}}{h}{n} \cdot \nabla \femvec{\delta}{\vec{m}}{n-1})\femvec{m}{h}{n-1}}{\Delta_h \femvec{\dot{m}}{h}{n}}\right| + \left|\tau \liprod{(\nabla \femvec{\dot{m}}{h}{n} \cdot \nabla \vec{m}^{n-1})\femvec{m}{h}{n-1}}{\Delta_h \femvec{\dot{m}}{h}{n}}\right| \\
    &\leq  \tau \lnorm{\femvec{m}{h}{n-1}}{\infty} \lnorm{\nabla \femvec{\dot{m}}{h}{n}}{6} \lnorm{\nabla \femvec{\delta}{\vec{m}}{n-1}}{3} \lnorm{\Delta_h \femvec{\dot{m}}{h}{n}}{2} \\
    &\qquad +  \tau \lnorm{\femvec{m}{h}{n-1}}{\infty} \lnorm{\nabla \vec{m}^{n-1}}{\infty} \lnorm{\nabla \femvec{\dot{m}}{h}{n}}{2} \lnorm{\Delta_h \femvec{\dot{m}}{h}{n}}{2} \\
    &\leq C h^{1/2} \tau \lnorm{\Delta_h \femvec{\dot{m}}{h}{n}}{2}^2 + C \tau \hnorm{\femvec{\dot{m}}{h}{n}}{1} \lnorm{\Delta_h \femvec{\dot{m}}{h}{n}}{2} \\
    &\leq c(\varepsilon) \tau \hnorm{\femvec{\dot{m}}{h}{n}}{1}^2 + C (\varepsilon + h^{1/2}) \tau \lnorm{\Delta_h \femvec{\dot{m}}{h}{n}}{2}^2.
\end{align*}
The estimate for $I_4$ and $I_5$ follow easily so that we have
\begin{align*}
    |I_4| &\leq c(\varepsilon) \tau \lnorm{\femvec{\dot{B}}{h}{n}}{2}^2 + \varepsilon \tau \lnorm{\Delta_h \femvec{\dot{m}}{h}{n}}{2}^2, \\
    |I_5| &\leq c(\varepsilon) \tau \lnorm{\femvec{\dot{B}}{h}{n}}{2}^2 + C \varepsilon \tau \lnorm{\Delta_h \femvec{\dot{m}}{h}{n}}{2}^2.
\end{align*}
The result is proved by summing up the estimates.
\end{proof}

\section*{Acknowledgements}

The first author is supported by the Australian Government's Research Training Program Scholarship awarded at the University of New South Wales, Sydney. The second author is partially supported by the Australian Research Council under grant number DP220101811.

The first author would like to acknowledge that numerical experiments were performed on the UNSW Compute Cluster Katana (DOI: 10.26190/669X-A286), and would like to thank Michael Ndjinga at Université Paris-Saclay for his insightful comments that helped identify and resolve a bug in the code for our experiments. Additionally, he would also like to thank ministry apprentice Jaison Jacob for the many conversations that have kept him encouraged while pursuing research.

\bibliographystyle{abbrv}
\bibliography{Bibliography2}

\end{document}